\documentclass[a4paper,12pt]{amsart}

\usepackage[english]{babel}
\usepackage[colorlinks,plainpages]{hyperref}

\usepackage{amssymb,amsfonts,amsxtra,amsmath}
\usepackage{mathabx}

\usepackage{dsfont}
\renewcommand{\mathbb}{\mathds}
\usepackage{mathrsfs}

\usepackage{csquotes}
\usepackage{stmaryrd}
\usepackage{paralist}
\usepackage{url,breakurl}

\usepackage[all,cmtip]{xy}

\usepackage{tikz}
\usetikzlibrary{calc,decorations.markings}

\usepackage[backend=biber,style=alphabetic,doi=false,isbn=false,url=false]{biblatex}
\bibliography{confpoly}

\theoremstyle{plain}
\newtheorem{thm}{Theorem}[section]
\newtheorem*{mainthm}{Main Theorem}
\newtheorem{prp}[thm]{Proposition}
\newtheorem{cor}[thm]{Corollary}
\newtheorem{lem}[thm]{Lemma}
\newtheorem{cnj}[thm]{Conjecture}
\theoremstyle{definition}
\newtheorem{dfn}[thm]{Definition}

\theoremstyle{remark}
\newtheorem{rmkx}[thm]{Remark}
\newenvironment{rmk}
{\pushQED{\qed}\rmkx}
{\popQED\endrmkx}
\newtheorem{exax}[thm]{Example}
\newenvironment{exa}
{\pushQED{\qed}\exax}
{\popQED\endexax}
  
\newcommand{\ol}[1]{{\overline{#1}}}

\newcommand{\abs}[1]{{\left\vert#1\right\vert}}

\newcommand{\ideal}[1]{{\left\langle#1\right\rangle}}
\newcommand{\set}[1]{{\left\{#1\right\}}}

\newcommand{\into}{\hookrightarrow}
\newcommand{\onto}{\twoheadrightarrow}

\newcommand{\xmid}{\;\middle|\;}

\newcommand{\red}{\mathrm{red}}

\renewcommand{\AA}{\mathbb{A}}
\newcommand{\KK}{\mathbb{K}}
\newcommand{\FF}{\mathbb{F}}

\newcommand{\PP}{\mathbb{P}}
\newcommand{\QQ}{\mathbb{Q}}

\newcommand{\TT}{\mathbb{T}}
\newcommand{\NN}{\mathbb{N}}
\newcommand{\ZZ}{\mathbb{Z}}

\newcommand{\B}{\mathcal{B}}
\newcommand{\C}{\mathcal{C}}
\newcommand{\T}{\mathcal{T}}
\renewcommand{\H}{\mathcal{H}}
\newcommand{\I}{\mathcal{I}}
\renewcommand{\L}{\mathcal{L}}
\newcommand{\M}{\mathsf{M}}

\renewcommand{\P}{\mathcal{P}}
\newcommand{\U}{\mathsf{U}}
\newcommand{\W}{\mathsf{W}}

\newcommand{\mm}{\mathfrak{m}}
\newcommand{\nn}{\mathfrak{n}}
\newcommand{\pp}{\mathfrak{p}}
\newcommand{\qq}{\mathfrak{q}}

\DeclareMathOperator{\Ass}{Ass}
\DeclareMathOperator{\Aut}{Aut}

\DeclareMathOperator{\ch}{ch}
\DeclareMathOperator{\cl}{cl}
\DeclareMathOperator{\codim}{codim}
\DeclareMathOperator{\coker}{coker}

\DeclareMathOperator{\Fitt}{Fitt}

\DeclareMathOperator{\gr}{gr}
\DeclareMathOperator{\grade}{grade}
\DeclareMathOperator{\Gr}{Gr}
\DeclareMathOperator{\height}{height}
\DeclareMathOperator{\Hom}{Hom}
\DeclareMathOperator{\id}{id}
\DeclareMathOperator{\Max}{Max}
\DeclareMathOperator{\Min}{Min}
\DeclareMathOperator{\rk}{rk}
\DeclareMathOperator{\Rees}{Rees}
\DeclareMathOperator{\Spec}{Spec}

\begin{document}

\title[Configuration hypersurfaces]{Matroid connectivity \\
and singularities of \\
configuration hypersurfaces}

\author[G.~Denham]{Graham Denham}
\address{G.~Denham\\
Department of Mathematics\\
University of Western Ontario\\
London, ON\\
Canada N6A 5B7}
\email{\href{mailto:gdenham@uwo.ca}{gdenham@uwo.ca}}
\thanks{}

\author[M.~Schulze]{Mathias Schulze}
\address{M.~Schulze\\
Department of Mathematics\\
TU Kaiserslautern\\
67663 Kaiserslautern\\
Germany}
\email{\href{mailto:mschulze@mathematik.uni-kl.de}{mschulze@mathematik.uni-kl.de}}
\thanks{}

\author[U.~Walther]{Uli Walther}
\address{U.~Walther\\
Department of Mathematics\\
Purdue University\\
West Lafayette, IN 47907\\
USA}
\email{\href{mailto:walther@math.purdue.edu}{walther@math.purdue.edu}}

\thanks{
GD was supported by NSERC of Canada.
MS was supported by the German Research Foundation (DFG) in the collaborative research center TRR 195 under Project B05 (\#324841351).
UW was supported by the NSF Grant DMS-1401392 and by the Simons Foundation Collaboration Grant for Mathematicians \#580839.}


\subjclass[2010]{Primary 14N20; Secondary 05C31, 14B05, 14M12, 81Q30}


\keywords{Configuration, matroid, singularity, Feynman, Kirchhoff, Symanzik, Cohen--Macaulay, determinantal}

\begin{abstract}
Consider a linear realization of a matroid over a field.
One associates with it a configuration polynomial and a symmetric bilinear form with linear homogeneous coefficients.
The corresponding configuration hypersurface and its non-smooth locus support the respective first and second degeneracy scheme of the bilinear form.

We show that these schemes are reduced and describe the effect of matroid connectivity:
for ($2$-)connected matroids, the configuration hypersurface is integral,
and the second degeneracy scheme is reduced Cohen--Macaulay of
codimension $3$. 
If the matroid is $3$-connected, then also the second degeneracy scheme is integral. 

In the process, we describe the behavior of configuration polynomials, forms and schemes with respect to various matroid constructions.
\end{abstract}

\maketitle

\tableofcontents

\numberwithin{equation}{section}

\section{Introduction}

\subsection{Feynman diagrams}

A basic problem in high-energy physics is to understand the scattering of particles. 
The basic tool for theoretical predictions is the Feynman diagram with underlying Feynman graph $G=(V,E)$. 
The scattering data correspond to Feynman integrals, computed in the positive orthant of the projective space labeled by the internal edges of the Feynman graph. 
The integrand is the square root of a rational function in the edge variables $x_e$, $e\in E$, that depends parametrically on the masses and moments of the involved particles (see \cite{Bro17}).

The convergence of a Feynman integral is determined by the structure of the denominator of this rational function, which always involves a power of the square root of the \emph{Symanzik polynomial} $\sum_{T\in\T_G}\prod_{e\not\in T}x_e$ of $G$ where $\T_G$ denotes the set of spanning trees of $G$. 
The remaining factor of the denominator, appearing for graphs with edge number less than twice the loop number, is a power of the square root of the second Symanzik polynomial obtained by summing over $2$-forests and involves masses and moments. 
Symanzik polynomials can factor, and the singularities and intersections of the individual components determine the behavior of the Feynman integrals.

Until about a decade ago, all explicitly computed integrals were built from multiple Riemann zeta values and polylogarithms; for example, Broadhurst and Kreimer display a large body of such computations in \cite{BK97}. 
In fact, Kontsevich at some point speculated that Symanzik polynomials, or equivalently their cousins the \emph{Kirchhoff polynomials}
\[
\psi_G(x)=\sum_{T\in\T_G}\prod_{e\in T}x_e
\]
be mixed Tate; this would imply the relation to multiple zeta values.
However, Belkale and Brosnan \cite{BB03} proved that the collection of Kirchhoff polynomials is a rather complicated class of singularities: 
their hypersurface complements generate the ring of all geometric motives.
This does not exactly rule out that Feynman integrals are in some way well-behaved, but makes it rather less likely, and explicit counterexamples to Kontsevich's conjecture were subsequently worked out by Doryn~\cite{Dor11} as well as by Brown and Schnetz~\cite{BS12}.
On the other hand, these examples make the study of these singularities, and especially any kind of uniformity results, that much more interesting.

The influential paper \cite{BEK06} of Bloch, Esnault and Kreimer generated a significant amount of work from the point of view of complex geometry: 
we refer to the book \cite{Mar10} of Marcolli for exposition, as well as \cite{Bro17,Dor11,BW10}. 
Varying ideas of Connes and Kreimer on renormalization that view Feynman integrals as specializations of the Tutte polynomial, Aluffi and Marcolli formulate in \cite{AM11a,AM11b} parametric Feynman integrals as periods, leading to motivic studies on cohomology.
On the explicit side, there is a large body of publications in which specific graphs and their polynomials and Feynman integrals are discussed. 
But, as Brown writes in \cite{Bro15}, while a diversity of techniques is used to study Feynman diagrams, \enquote{each new loop order involves mathematical objects which are an order of magnitude more complex than the last, [\ldots] the unavoidable fact is that arbitrary integrals remain out of reach as ever.}

The present article can be seen as the first step towards a search for uniform properties in this zoo of singularities. 
We view it as a stepping stone for further studies of invariants such as log canonical threshold, logarithmic differential forms and embedded resolution of singularities.

\subsection{Configuration polynomials}

The main idea of Belkale and Brosnan is to move the burden of proof into the more general realm of polynomials and constructible sets derived from matroids rather than graphs, and then to reduce to known facts about such polynomials. 
The article \cite{BEK06} casts Kirchhoff and Symanzik polynomials as very special instances of \emph{configuration polynomials}; this idea was further developed by Patterson in \cite{Pat10}. 
We consider this as a more natural setting since notions such as duality and quotients behave well for configuration polynomials as a whole, but these operations do not preserve the subfamily of matroids derived from graphs. 
In particular, we can focus exclusively on Kirchhoff/configuration  polynomials, since the Symanzik polynomial of $G$ appears as the configuration polynomial of the dual configuration induced by the incidence matrix of $G$.

The configuration polynomial does not depend on a matroid itself but on a configuration, that is, on a (linear) realization of a matroid over a field $\KK$.
The same matroid can admit different realizations, which, in turn, give rise to different configuration polynomials (see Example~\ref{167}).
The \emph{matroid (basis) polynomial} is a competing object, which is assigned to any, even non-realizable, matroid.
It has proven useful for combinatorial applications (see \cite{AGV18,Piq19}).
For graphs and, more generally, regular matroids, all configuration polynomials essentially agree with the matroid polynomial.
In general, however, configuration polynomials differ significantly from matroid polynomials, as documented in
Example~\ref{55}.

Configuration polynomials have a geometric feature that matroid polynomials lack:
generalizing Kirchhoff's matrix-tree theorem, the configuration polynomial arises as the determinant of a symmetric bilinear \emph{configuration form} with linear polynomial coefficients.
As a consequence, the corresponding \emph{configuration hypersurface} maps naturally to the generic symmetric determinantal variety.
In the present article, we establish further uniform, geometric properties of configuration polynomials, which we observe do not hold for matroid polynomials in general.

\subsection{Summary of results}

Some indication of what is to come can be gleaned from the following note by Marcolli in \cite[p.~71]{Mar10}: \enquote{graph hypersurfaces tend to have singularity loci of small codimension}.

Let $W\subseteq\KK^E$ be a realization of a matroid $\M$ of rank $\rk\M=\dim W$ on a set $E$ (see Definition~\ref{60}).
Fix coordinates $x_E=(x_e)_{e\in E}$.
There is an associated symmetric \emph{configuration (bilinear) form} $Q_W$ with linear homogeneous coefficients (see Definition~\ref{64}).
Its determinant is the \emph{configuration polynomial} (see Definition~\ref{48} and Lemma~\ref{36})
\[
\psi_W=\det Q_W=\sum_{B\in\B_\M}c_{W,B}\cdot\prod_{e\in B}x_e\in\KK[x_E]
\]
where $\B_\M$ denotes the set of bases of $\M$ and the coefficients $c_{W,B}\in\KK^*$ depend of the realization $W$.
The \emph{configuration hypersurface} defined by $\psi_W$ is the scheme
\[
X_W=\Spec(\KK[x_E]/\ideal{\psi_W})\subseteq\KK^E.
\]
It can be seen as the \emph{first degeneracy scheme} of $Q_W$ (see Definition~\ref{49}).
The \emph{second degeneracy scheme} $\Delta_W\subseteq\KK^E$ of $Q_W$, defined by the submaximal minors of $Q_W$, is a subscheme of the \emph{Jacobian scheme} $\Sigma_W\subseteq\KK^E$ of $X_W$, defined by $\psi_W$ and its partial derivatives (see Lemma~\ref{50}).
The latter defines the non-smooth locus of $X_W$ over $\KK$, which is the singular locus of $X_W$ if $\KK$ is perfect (see Remark~\ref{14}).
Patterson showed $\Sigma_W$ and $\Delta_W$ have the same underlying reduced scheme (see Theorem~\ref{13}), that is,
\[
\Delta_W\subseteq\Sigma_W\subseteq\KK^E,\quad \Sigma_W^\red=\Delta_W^\red.
\]
We give a simple proof of this fact.
He mentions that he does not know the reduced scheme structure (see \cite[p.~696]{Pat10}).
We show that $\Sigma_W$ is typically not reduced (see Example~\ref{45}), whereas $\Delta_W$ often is.
Our main results from Theorems~\ref{112}, \ref{100}, \ref{101} and \ref{133} can be summarized as follows.


\begin{mainthm}
Let $\M$ be a matroid on the set $E$ with a linear realization $W\subseteq\KK^E$ over a field $\KK$.
Then the configuration hypersurface $X_W$ is reduced and generically smooth over $\KK$.
Moreover, the second degeneracy scheme $\Delta_W$ is also reduced and agrees with $\Sigma_W^\red$, the non-smooth locus of $X_W$ over $\KK$.
Unless $\KK$ has characteristic $2$, the Jacobian scheme $\Sigma_W$ is generically reduced.

Suppose now that $\M$ is connected.
Then $X_W$ is integral unless $\M$ has rank zero.
Suppose in addition that the rank of $\M$ is at least $2$.
Then $\Delta_W$ is Cohen--Macaulay of codimension $3$ in $\KK^E$.
If, moreover, $\M$ is $3$-connected, then $\Delta_W$ is integral.\qed
\end{mainthm}


Note that $X_W=\emptyset$ if $\rk\M=0$ and $\Sigma_W=\emptyset=\Delta_W$ if $\rk\M\le1$ (see Remarks~\ref{6} and \ref{20}.\eqref{20a}).
It suffices to require the connectedness hypotheses after deleting all loops (see Remark~\ref{84}).
If $\M$ is disconnected even after deleting all loops, then $\Sigma_W$ and hence $\Delta_W$ has codimension $2$ in $\KK^E$ (see Proposition~\ref{112}).

While our main objective is to establish the results above, along the way we continue the systematic study of configuration polynomials in the spirit of \cite{BEK06,Pat10}. 
For instance, we describe the behavior of configuration polynomials with respect to connectedness, duality, deletion/contraction and $2$-separations (see Propositions~\ref{4}, \ref{30}, \ref{3} and \ref{63}).
Patterson showed that the \emph{second Symanzik polynomial} associated with a Feynman graph is, in fact, a configuration polynomial.
More precisely, we explain that its dual, the \emph{second Kirchhoff polynomial}, is associated with the quotient of the graph configuration by the momentum parameters (see Proposition~\ref{78}).
In this way, Patterson's result becomes a special case of a formula for configuration polynomials of elementary quotients (see Proposition~\ref{28}).

\subsection{Outline of the proof}\label{191}

The proof of the Main Theorem intertwines methods from matroid theory, commutative algebra and algebraic geometry.
In order to keep our arguments self-contained and accessible, we recall preliminaries from each of these subjects and give detailed proofs (see \S\ref{58}, \S\ref{57} and \S\ref{108}).
One easily reduces the claims to the case where $\M$ is connected (see Proposition~\ref{4} and Theorem~\ref{101}).

An important commutative algebra ingredient is a result of Kutz (see \cite{Kut74}):
the grade of an ideal of submaximal minors of a symmetric matrix cannot exceed $3$, and equality forces the ideal to be perfect.
Kutz' result applies to the defining ideal of $\Delta_W$.
The codimension of $\Delta_W$ in $\KK^E$ is therefore bounded by $3$ and $\Delta_W$ is Cohen--Macaulay in case of equality (see Proposition~\ref{44}).
In this case, $\Delta_W$ is pure-dimensional, and hence, it is reduced if it is generically reduced (see Lemma~\ref{176}).

On the matroid side our approach makes use of \emph{handles} (see Definition~\ref{126}), which are called \emph{ears} in case of graphic matroids.
A \emph{handle decomposition} builds up any connected matroid from a circuit by successively attaching handles (see Definition~\ref{196}).
Conversely, this yields for any connected matroid which is not a circuit a \emph{non-disconnective} handle which leaves the matroid connected when deleted (see Definition~\ref{126}).
This allows one to prove statements on connected matroids by induction.

We describe the effect of deletion and contraction of a handle $H$ to the configuration polynomial (see Corollary~\ref{18}).
In case the Jacobian scheme $\Sigma_{W\setminus H}$ associated with the deletion $\M\setminus H$ has codimension $3$ we prove the same for $\Sigma_W$ (see Lemma~\ref{99}).
Applied to a non-disconnective $H$ it follows with Patterson's result that $\Delta_W$ reaches the dimension bound and is thus Cohen--Macaulay of codimension $3$ (see Theorem~\ref{100}).
We further identify three (more or less explicit) types of generic points with respect to a non-disconnective handle (see Corollary~\ref{104}).

In case $\ch\KK\ne2$, generic reducedness of $\Sigma_W$ implies (generic) reducedness of $\Delta_W$.
The schemes $\Sigma_W$ and $\Delta_W$ show similar behavior with respect to  deletion and contraction (see Lemmas~\ref{31} and \ref{32}).
As a consequence, generic reducedness can be proved along the same lines (see Lemma~\ref{115}).
In both cases, we have to show reducedness at all (the same) generic points.
In what follows, we restrict ourselves to $\Delta_W$.
Our proof proceeds by induction on the cardinality $\abs{E}$ of the underlying set $E$ of the matroid $\M$.

Unless $\M$ a circuit, the handle decomposition guarantees the existence of a non-disconnective handle $H$.
In case $H=\set{h}$ has size $1$, the scheme $\Delta_{W\setminus h}$ associated with the deletion $\M\setminus h$ is the intersection of $\Delta_W$ with the divisor $x_e$ (see Lemma~\ref{31}).
This serves to recover generic reducedness of $\Delta_W$ from $\Delta_{W\setminus h}$ (see Lemma~\ref{88}).
The same argument works if $H$ does not arise from a handle decomposition.

This leads us to consider non-disconnective handles independently of a handle decomposition.
They turn out to be special instances of maximal handles which form the \emph{handle partition} of the matroid (see Lemma~\ref{16}).
As a purely matroid-theoretic ingredient, we show that the number of non-disconnective handles is strictly increasing when adding handles (see Proposition~\ref{15}).
For handle decompositions of length $2$, a distinguished role is played by the prism matroid (see Example~\ref{201}).
Its handle partition consists of $3$ non-disconnective handles of size $2$ (see Lemmas~\ref{17} and \ref{42}).
Here an explicit calculation shows that $\Delta_W$ is reduced in the torus $(\KK^*)^6$ (see Lemma~\ref{19}).
The corresponding result for $\Sigma_W$ holds only if $\ch\KK\ne2$.

Suppose now that $\M$ is not a circuit and has no non-disconnective handles of size $1$.
Then $\M\setminus e$ might be disconnected for all $e\in E$ and does not qualify for an inductive step.
In this case, we aim instead for contracting $W$ by a suitable subset $G\subsetneq E$ which keeps $\M$ connected.
In the partial torus $\KK^F\times(\KK^*)^G$ where $F:=E\setminus G$, the scheme $\Delta_{W/G}$ associated with the contraction $\M/G$ relates to the \emph{normal cone} of $\Delta_W$ along the coordinate subspace $V(x_F)$ where $x_F=(x_f)_{f\in F}$ (see Lemma~\ref{32}).
To induce generic reducedness from $\Delta_{W/G}$ to $\Delta_W$, we pass through a \emph{deformation to the normal cone}, which is our main ingredient from algebraic geometry.
The role of $x_h$ above is then played by the deformation parameter $t$.

In algebraic terms, this deformation is represented by the \emph{Rees algebra} $\Rees_IR$ with respect to an ideal $I\unlhd R$, and the normal cone by the \emph{associated graded ring} $\gr_IR$ (see Definition~\ref{173}).
Passing through $\Rees_IR$, we recover generic reducedness of $R$ \emph{along} $V(I)$ from generic reducedness of $\gr_IR$ (see Definition~\ref{131} and Lemma~\ref{95}).
By assumption on $\M$, there are at least $3$ more elements in $E$ than maximal handles (see Proposition~\ref{15}), and $\M$ is the prism matroid in case of equality.
Based on a strict inequality, we use a codimension argument to construct a suitable partition $E=F\sqcup G$ for which \emph{all} generic points of $\Delta_W$ are \emph{along} $V(x_F)$ (see Lemma~\ref{86}).
This yields generic reducedness of $\Delta_W$ in this case (see Lemma~\ref{87}).
A slight modification of the approach also covers the generic points outside the torus $(\KK^*)^6$ if $\M$ is the prism matroid.
The case where $\M$ is a circuit is reduced to that where $\M$ is a triangle by successively contracting an element of $E$ (see Lemma~\ref{97}).
In this base case $\Delta_W$ is a reduced point, but $\Sigma_W$ is reduced only if $\ch\KK\ne2$ (see Example~\ref{107}).

Finally, suppose that $\M$ is a $3$-connected matroid.
Here we prove that $\Delta_W$ is irreducible and hence integral, which implies that $\Sigma$ is irreducible (see Theorem~\ref{133}).
We first observe that handles of (co)size at least $2$ are $2$-separations (see Lemma~\ref{16}.\eqref{16c}).
It follows that the handle decomposition consists entirely of non-disconnective $1$-handles (see Proposition~\ref{53}) and that all generic points of $\Delta_W$ lie in the torus $(\KK^*)^E$ (see Corollary~\ref{106}).
We show that the number of generic points is bounded by that of $\Delta_{W\setminus e}$ for all $e\in E$ (see Lemma~\ref{88}).
Duality switches deletion and contraction and identifies generic points of $\Delta_W$ and $\Delta_{W^\perp}$ (see Corollary~\ref{71}).
Using Tutte's wheels-and-whirls theorem, the irreducibility of $\Delta_W$ can therefore be reduced to the cases where $\M$ is a wheel $\W_n$ or a whirl $\W^n$ for some $n\geq3$ (see Example~\ref{145} and Lemma~\ref{132}).
For fixed $n$, we show that the schemes $X_W$, $\Sigma_W$ and $\Delta_W$ are all isomorphic for all realizations $W$ of $\W_n$ and $\W^n$ (see Proposition~\ref{146}).
An induction on $n$ with an explicit study of the base cases $n\le4$ finishes the proof (see Corollary~\ref{138} and Lemma~\ref{137}).

\subsection*{Acknowledgments}

The project whose results are presented here started with a research in pairs meeting at the Centro de Giorgi in Pisa in February 2018.
We thank the Institute for a pleasant stay in a stimulating research environment.
We also thank Aldo Conca, Raul Epure, Darij Grinberg, Delphine Pol and Karen Yeats for helpful comments.
We are grateful to the referees for a careful reading of the manuscript and resulting improvements to the exposition.
This is a preprint of an article published in Letters in Mathematical Physics. The final authenticated version is available online at: \url{https://doi.org/10.1007/s11005-020-01352-3}

\section{Matroids and configurations}\label{38}

Our algebraic objects of interest are associated with a realization of a matroid.
In this section, we prepare the path for an inductive approach driven by the underlying matroid structure.
Our main tool is the handle decomposition, a matroid version of the ear decomposition of graphs.

\subsection{Matroid basics}\label{58}

In this subsection, we review the relevant basics of matroid theory using Oxley's book (see \cite{Oxl11}) as a comprehensive reference.

Denote by $\Min\P$ and $\Max\P$ the set of minima and maxima of a poset $\P$.
Let $\M$ be a \emph{matroid} on a set $E=:E_\M$.
We use this font throughout to denote matroids.
With $2^E$ partially ordered by inclusion, $\M$ can be defined by a monotone submodular \emph{rank} function (see \cite[Cor.~1.3.4]{Oxl11})
\[
\rk=\rk_\M\colon 2^E\to\NN=\set{0,1,2,\dots}
\]
with $\rk(S)\le\abs{S}$ for any subset $S\subseteq E$. 
The \emph{rank} of $\M$ is then
\[
\rk\M:=\rk_\M(E).
\]
Alternatively, it can be defined in terms of each of the following collections of subsets of $E$ (see \cite[Prop.~1.3.5, p.~28]{Oxl11}):
\begin{itemize}
\item \emph{independent sets} $\I_\M=\set{I\subseteq E\xmid\abs{I}=\rk_\M(I)}\subseteq2^E$,
\item \emph{bases} $\B_\M=\Max\I_\M=\set{B\subseteq E\xmid\abs{B}=\rk_\M(B)=\rk\M}\subseteq2^E$,
\item \emph{circuits} $\C_\M=\Min(2^E\setminus\I_\M)\subseteq 2^E$,
\item \emph{flats} $\L_\M=\set{F\subseteq E \xmid\forall e\in E\setminus F\colon\rk_\M(F\cup\set{e})>\rk_\M(F)}$.
\end{itemize}
For instance (see \cite[Lem.~1.3.3]{Oxl11}), for any subset $S\subseteq E$,
\begin{equation}\label{199}
\rk_\M(S)=\max\set{\abs{I}\xmid S\supseteq I\in\I_\M}.
\end{equation}
The \emph{closure} operator of $\M$ is defined by (see \cite[Lem.~1.4.2]{Oxl11})
\begin{equation}\label{195}
\cl_\M\colon 2^E\mapsto\L_\M,\quad\rk_\M=\rk_\M\circ\cl_\M.
\end{equation}


The following matroid plays a special role in the proof of our main result.

\begin{dfn}[Prism matroid]\label{200}
The \emph{prism matroid} has underlying set $E$ with $\abs{E}=6$ and circuits
\[
\C_{\M}=\set{\set{e_1,e_2,e_3,e_4},\set{e_1,e_2,e_5,e_6},\set{e_3,e_4,e_5,e_6}}.
\] 
The name comes from the observation that its independent sets $\I_{\M}$ are the affinely independent subsets of the vertices of the triangular prism (see Figure~\ref{113}).
\begin{figure}[h]
\caption{The triangular prism.}\label{113}
\begin{tikzpicture}[scale=2,baseline=(current bounding box.center)]
\def\dx{1.4}\def\dy{1}
\tikzstyle{root}=[circle,draw,inner sep=1.2pt,fill=black]
\node (1) at (0,0) [root,label=below:$e_1$] {};
\node (3) at (1,0) [root,label=below:$e_3$] {};
\node (5) at (0.5,0.707) [root,label=above:$e_5$] {};
\node (2) at (0+\dx,0+\dy) [root,label=below:$e_2$] {};
\node (4) at (1+\dx,0+\dy) [root,label=below:$e_4$] {};
\node (6) at (0.5+\dx,0.707+\dy) [root,label=above:$e_6$] {};
\draw (1) -- (3) -- (5) -- (1);
\draw (2) -- (4) -- (6) -- (2);
\draw (1) -- (2);
\draw (3) -- (4);
\draw (5) -- (6);
\end{tikzpicture}
\end{figure}
\end{dfn}


The elements of $E\setminus\bigcup\B_\M$ and $\bigcap\B_\M$ are called \emph{loops} and \emph{coloops} in $\M$, respectively.
A matroid is \emph{free} if $E\in\B_\M$, that is, every $e\in E$ is a coloop in $\M$.
By a \emph{$k$-circuit} in $\M$ we mean a circuit $C\in\C_\M$ with $\abs{C}=k$ elements, $3$-circuits are called \emph{triangles}.

The circuits in $\M$ give rise to an equivalence relation on $E$ by declaring $e,f\in E$ equivalent if $e=f$ or $e,f\in C$ for some $C\in\C_\M$ (see \cite[Prop.~4.1.2]{Oxl11}).
The corresponding equivalence classes are the \emph{connected components} of $\M$.
If there is at most one such a component, then $\M$ is said to be \emph{connected}.
The \emph{connectivity function} of $\M$ is defined by
\[
\lambda_\M\colon 2^E\to\NN,\quad\lambda_\M(S):=\rk_\M(S)+\rk_\M(E\setminus S)-\rk(\M).
\]
For $k\ge1$, a subset $S\subseteq E$, or the partition $E=S\sqcup(E\setminus S)$, is called a \emph{$k$-separation} of $\M$ if
\[
\lambda_\M(S)<k\le\min\set{\abs{S},\abs{E\setminus S}}.
\]
It is called \emph{exact} if the latter is an equality.
The \emph{connectivity} $\lambda(\M)$ of $\M$ is the minimal $k$ for which there is a $k$-separation of $\M$, or $\lambda(\M)=\infty$ if no such exists.
The matroid $\M$ is said to be \emph{$k$-connected} if $\lambda(\M)\ge k$.
Connectedness is the special case $k=2$.

We now review some standard constructions of new matroids from old.
Their geometric significance is explained in \S\ref{57}.

The \emph{direct sum} $\M_1\oplus \M_2$ of matroids $\M_1$ and $\M_2$ is the matroid on $E_{\M_1}\sqcup E_{\M_2}$ with independent sets
\begin{equation}\label{209}
\I_{\M_1\oplus \M_2}:=\set{I_1\sqcup I_2\xmid I_1\in\I_{\M_1}, I_2\in\I_{\M_2}}.
\end{equation}
The sum is \emph{proper} if $E_{\M_1}\ne\emptyset\ne E_{\M_2}$.
Connectedness means that a matroid is not a proper direct sum (see \cite[Prop.~4.2.7]{Oxl11}).
In particular, any (co)loop is a connected component.

Let $F\subseteq E$ be any subset.
Then the \emph{restriction matroid} $\M\vert_F$ is the matroid on $F$ with independent sets and bases (see \cite[\nopp 3.1.12, 3.1.14]{Oxl11})
\begin{equation}\label{153}
\I_{\M\vert_F}=\I_\M\cap 2^F,\quad\B_{\M\vert_F}=\Max\set{B\cap F\xmid B\in\B_\M}.
\end{equation}
Its set of circuits is (see \cite[\nopp 3.1.13]{Oxl11})
\begin{equation}\label{142}
\C_{\M\vert_F}=\C_{\M}\cap 2^F.
\end{equation}
By definition, $\rk_{\M\vert_F}=\rk_\M\vert_{2^F}$, so we may omit the index without ambiguity.
Thinking of restriction to $E\setminus F$ as an operation that deletes elements in $F$ from $E$, one defines the \emph{deletion matroid} 
\[
\M\setminus F:=\M\vert_{E\setminus F}.
\]
The \emph{contraction matroid} $\M/F$ is the matroid on $E\setminus F$ with independent sets and bases (see \cite[Prop.~3.1.7, Cor.~3.1.8]{Oxl11})
\begin{align}\label{154}
\I_{\M/F}&=\set{I\subseteq E\setminus F\xmid I\cup B\in\I_\M\text{ for some/every }B\in\B_{\M\vert F}},\\\nonumber
\B_{\M/F}&=\set{B'\subseteq E\setminus F\xmid B'\cup B\in\B_\M\text{ for some/every }B\in\B_{\M\vert F}}.
\end{align}
Its circuits are the minimal non-empty sets $C\setminus F$ where $C\in\C_{\M}$ (see \cite[Prop.~3.1.10]{Oxl11}), that is,
\begin{equation}\label{143}
\C_{\M/F}=\Min\set{C\setminus F\mid F\not\supseteq C\in\C_{\M}}.
\end{equation}

In \S\ref{57}, $E$ will be a basis and $E^\vee$ the corresponding dual basis.
We often identify $E=E^\vee$ by the bijection
\begin{equation}\label{158}
\nu\colon E\to E^\vee,\quad e\mapsto e^\vee.
\end{equation}
The complement of a subset $S\subseteq E$ corresponds to
\[
S^\perp:=\nu(E\setminus S)\subseteq E^\vee.
\]
The \emph{dual matroid} $\M^\perp$ is the matroid on $E^\vee$ with bases  
\begin{equation}\label{212}
\B_{\M^\perp}=\set{B^\perp\xmid B\in \B_\M}.
\end{equation}
In particular, we have (see \cite[\nopp 2.1.8]{Oxl11})
\[
\rk\M+\rk\M^\perp=\abs{E}.
\]
Connectivity is invariant under dualizing (see \cite[Cor.~8.1.5]{Oxl11}),
\begin{equation}\label{207}
\lambda_\M=\lambda_{\M^\perp}\circ\nu,\quad\lambda(\M)=\lambda(\M^\perp).
\end{equation}
We use $\nu^{-1}$ in place of \eqref{158} for $\M^\perp$, so that $S^{\perp\perp}=S$.
For subsets $F\subseteq E$ and $G\subseteq E^\vee$, one can identify (see \cite[\nopp 3.1.1]{Oxl11})
\begin{align}\label{144}
(\M/F)^\perp&=\M^\perp\vert_{F^\perp}=\M^\perp\setminus\nu(F),\\
\nonumber(\M\setminus\nu^{-1}(G))^\perp&=(\M\vert_{G^\perp})^\perp=\M^\perp/G.
\end{align}
Various matroid data of $\M^\perp$ is also considered as \emph{co}data of $\M$.
A \emph{triad} of $\M$ is a $3$-cocircuit of $\M$, that is, a triangle of $\M^\perp$.


\begin{exa}[Uniform matroids]\label{118}
The \emph{uniform matroid} $\U_{r,n}$ of rank $r\geq0$ on a set $E$ of size $\abs{E}=n$ has bases
\[
\B_{\U_{r,n}}=\set{B\subseteq E\mid \abs{B}=r}.
\]
For $r=n$ it is the free matroid of rank $r$.
It is connected if and only if $0<r<n$.
By definition, $\U_{r,n}^\perp=\U_{n-r,n}$ for all $0\leq r\leq n$.

Informally, we refer to a matroid $\M$ on $E$ for which $E\in\C_\M$, and hence, $\C_\M=\set{E}$, as a \emph{circuit}, and as a \emph{triangle} if $\abs{E}=3$.
It is easily seen that such a matroid is $\U_{n-1,n}$ where $n=\abs{E}$, and
that $\lambda(\U_{n-1,n})=2$.
\end{exa}

\subsection{Handle decomposition}\label{54}

In this subsection, we investigate handles as building blocks of connected matroids.


\begin{dfn}[Handles]\label{126}
Let $\M$ be a matroid.
A subset $\emptyset\ne H\subseteq E$ is a \emph{handle} in $\M$ if $C\cap H\neq\emptyset$ implies $H\subseteq C$ for all $C\in\C_\M$. 
Write $\H_\M$ for the set of handles in $\M$, ordered by inclusion.
A \emph{subhandle} of $H\in\H_\M$ is a subset $\emptyset\ne H'\subseteq H$.
We call $H\in\H_\M$
\begin{itemize}
\item \emph{proper} if $H\ne E$,
\item \emph{maximal} if $H\in\Max\H_\M$,
\item a \emph{$k$-handle} if $\abs{H}=k$,
\item \emph{disconnective} if $\M\setminus H$ is disconnected and
\item \emph{separating} if $\min\set{\abs{H},\abs{E\setminus H}}\ge2$.
\end{itemize}
\end{dfn}


Singletons $\set{e}$ and subhandles are handles.
If $\bigcup\C_\M\ne E$, then $E\setminus\bigcup\C_\M\in\Max\H_\M$ and is a union of coloops.
The maximal handles in $\bigcup\C_\M$ are the minimal non-empty intersections of all subsets of $\C_\M$.
Together they form the \emph{handle partition} of $E$
\[
E=\bigsqcup_{H\in\Max\H_\M}H,
\]
which refines the partition of $\bigcup\C_\M$ into connected components.
In particular, each circuit is a disjoint union of maximal handles.
For any subset $F\subseteq E$, \eqref{142} yields an inclusion
\[
\H_\M\cap 2^F\subseteq\H_{\M\vert_F}.
\]


\begin{lem}[Handle basics]\label{16}
Let $\M$ be a matroid and $H\in\H_\M$.

\begin{enumerate}[(a)]

\item\label{16e} If $H=E$, then $\M=\U_{r,n}$ where $n=\abs{E}\ge1$ and $r\in\set{n-1,n}$ (see Example~\ref{118}).
In the latter case, $\abs{E}=1$ if $\M$ is connected.

\item\label{16d}
Either $H\in\I_\M$ or $H\in\C_\M$.
In the latter case, $H$ is maximal and a connected component of $\M$.
In particular, if $\M$ is connected and $H$ is proper, then $H\in\I_\M$, $H\subsetneq C$ for some circuit $C\in\C_\M$, and $H\in\C_{\M/(E\setminus H)}$.

\item\label{16a}
For any subhandle $\emptyset\ne H'\subseteq H$, $H\setminus H'$ consists of coloops in $\M\setminus H'$.
In particular, non-disconnective handles are maximal.

\item\label{16b}
If $H\not\in\C_\M$, then there is a bijection
\[
\C_\M\to\C_{\M/H},\quad C\mapsto C\setminus H.
\]
If $H\not\in\Max\H_\M$, then there is a bijection
\[
\Max\H_\M\to\Max\H_{\M/H},\quad H'\mapsto H'\setminus H,
\]
which identifies non-disconnective handles.
In this case, the connected components of $\M$ which are not contained in $H\setminus\bigcup\C_\M$ correspond to the connected components of $\M/H$.

\item\label{16c} Suppose that $\M$ is connected and $H$ is proper.
Then
\[
\rk(\M/H)=\rk\M-\abs{H},\quad \lambda_\M(H)=1.
\]
In particular, if $H$ is separating, then $H$ is a $2$-separation of $\M$.

\end{enumerate}
\end{lem}

\begin{proof}\
\begin{asparaenum}[(a)]

\item Suppose that $H=E$.
Then $\C_\M\subseteq\set{E}$ and $\M=\U_{n-1,n}$ in case of equality.
Otherwise, $\C_\M=\emptyset$ implies $\B_\M=\set{E}$ and $\M=\U_{n,n}$ (see \cite[Prop.~1.1.6]{Oxl11}).

\item Suppose that $H\not\in\I_\M$.
Then there is a circuit $H\supseteq C\in\C_\M$.
By definition of handle and incomparability of circuits, $H=C\setminus(E\setminus H)\in\C_{\M/(E\setminus H)}$ (see \eqref{143}) and $H=C$ is disjoint from all other circuits and hence a connected component of $\M$.

\item Suppose that $h\in H\setminus H'$ is not a coloop in $\M\setminus H'$.
Then $h\in C\cap H$ for some $C\in\C_{\M\setminus H'}\subseteq\C_\M$ (see \eqref{142}) and hence $H'\subseteq H\subseteq C$ since $H$ is a handle, a contradiction.

\item The first bijection follows from \eqref{143} with $F=H$.
The remaining claims follow from the discussion preceding the lemma.

\item Part \eqref{16d} yields the first equality (see \cite[Prop.~3.1.6]{Oxl11}) along with a circuit $H\ne C\in\C_\M$.
Pick a basis $B\in\B_{\M\setminus H}$.
Clearly $S:=B\sqcup H$ spans $\M$.
For any $h\in H$, we check that $S\setminus\set{h}\in\I_\M$.
Otherwise, there is a circuit $S\setminus\set{h}\supseteq C\in\C_\M$.
Since $C\not\subseteq B$ and by definition of handle, we have $H\cap C\neq\emptyset$ and hence $h\in H\subseteq C$, a contradiction.
It follows that $\rk\M=\abs{S}-1=\rk(\M\setminus H)+\abs{H}-1$ and hence the second equality.\qedhere

\end{asparaenum}
\end{proof}


\begin{prp}[Handles in $3$-connected matroids]\label{53}
Let $\M$ be a $3$-connected matroid on $E$ with $\abs{E}>3$.
Then all its handles are non-disconnective $1$-handles.
\end{prp}

\begin{proof}
Let $H\in\H_\M$ be any handle.
By Lemma~\ref{16}.\eqref{16e}, $H$ must be proper.
By Lemma~\ref{16}.\eqref{16c}, $H$ is not separating, that is, $\abs{H}=1$ or $\abs{E\setminus H}=1$.
In the latter case, $\M$ is a circuit by Lemma~\ref{16}.\eqref{16d} and hence not $3$-connected (see Example~\ref{118}).
So $H$ is a $1$-handle.

Suppose that $H$ is disconnective.
Consider the deletion $\M':=\M\setminus H$ on the set $E':=E\setminus H$.
Pick a connected component $X$ of $\M'$ of minimal size $\abs{X}<\abs{E'}$.
Since $H\ne\emptyset$ and $\abs{E}>3$, both $X\cup H$ and its complement $E\setminus (X\cup H)=E'\setminus X$ have at least $2$ elements.
Since $X$ is a connected component of $\M'$ and by Lemma~\ref{16}.\eqref{16c},
\[
\rk(X)+\rk(E'\setminus X)=\rk\M'=\rk\M.
\]
Since $\rk(X\cup H)\leq\rk(X)+\abs{H}=\rk(X)+1$, it follows that
\[
\lambda_\M(X\cup H)=\rk(X\cup H)+\rk(E\setminus(X\cup H))-\rk\M<2.
\]
Whence $X\cup H$ is a $2$-separation of $\M$, a contradiction.
\end{proof}


The following notion is the basis for our inductive approach to connected matroids.


\begin{dfn}[Handle decompositions]\label{196}
Let $\M$ be a connected matroid.
A \emph{handle decomposition} of length $k$ of $\M$ is a filtration 
\[
\C_\M\ni F_1\subsetneq\cdots\subsetneq F_k=E
\]
such that $\M\vert_{F_i}$ is connected and $H_i:=F_i\setminus F_{i-1}\in\H_{\M\vert_{F_i}}$ for $i=2,\dots,k$.
\end{dfn}


By Lemma~\ref{16}.\eqref{16d} and \eqref{142}, a handle decomposition yields circuits
\begin{equation}\label{197}
C_1:=F_1\in\C_\M,\quad H_i\subsetneq C_i\in\C_{\M\vert_{F_i}}\subseteq\C_\M,\quad i=2,\dots,k.
\end{equation}
Conversely, it can be constructed from a suitable sequence of circuits.


\begin{exa}[Handle decomposition of the prism matroid]\label{201}
The prism matroid (see Example~\ref{200}) has handle partition
\[
E=\set{e_1,e_2}\sqcup\set{e_3,e_4}\sqcup\set{e_5,e_6}.
\]
A handle decomposition of length $2$ is given by
\[
F_1=\set{e_1,e_2,e_3,e_4}\subsetneq F_2=E.
\]
Note that all handles are proper, maximal, separating $2$-handles.
\end{exa}


\begin{prp}[Existence of handle decompositions]\label{12}
Let $\M$ be a connected matroid and $C_1\in\C_\M$.
Then there is a handle decomposition of $\M$ starting with $F_1=C_1$.
\end{prp}

\begin{proof}
There is a sequence of circuits $C_1,\ldots,C_k\in\C_\M$ which defines a filtration $F_i:=\bigcup_{j\le i}C_j$ such that $C_i\cap F_{i-1}\ne\emptyset$ and $C_i\setminus F_{i-1}\in\C_{\M/F_{i-1}}$ for $i=2,\dots,k$ (see \cite{CH96}).
The hypothesis $C_i\cap F_{i-1}\neq\emptyset$ implies that $\M\vert_{F_i}$ is connected for $i=1,\dots,k$.

It remains to check that $H_i=C_i\setminus F_{i-1}\in\H_{\M\vert_{F_i}}$ for $i=2,\dots,k$. 
Since circuits are nonempty, $\emptyset\ne H_i\subsetneq F_i$.
Let $C\in\C_{\M\vert_{F_i}}$ be a circuit such that $e\in C\cap H_i\subseteq C\cap C_i$. 
Suppose by way of contradiction that $H_i\not\subseteq C$.
Then there exists some $d\in C_i\setminus(C\cup F_{i-1})$.
By the strong circuit elimination axiom (see \cite[Prop.~1.4.12]{Oxl11}), there is a circuit $C'\in\C_{\M\vert_{F_i}}\subseteq\C_\M$ (see \eqref{142}) for which $d\in C'\subseteq (C\cup C_i) \setminus\set{e}$.
Then $C'\setminus F_{i-1}\subseteq C_i\setminus F_{i-1}\in\C_{\M/F_{i-1}}$ by assumption on $C_i$.
It follows that either $C'\subseteq F_{i-1}$ or $C'\setminus F_{i-1}=C_i\setminus F_{i-1}$ (see \eqref{143}).
The former is impossible because $C'\ni d\not\in F_{i-1}$, and the latter because $C'\cup F_{i-1}\not\ni e\in C_i$.
\end{proof}


In the sequel, we develop a bound for the number of non-disconnective handles.


\begin{lem}[Non-disconnective handles]\label{10}
Let $\M$ be a connected matroid.
Suppose that $H\in\H_\M$ and $H'\in\H_{\M\setminus H}$ are non-disconnective with $H\cup H'\ne E$.
Then there is a non-disconnective handle $H''\in\H_\M$ such that $H''\subseteq H'$, with equality if $H'\in\H_\M$.
\end{lem}

\begin{proof}
By hypothesis, $\M$ and $\M\setminus H$ are connected and $H\cup H'\ne E$ implies that both $H$ and $H'$ are proper handles.
Then Lemma~\ref{16}.\eqref{16d} yields circuits $C\in\C_\M$ and $C'\in\C_{\M\setminus H}\subseteq\C_\M$ (see \eqref{142}) such that $H\subsetneq C$ and $H'\subsetneq C'$.

Suppose that $C\subseteq H\cup H'$.
Then the strong circuit elimination axiom (see \cite[Prop.~1.4.12]{Oxl11}) yields a circuit $C''\in\C_{\M}$ for which $C''\subseteq H\cup C'$, $H'\not\subseteq C''$ and $C''\not\subseteq H\cup H'$.
Since $C''\subsetneq C'$ contradicts incomparability of circuits, $H\subsetneq C''$ since $H$ is a handle and Lemma~\ref{16}.\eqref{16d} forbids equality.

Replacing $C$ by $C''$ if necessary, we may assume that $H'\not\subseteq C$ and $C\not\subseteq H\cup H'$.
In particular, $H'':=H'\setminus C\in\H_{\M\setminus H}$ and $H''=H'$ if $H'\in\H_\M$.
Since $\M\setminus(H\cup H')$ is connected by hypothesis, $C$ witnesses the fact that $H$,  $C\cap H'$ and $E\setminus(H\cup H')$ are in the same connected component of $\M\setminus H''$ (see \eqref{142}).
In other words, $\M\setminus H''$ is connected.
If $H''\in\H_\M$ is a handle, then $H''$ is therefore non-disconnective.

Otherwise, there is a circuit $C''\in\C_\M$ such that $\emptyset\ne C''\cap H''\ne H''$.
In particular, $H\subseteq C''$ since otherwise $C''\cap H=\emptyset$ and $C''\in\C_{\M\setminus H}$ (see \eqref{142}) which would contradict $H''\in\H_{\M\setminus H}$.
This means that $C''$ connects $H$ with $C''\cap H''$.
We may therefore replace $H''$ by $\emptyset\ne H''\setminus C''\subsetneq H''$ and iterate. 
Then $H''\in\H_\M$ after finitely many steps.
\end{proof}


\begin{lem}[Handle decomposition of length $2$]\label{17}
Let $\M$ be a connected matroid with a handle decomposition of length $2$.
Then $\M$ has at least $3$ (disjoint) non-disconnective handles.
In case of equality, they form the handle partition of $\M$.
\end{lem}

\begin{proof}
Consider the circuits $C':=C_1\in\C_\M$, $C:=C_2\in\C_\M$ (see \eqref{197}), the non-disconnective handle $H:=H_2\in\H_\M$ and the subsets $\emptyset\ne H':=C'\setminus C\subseteq E$ and $\emptyset\ne H'':=C\cap C'\subseteq E$.
Then $E=H\sqcup H'\sqcup H''$ and $C'=H'\cup H''$ and $C=H\cup H''$.

Let $C''\in\C_\M$ be any circuit with $C'\ne C''\ne C$.
By incomparability of circuits, $C''\not\subseteq C'$ and hence $H\subseteq C''$ since $H$ is a handle.
By Lemma~\ref{16}.\eqref{16b}, we may assume that $\abs{H}=1$.
Then $H'\subseteq C''$ (see \cite[\S1.1, Exc.~5]{Oxl11}).
In particular, $H'\in\H_\M$ is a third non-disconnective handle.
If $H\cup H'\subseteq C''$ is an equality, then also $H''\in\H_\M$ is a non-disconnective handle and $H\sqcup H'\sqcup H''$ is the handle decomposition.

Otherwise, $C''$ witnesses the fact that $H$, $H'$ and $\emptyset\ne C''\cap H''\ne H''$ are in the same connected component of $\M\vert_{C''}$ (see \eqref{142}).
If $H''\setminus C''\in\H_\M$ is a handle, then it is therefore non-disconnective.
Otherwise, iterating yields a third non-disconnective handle $H''\setminus C''\supseteq H'''\in\H_\M$.
\end{proof}


\begin{exa}[Unexpected handles]\label{217}
Consider the matroid $\M$ on $E=\set{1,\dots,6}$ whose bases
\begin{gather*}
\B_\M=\{\set{1,2,3,4},\set{1,2,3,5},\set{1,2,4,5},\set{1,3,4,5},\set{2,3,4,5},\\
\set{1,2,3,6},\set{1,2,4,6},\set{1,3,4,6},\set{2,3,4,6},\\
\set{1,3,5,6},\set{1,4,5,6},\set{2,3,5,6},\set{2,4,5,6}\}
\end{gather*}
index those sets of columns of the matrix
\[
\begin{pmatrix}
1 & 0 & 0 & 0 & 1 & 1 \\
0 & 1 & 0 & 0 & 1 & 1 \\
0 & 0 & 1 & 0 & 1 & 2 \\
0 & 0 & 0 & 1 & 1 & 2 
\end{pmatrix}
\]
which form a basis of $\FF_3^4$ (see Remark~\ref{65}).
Its circuits and maximal handles are given by
\begin{align*}
\C_\M&=\set{F_1:=\set{1,2,3,4,5},\set{1,2,3,4,6},\set{1,2,5,6},\set{3,4,5,6}},\\
\Max\H_\M&=\set{\set{1,2},\set{3,4},\set{5},\set{6}=:H_2}.
\end{align*}
In particular, $\M$ is connected with a handle decomposition
\[
F_1\subsetneq F_1\sqcup H_2=:F_2=E
\]
of length $2$.
Here all $4$ maximal handles are non-disconnective and the inequality in Lemma~\ref{17} is strict.
This can happen because $\M$ is not a graphic matroid (see Lemma~\ref{42}).
\end{exa}


\begin{prp}[Lower bound for non-disconnective handles]\label{15}
Let $\M$ be a connected matroid with a handle decomposition of length $k\ge2$.
Then $\M$ has at least $k+1$ (disjoint) non-disconnective handles.
\end{prp}

\begin{proof}
We argue by induction on $k$.
The base case $k=2$ is covered by Lemma~\ref{17}.
Suppose now that $k\ge3$.
By hypothesis (see Definition~\ref{196}), $H_k\in\H_\M$ is a non-disconnective handle and the connected matroid $\M\setminus H_k=\M\vert_{F_{k-1}}$ has a handle decomposition of length $k-1$.
By induction, there are $k$ (disjoint) non-disconnective handles $H'_0,\dots,H'_{k-1}\in\H_{\M\setminus H_k}$.
Since $k\ge3$ and handles are non-empty, $H_k\cup H'_i\ne E$ for $i=0,\dots,k-1$.
For each $i=0,\dots,k-1$, Lemma~\ref{10} now yields a non-disconnective handle $H'_i\supseteq H''_i\in\H_\M$.
Thus, $\M$ has $k+1$ (disjoint) non-disconnective handles $H''_0,\dots,H''_{k-1},H_k\in\H_\M$.
\end{proof}


We conclude this section with an observation.

\begin{lem}[Existence of circuits]\label{66}
Let $\M$ be a connected matroid of rank $\rk\M\ge2$.
Then there is a circuit $C\in\C_\M$ of size $\abs{C}\ge3$.
\end{lem}

\begin{proof}
Pick $e\in E$.
Since $\M$ is connected, $E$ is the union of all circuits $e\in C\in\C_\M$.
Suppose that there are only $2$-circuits.
Then $E=\cl_\M(e)$ (see \cite[Prop.~1.4.11.(ii)]{Oxl11}) and hence $\rk\M=1$ (see \eqref{195}), a contradiction.
\end{proof}

\subsection{Configurations and realizations}\label{57}

Our objects of interest are not associated with a matroid itself but with a realization as defined in the following.
All matroid operations we consider come with a counterpart for realizations.
For graphic matroids, these agree with familiar operations on graphs (see \S\ref{61}).


Fix a field $\KK$ and denote the $\KK$-dualizing functor by
\[
-^\vee:=\Hom_\KK(-,\KK).
\]
We consider a finite set $E$ as a basis of the \emph{based $\KK$-vector space} $\KK^E$ and denote by $E^\vee=(e^\vee)_{e\in E}$ the dual basis of 
\begin{equation}\label{188}
(\KK^E)^\vee=\KK^{E^\vee}.
\end{equation}
By abuse of notation, we set $S^\vee:=(e^\vee)_{e\in S}$ for any subset $S\subseteq E$.


We consider configurations as defined by Bloch, Esnault and Kreimer (see \cite[\S1]{BEK06}).

\begin{dfn}[Configurations and realizations]\label{60}
Let $E$ be a finite set.
A $\KK$-vector subspace $W\subseteq\KK^E$ is called a \emph{configuration} (over $\KK$).
It defines a matroid $\M_W$ on $E$ with independent sets
\begin{equation}\label{198}
\I_{\M_W}=\set{S\subseteq E\mid S^\vee\vert_W\text{ is $\KK$-linearly independent in $W^\vee$}}.
\end{equation}
Let $\M$ be a matroid and $W\subseteq\KK^E$ a configuration (over $\KK$).
If $\M=\M_W$, then $W$ is called a \emph{(linear) realization} of $\M$ and $\M$ is called \emph{(linearly) realizable} (over $\KK$).
A matroid is called \emph{binary} if it is realizable over $\FF_2$.
A configuration $W\subseteq\KK^E$ is called \emph{totally unimodular} if $\ch\KK=0$ and $W$ admits a basis whose coefficient matrix with respect to $E$ has all (maximal) minors in $\set{0,\pm1}$.
A matroid is called \emph{regular} if it admits a totally unimodular realization.
Equivalently, a regular matroid is realizable over every field (see \cite[Thm.~6.6.3]{Oxl11}).
\end{dfn}


Since $E^\vee\vert_W$ generates $W^\vee$, we have (see \eqref{198})
\begin{equation}\label{202}
\rk(\M_W)=\dim W^\vee=\dim W.
\end{equation}


\begin{rmk}[Matroids and linear algebra]\label{65}
The notions in matroid theory (see \S\ref{58}) are derived from linear (in)dependence over $\KK$.
Let $W\subseteq\KK^E$ be a realization of a matroid $\M$.
Pick a basis $w=(w^1,\dots,w^r)$ of $W$ where $r:=\rk\M$ (see \eqref{202}).
For each $e\in E$, $e^\vee\vert_W$ is then represented by the vector $(w^i_e)_i\in\KK^r$ where $w^i_e:=e^\vee(w^i)$ for $i=1,\dots,r$.
Order $E=\set{e_1,\dots,e_n}$ and set $w^i_j:=w^i_{e_j}$ for $j=1,\dots,n$.
Then these vectors form the columns of the coefficient matrix $A=(w^i_j)_{i,j}\in\KK^{r\times n}$ of $w$.
By construction, $W$ is the row span of $A$.
The matroid rank $\rk_\M(S)$ of any subset $S\subseteq E$ now equals the $\KK$-linear rank of the submatrix of $A$ with columns $S$ (see \eqref{199} and \eqref{198}).
An element $e\in E$ is a loop in $\M$ if and only if column $e$ of $A$ is zero; 
$e$ is a coloop in $\M$ if and only if column $e$ is not in the span of the other columns.
\end{rmk}


\begin{rmk}[Classical configurations]\label{68}
Suppose that $\M_W$ has no loops, that is, $e^\vee\vert_W\neq 0$ for each $e\in E$.
Then the images of the $e^\vee\vert_W$ in $\PP W^\vee$ form a \emph{projective point configuration} in the classical sense (see \cite{HC52}).
Dually, the hyperplanes $\ker(e^\vee)\cap W$ form a \emph{hyperplane arrangement} in $W$ (see \cite{OT92}), which is an equivalent notion in this case.
\end{rmk}


We fix some notation for realizations of basic matroid operations.
Any subset $S\subseteq E$ gives rise to an inclusion and a projection
\begin{equation}\label{206}
\iota_S\colon \KK^S\into\KK^E,\quad
\pi_S\colon\KK^E\onto\KK^E/\KK^{E\setminus S}=\KK^S
\end{equation}
of based $\KK$-vector spaces.


\begin{dfn}[Realizations of matroid operations]\label{33}

Let $W\subseteq\KK^E$ be a realization of a matroid $\M$, and let $F\subseteq E$ be any subset.

\begin{enumerate}[(a)]

\item\label{33c} The \emph{restriction configuration} (see \eqref{206})
\begin{align*}
W\vert_F&:=\pi_F(W)\subseteq\KK^F\\
&\cong(W+\KK^{E\setminus F})/\KK^{E\setminus F}\cong W/(W\cap\KK^{E\setminus F})
\end{align*}
realizes the restriction matroid $\M\vert_F$.

\item\label{33d} The \emph{deletion configuration}
\[
W\setminus F:=W\vert_{E\setminus F}
\]
realizes the deletion matroid $\M\setminus F$.
We write $W\setminus e:=W\setminus\set{e}$ for $e\in E$.

\item\label{33e} The \emph{contraction configuration}
\[
W/F:=W\cap\KK^{E\setminus F}\subseteq\KK^{E\setminus F}
\]
realizes the contraction matroid $\M/F$.

\item\label{33a} The \emph{dual configuration} (see \eqref{188})
\[
W^\perp:=(\KK^E/W)^\vee\subseteq\KK^{E^\vee}
\]
realizes the dual matroid $\M^\perp$.

\item\label{33b} Any $0\ne\varphi\in W^\vee$ defines an \emph{elementary quotient configuration}
\[
W_\varphi:=\ker\varphi\subseteq\KK^E.
\]

\end{enumerate}
\end{dfn}


\begin{rmk}\label{34}

Let $W\subseteq\KK^E$ be a realization of a matroid $\M$.

\begin{enumerate}[(a)]

\item\label{34a} An element $e\in E$ is a loop or coloop in $\M$ if and only if $W\subseteq\KK^{E\setminus\set{e}}$ or $W=(W\setminus e)\oplus\KK^\set{e}$, respectively.
In both cases, $W\setminus e=W/e\subseteq\KK^{E\setminus\set{e}}$.

\item\label{34b} For $0\ne\varphi\in W^\vee$, pick $w\in W\setminus W_\varphi$ and $e\notin E$.
Consider the configuration
\[
W_{\varphi,w}:=W_\varphi\oplus\KK\cdot(w+e)\subseteq\KK^{E\sqcup\set{e}}.
\]
Then $W_{\varphi,w}\setminus e=W$ and $W_{\varphi,w}/e=W_\varphi$.
By definition, $\M_{W_\varphi}$ is therefore an \emph{elementary quotient} of $\M_W$; it can be characterized in terms of the notion of a \emph{modular cut} (see \cite[\S5.5]{Kat16} and \cite[\S7.3]{Oxl11}).\qedhere

\end{enumerate}
\end{rmk}


\begin{lem}[Lift of direct sums to realizations]\label{72}
Let $W\subseteq\KK^E$ be a realization of a matroid $\M$.
Suppose that $\M=\M_1\oplus\M_2$ decomposes with underlying partition $E=E_1\sqcup E_2$.
Then $W=W_1\oplus W_2$ where $W_i:=\M/E_j\subseteq\KK^{E_i}$ realizes $\M_i=\M\vert_{E_i}$ for $\set{i,j}=\set{1,2}$.
\end{lem}

\begin{proof}
By definition (see Definition~\ref{33}.\eqref{33c} and \eqref{33e}), 
\[
W_1\oplus W_2\into W\into W\vert_{E_1}\oplus W\vert_{E_2},\quad W_i\into W\vert_{E_i},\quad i=1,2.
\]
By the direct sum hypothesis, $W_i$ and $W\vert_{E_i}$ realize the same matroid (see \eqref{209}, \eqref{153} and \eqref{154})
\[
\M/E_j=\M\vert_{E_i}=\M_i,\quad\set{i,j}=\set{1,2}.
\]
Thus, $\dim W_i=\dim(W\vert_{E_i})$ for $i=1,2$ (see \eqref{202}) and the claim follows.
\end{proof}


\begin{exa}[Realizations of uniform matroids]\label{119}
Let $W\subseteq\KK^E$ be the row span of a matrix $A\in\KK^{r\times n}$ (see Remark~\ref{65}).
If $A$ is generic in the sense that all maximal minors of $A$ are nonzero, then $W$ realizes the uniform matroid $\U_{r,n}$ (see Example~\ref{118}).
\end{exa}

\subsection{Graphic matroids}\label{61}

Configurations arising from graphs are the most prominent examples for our results.
In this subsection, we review this construction and discuss important examples such as prism, wheel and whirl matroids.

A \emph{graph} $G=(V,E)$ is a pair of finite sets $V$ of \emph{vertices} and $E$ of (unoriented) \emph{edges} where each edge $e\in E$ is associated with a set of one or two vertices in $V$.
This allows for multiple edges between pairs of vertices, and loops at vertices.

A graph $G$ determines a \emph{graphic matroid} $\M_G$ on the edge set $E$.
Its independent sets are the \emph{forests} and its circuits the \emph{simple cycles} in $G$.
Any graphic matroid comes from a (non-unique) \emph{connected} graph (see \cite[Prop.~1.2.9]{Oxl11}).
Unless specified otherwise, we therefore assume that $G$ is connected.
Then the bases of $\M_G$ are the \emph{spanning trees} of $G$ (see \cite[18]{Oxl11}),
\begin{equation}\label{127}
\B_{\M_G}=\T_G.
\end{equation}


\begin{rmk}[Graph and matroid connectivity]\label{203}
A \emph{vertex cut} of a graph $G=(V,E)$ is a subset of $V$ whose removal (together with all incident edges) disconnects $G$.
If $G$ has at least one pair of distinct non-adjacent vertices, then $G$ is called \emph{$k$-connected} if $k$ is the minimal size of a vertex cut.
Otherwise, $G$ is $(\abs{V}-1)$-connected by definition.
Suppose that $\abs{V}\ge3$.
Then $\M_G$ is ($2$-)connected if and only if $G$ is $2$-connected and loopless (see \cite[Prop.~4.1.7]{Oxl11}).
Provided that $\abs{E}\ge4$, $\M_G$ is $3$-connected if and only if $G$ is $3$-connected and simple (see \cite[Prop.~8.1.9]{Oxl11}).
\end{rmk}


\begin{exa}[Prism matroid as graphic matroid]\label{47}
The prism matroid (see Definition~\ref{200}) is associated with the $(2,2,2)$-theta graph in Figure~\ref{103}.
In particular it is $3$-connected as witnessed by the minimal vertex cut $\set{v_1,v_2,v_3}$ (see Remark~\ref{203}).
\begin{figure}[h]
\caption{The $(2,2,2)$-theta graph with a choice of orientation.}\label{103}
\begin{tikzpicture}[scale=1,baseline=(current bounding box.center),decoration={markings,mark=at position 0.5 with {\arrow{latex}}}]
\tikzstyle{root}=[circle,draw,inner sep=1.2pt,fill=black]
\node (0) at (0,0) [root,label=right:$v_2$] {};
\node (1) at (-1,-1) [root,label=left:$v_4$] {};
\node (2) at (-1,1) [root,label=left:$v_1$] {};
\node (3) at (1,1) [root,label=right:$v_5$] {};
\node (4) at (1,-1) [root,label=right:$v_3$] {};
\draw[postaction={decorate}] (2) -- (1) node[midway,left]{$e_1$};
\draw[postaction={decorate}] (2) -- (3) node[midway,above]{$e_2$};
\draw[postaction={decorate}] (4) -- (3) node[midway,right]{$e_6$};
\draw[postaction={decorate}] (4) -- (1) node[midway,below]{$e_5$};
\draw[postaction={decorate}] (0) -- (1) node[midway,right]{$e_3$};
\draw[postaction={decorate}] (0) -- (3) node[midway,left]{$e_4$};
\end{tikzpicture}
\end{figure}
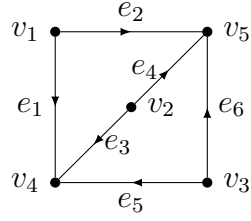
\end{exa}


Graphic matroids have realizations derived from the edge-vertex incidence matrix of the graph (see \cite[\S2]{BEK06}).
A choice of orientation on the edge set $E$ turns the graph $G$ into a CW-complex.
This gives rise to an exact sequence
\begin{equation}\label{41}
\xymatrix@R=5pt{
0\ar[r] & H_1\ar[r] & \KK^E\ar[r]^-\delta & \KK^V\ar[r]^-\sigma & H_0\ar[r] & 0\\
&& (s\to t)\ar@{}[u]|*[@]{\in}\ar@{|->}[r] & t-s\ar@{}[u]|*[@]{\in} & \KK\ar@{}[u]|*[@]{\cong}\\
&&& v\ar@{|->}[r] & 1\ar@{}[u]|*[@]{\in}
}
\end{equation}
where $H_\bullet:=H_\bullet(G,\KK)$ denotes the graph homology of $G$ over $\KK$.
The dual exact sequence
\begin{equation}\label{184}
\xymatrix{
0 & H^1\ar[l] & \KK^{E^\vee}\ar[l] && \KK^{V^\vee}\ar[ll]_-{\delta^\vee} & H^0\ar[l] & 0\ar[l]
}
\end{equation}
involves the graph cohomology $H^\bullet:=H^\bullet(G,\KK)$ of $G$ over $\KK$.


\begin{dfn}[Graph configurations]\label{43}
We call the image
\[
\xymatrix{
\KK^{E^\vee}\supseteq W_G:=\delta^\vee(\KK^{V^\vee})&\ar[l]_-{\ol{\delta^\vee}}^-\cong\ker(\sigma)^\vee
}
\]
of $\delta^\vee$ the \emph{graph configuration} of the graph $G$ over $\KK$.
Note that it is independent of the orientation chosen to define $\delta$ in \eqref{41}.
\end{dfn}


For any $S\subseteq E$, the sequence \eqref{41} induces a short exact sequence
\[
\xymatrix{
0\ar[r] & H_1\cap\KK^S\ar[r] & \KK^S\ar[r] & W_G^\vee.
}
\]
By definition of $\M_G$ and $\M_{W_G}$ (see Definition~\ref{60}) and since $H_1$ is generated by indicator vectors of (simple) cycles, we have
\[
S\in\I_{\M_G}\iff H_1\cap\KK^S=0\iff S\in\I_{\M_{W_G}},
\]
which implies that
\[
\M_G=\M_{W_G}.
\]
The configuration $W_G$ is totally unimodular if $\ch\KK=0$ (see \cite[Lem.\ 5.1.4]{Oxl11}) which makes $\M_G$ a regular matroid.
By construction, $W_G^\perp=H_1\subseteq\KK^E$ realizes the dual matroid $\M_G^\perp$ (see Definition~\ref{33}.\eqref{33a}).


\begin{exa}[Configuration of the $(2,2,2)$-theta graph]\label{216}
With the orientation of the $(2,2,2)$-theta graph $G$ depicted in Figure~\ref{103}, the map $\delta^\vee$ in \eqref{184} is represented by the transpose of the matrix
\[
\begin{pmatrix}
1 & 1 & 0 & 0 & 0 & 0 \\
0 & 0 & 1 & 1 & 0 & 0 \\
0 & 0 & 0 & 0 & 1 & 1 \\
-1 & 0 & -1 & 0 & -1 & 0
\end{pmatrix}.
\]
Its rows generate the graph configuration $W_G$ realizing the prism matroid (see Example~\ref{47}).
\end{exa}


\begin{lem}[Characterization of the prism matroid]\label{42}
Let $\M$ be a connected matroid on $E=\set{e_1,\dots,e_6}$ with $\abs{E}=6$ whose handle partition 
\[
E=H_1\sqcup H_2\sqcup H_3,\quad H_1=\set{e_1,e_2},\quad H_2=\set{e_3,e_4},\quad H_3=\set{e_5,e_6},
\]
is made of $3$ maximal $2$-handles (see Example~\ref{201} and Lemma~\ref{17}).
Then $\M$ is the prism matroid (see Definition~\ref{200}).
Up to scaling $E$, it has the unique realization $W\subseteq\KK^E$ with basis
\[
w^1:=e_1+e_2,\quad w^2:=e_3+e_4,\quad w^3:=e_5+e_6,\quad w^4:=e_1+e_3+e_5,
\]
the graph configuration of the $(2,2,2)$-theta graph (see Example~\ref{216}).
\end{lem}

\begin{proof}
Each circuit $C\in\C_\M$ is a (non-empty) disjoint union of $H_1,H_2,H_3$ (see Definition~\ref{126}).
By Lemma~\ref{16}.\eqref{16d}, no $H_i$ is a circuit, but each $H_i$ is properly contained in one.
By hypothesis, $E$ is not a maximal handle and hence $E\not\in\C_\M$.
Up to renumbering $H_1,H_2,H_3$, this yields circuits $H_2\sqcup H_3$ and $H_1\sqcup H_3$.
By the strong circuit elimination axiom (see \cite[Prop.~1.4.12]{Oxl11}), there is a third circuit $H_1\sqcup H_2$.
Then
\[
\C_\M=\set{C_1,C_2,C_3},\quad C_1=H_2\sqcup H_3,\quad C_2=H_1\sqcup H_3,\quad C_3=H_1\sqcup H_2,
\]
identifies with the circuits of the prism matroid.
It follows that $\M$ must be the prism matroid.

Let $W\subseteq\KK^E$ be any realization of $\M$.
Then $\dim W=\rk\M=4$ (see \eqref{202} and \eqref{127}).
Pick a basis $w=(w^1,\dots,w^4)$ of $W$ and denote by $A=(w^i_j)_{i,j}$ the coefficient matrix (see Remark~\ref{65}).
We may assume that columns $2,4,6,5$ of $A$ form an identity matrix.
Since $C_1$ and $C_2$ are circuits, $w^1_3=0\ne w^2_3$ and $w^2_1=0\ne w^1_1$.
Thus,
\[
A=
\begin{pmatrix}
* & 1 & 0 & 0 & 0 & 0 \\
0 & 0 & * & 1 & 0 & 0 \\
* & 0 & * & 0 & 0 & 1 \\
* & 0 & * & 0 & 1 & 0
\end{pmatrix}.
\]
Since $C_3$ is a circuit, suitably replacing $w^3,w^4\in\ideal{w^3,w^4}$, reordering $H_3$ and scaling $e_1,e_3$ makes
\[
A=
\begin{pmatrix}
* & 1 & 0 & 0 & 0 & 0 \\
0 & 0 & * & 1 & 0 & 0 \\
0 & 0 & 0 & 0 & * & 1 \\
1 & 0 & 1 & 0 & 1 & 0
\end{pmatrix},
\]
where $w^1_1,w^2_3,w^3_5\ne0$.
Now suitably scaling first $w^1,w^2,w^3$ and then $e_2,e_4,e_6$ makes
\[
A=
\begin{pmatrix}
1 & 1 & 0 & 0 & 0 & 0 \\
0 & 0 & 1 & 1 & 0 & 0 \\
0 & 0 & 0 & 0 & 1 & 1 \\
1 & 0 & 1 & 0 & 1 & 0
\end{pmatrix}.
\]
Now $w=(w^1,\dots,w^4)$ is the desired basis.
\end{proof}


The following classes of matroids play a distinguished role in connection with $3$-connectedness.


\begin{exa}[Wheels and whirls]\label{145}
For $n\ge2$, the \emph{wheel graph} $W_n$ in Figure~\ref{147} is obtained from an $n$-cycle, the \enquote{rim}, by adding an additional vertex and edges, the \enquote{spokes}, joining it to each vertex in the rim.
There is a partition of the set of edges 
\[
E=S\sqcup R,\quad S=\set{s_1,\dots,s_n},\quad R=\set{r_1,\dots,r_n},
\]
into the set $S$ of spokes and the set $R$ of edges in the rim.
The symmetry suggests to use a cyclic index set $\ZZ_n:=\ZZ/n\ZZ=\set{1,\dots,n}$.


\begin{figure}[h]
\begin{tikzpicture}[scale=1,baseline=(current bounding box.center)]
\tikzstyle{root}=[circle,draw,inner sep=1.2pt,fill=black]
\draw (0,0) -- (0:3) node [root] {} -- (30:3) node [root] {}
(5:2.2) node {$s_n$}
(15:3.1) node {$r_n$};
\draw \foreach \x in {1,...,9} {
(0,0) -- (\x*30:3) -- (\x*30+30:3) node [root] {} 
(\x*30+6:2.2) node {$s_\x$}
(\x*30+15:3.1) node {$r_\x$}
};
\draw (0,0) node [root] {} -- (300:3) node [root] {};
\draw [line width=1pt,line cap=round,dash pattern={on 0pt off 4pt},dash phase=2pt,domain=300:360] plot ({3*cos(\x)}, {3*sin(\x)});
\end{tikzpicture}
\caption{The wheel graph $W_n$.}\label{147}
\end{figure}
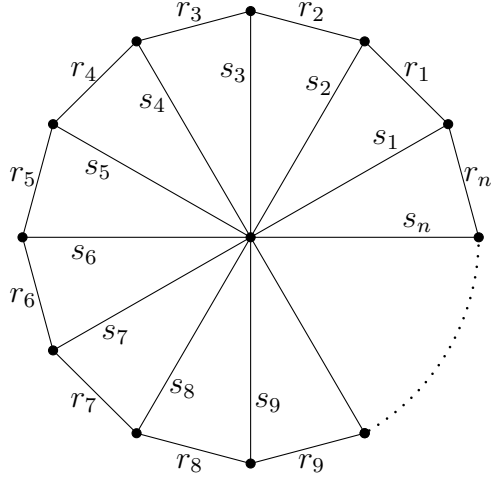


For $n\ge 3$, the \emph{wheel matroid} is the graphic matroid $\W_n:=\M_{W_n}$ on $E$.
For $n\ge 2$, the \emph{whirl matroid} is the (non-graphic) matroid on $E$ obtained from $\M_{W_n}$ by \emph{relaxation} of the rim $R$, that is,
\[
\B_{\W^n}:=\B_{\M_{W_n}}\sqcup\set{R}.
\]
In terms of circuits, this means that
\[
\C_{\W^n}=\C_{\M_{W_n}}\setminus R\sqcup\set{\set{s}\sqcup R\mid s\in S}.
\]
The matroids $\W_n$ and $\W^n$ are $3$-connected (see \cite[Exa.~8.4.3]{Oxl11}) of rank
\[
\rk\W_n=n=\rk\W^n.
\]

For each $i\in\ZZ_n$, $\set{s_i,r_i,s_{i+1}}$ is a triangle and $\set{r_i,r_{i+1},s_{i+1}}$ a triad.
Conversely, this property enforces $\M\in\set{\W_n,\W^n}$ for any connected matroid $\M$ on $E$ (see \cite[(6.1)]{Sey80}).
\end{exa}


We describe all realizations of wheels and whirls up to equivalence.
In particular, we recover the well-known fact that whirls are not binary.


\begin{lem}[Realizations of wheels and whirls]\label{134}
Let $W\subseteq\KK^E$ be any realization of $\M\in\set{\W_n,\W^n}$.
Up to scaling $E=S\sqcup R$, $W$ has a basis
\begin{equation}\label{136}
w^1=s_1+r_1-t\cdot r_n,\quad w^i=s_i+r_i-r_{i-1},\quad i=2,\dots,n,
\end{equation}
where $t=1$ if $\M=\W_n$, and $t\in\KK\setminus\set{0,1}$ if $\M=\W^n$.
\end{lem}

\begin{proof}
Since $S\in\B_\M$, we may assume that the coefficients of $s_j$ in $w^i$ form an identity matrix, that is, $w^i_{s_j}=\delta_{i,j}$.
The triangle $\set{s_j,r_j,s_{j+1}}$ then forces $w^j_{r_j},w^{j+1}_{r_j}\ne0$ and $w^i_{r_j}=0$ for all $i\in\ZZ_n\setminus\set{j,j+1}$.
Suitably scaling $r_1,w^2,r_2,w^3,\dots,r_{n-1},w^n,s_1,\dots,s_n$ successively yields \eqref{136}.
The claim on $t$ follows from $R\in\C_{\W_n}$ and $R\in\B_{W^n}$, respectively.
\end{proof}

\section{Configuration polynomials and forms}\label{37}

In this section, we develop Bloch's strategy of putting graph polynomials into the context of configuration polynomials and configuration forms.
We lay the foundation for an inductive proof of our main result using a handle decomposition.
In the process, we generalize some known results on graph polynomials to configuration polynomials.

\subsection{Configuration polynomials}\label{24}

To prepare the definition of configuration polynomials we introduce some notation.

Let $W\subseteq\KK^E$ be a configuration, and let $S\subseteq E$ be any subset.
Compose the associated inclusion map with $\pi_S$ to a map (see \eqref{206})
\begin{equation}\label{62}
\xymatrix{
\alpha_{W,S}\colon W\ar@{^(->}[r] & \KK^E\ar[r]^-{\pi_S} & \KK^S.
}
\end{equation}
Fix an isomorphism
\begin{equation}\label{156}
\xymatrix{
c_W\colon\KK\ar[r]_-\cong & \bigwedge^{\dim W}W
}
\end{equation}
and set $c_0:=\id_\KK$ for the zero vector space.
Any basis of $W$ gives rise to such an isomorphism and any two such isomorphisms differ by a nonzero multiple $c\in \KK^*$.
Up to sign or ordering $E$, we identify
\begin{equation}\label{152}
\bigwedge^\abs{S}\KK^S=\KK,\quad\mathop{\wedge}\limits_{s\in S}s\mapsto 1,
\end{equation}
as based vector spaces.
Suppose that $\abs{S}=\dim W$.
Then the determinant
\begin{equation}\label{35}
\xymatrix@C=4em{
\det\alpha_{W,S}\colon\KK\ar[r]^-{c_W}_-\cong &
\bigwedge^\abs{S}W\ar[r]^-{\bigwedge^\abs{S}\alpha_{W,S}} &
\bigwedge^\abs{S}\KK^S=\KK
}
\end{equation}
is defined up to sign.
Its square
\begin{equation}\label{39}
c_{W,S}:=(\det\alpha_{W,S})^2\in\KK
\end{equation}
is defined up to a factor $c^2$ for some $c\in\KK^*$ independent of $S$.
Note that $\det\alpha_{0,\emptyset}=\id_\KK$ and hence $c_{0,\emptyset}=1$.
By definition (see \eqref{198}),
\begin{equation}\label{205}
c_{W,S}\ne0\iff S\in\B_{\M_W}.
\end{equation}


\begin{rmk}[Compatibility of coefficients with restriction]\label{157}
Let $W\subseteq\KK^E$ be a configuration, and let $S\subseteq F\subseteq E$ with $\abs{S}=\dim W$. 
Then the maps \eqref{62} for $W$ and $W\vert_F$ form a commutative diagram
\[
\xymatrix{
W\ar@/^2pc/[rr]^-{\alpha_{W,S}}\ar[d]_{\pi_F\vert_W}^-\cong\ar@{^(->}[r] & \KK^E\ar[d]_{\pi_F}\ar[r]^-{\pi_S} & \KK^S\ar@{=}[d]\\
W\vert_F\ar@/_2pc/[rr]_-{\alpha_{W\vert_F,S}}\ar@{^(->}[r] & \KK^F\ar[r]^-{\pi_S} & \KK^S
}
\]
and hence $c_{W,S}=c^2\cdot c_{W\vert_F,S}$ for some $c\in\KK^*$ independent of $S$.
\end{rmk}


Consider the dual basis $E^\vee=(e^\vee)_{e\in E}$ of $E$ as coordinates on $\KK^E$,
\begin{equation}\label{159}
x_e:=e^\vee,\quad \partial_e:=\frac\partial{\partial x_e},\quad e\in E.
\end{equation}
Given an enumeration of $E=\set{e_1,\dots,e_n}$, we write 
\[
x_i:=x_{e_i},\quad\partial_i:=\partial_{e_i},\quad i=1,\dots,n.
\]
For any subset $S\subseteq E$, we set
\begin{equation}\label{214}
x_S:=(x_e)_{e\in S},\quad x^S:=\prod_{e\in S}x_e,\quad x:=x_E.
\end{equation}


\begin{dfn}[Configuration polynomials]\label{48}
Let $W\subseteq\KK^E$ be a realization of a matroid $\M$.
Then the \emph{configuration polynomial} of $W$ is (see \eqref{39})
\[
\psi_W:=\sum_{B\in\B_\M}c_{W,B}\cdot x^B\in\KK[x].
\]
\end{dfn}


\begin{rmk}[Well-definedness of configuration polynomials]\label{150}
Any two isomorphisms $c_W$ (see \eqref{156}) differ by a nonzero multiple $c\in \KK^*$.
Using the isomorphism $c\cdot c_W$ in place of $c_W$ replaces $\psi_W$ by $c^2\cdot\psi_W$.  
In other words, $\psi_W$ is well-defined up to a  nonzero constant square factor.
Whenever $\psi_W$ occurs in a formula, we mean that the formula holds true for a suitable choice of such a factor.
\end{rmk}


\begin{rmk}[Equivalence of configuration polynomials]\label{160}
Dividing $e\in E$ by $c\in\KK^*$ multiplies both $x_e=e^\vee$ (see Remark~\ref{68}) and the identifications \eqref{152} with $e\in S$ by $c$.
For each $e\in B\in\B_\M$, this multiplies $c_{W,B}$ by $c^2$ and $x^B$ by $c$.
This is equivalent to substituting $c^3\cdot x_e$ for $x_e$ in $\psi_W$.
Scaling $E$ thus results in scaling $x$ in $\psi_W$.

However, dropping the equality \eqref{159} and scaling $e\in E$ for fixed $x_e$ replaces $W$ in $\psi_W$ by a projectively equivalent realization (see \cite[\nopp \S 6.3]{Oxl11}).
If $\M$ is binary, then all realizations of $\M$ over $\KK$ are projectively equivalent (see \cite[Prop.~6.6.5]{Oxl11}).
The corresponding configuration polynomials are geometrically equivalent in this case.
In general, however, there are geometrically different configuration polynomials for fixed $\M$ and $\KK$ (see Example~\ref{167}).
\end{rmk}


\begin{rmk}[Degree of configuration polynomials]\label{6}
Let $W\subseteq\KK^E$ be a realization of a matroid $\M$.
Then (see \eqref{202} and \eqref{205})
\[
\deg\psi_W=\rk\M=\dim W.
\]
In particular, $\psi_W\ne0$, and $\psi_W=1$ if and only if $\rk\M=0$.
By definition, $\psi_W$ is independent of (divided by) $x_e$ if and only if $e\in E$ is a (co)loop in $\M$.
\end{rmk}


\begin{rmk}[Matroid polynomials and regularity]\label{149}
For any matroid $\M$, not necessarily realizable, there is a \emph{matroid (basis) polynomial}
\[
\psi_{\M}:=\sum_{B\in\B_\M}x^B.
\]
If $\M$ is regular, then $\psi_W=\psi_\M$ for any totally unimodular realization $W$ of $\M$ over $\KK$.
Conversely, this equality for some realization $W$ over $\KK$ with $\ch\KK=0$ establishes regularity of $\M$.
For regular $\M$, all configuration polynomials over $\KK$ are geometrically equivalent (see Remark~\ref{160}).
In general, however, $\psi_W$ and $\psi_\M$ are geometrically different (see Example~\ref{55}).
\end{rmk}


\begin{exa}[Configuration polynomials of uniform matroids]\label{122}
Let $W\subseteq\KK^E$ be a realization of a uniform  matroid $\M=\U_{r,n}$ (see Example~\ref{119}).

\begin{asparaenum}[(a)]

\item\label{122a} Suppose that $\M=\U_{n,n}$ is a free matroid.
Then $E\in\B_\M$ and 
\[
\psi_W=x^E
\]
is the elementary symmetric polynomial of degree $n$ in $n$ variables.

\item\label{122b} Suppose that $\M=\U_{n-1,n}$ is a circuit.
Then $E\in\C_\M$ and by Remark~\ref{157} and \eqref{122a}
\[
\psi_{W}=\sum_{e\in E}\psi_{W\setminus e},\quad \psi_{W\setminus e}=x^{E\setminus\set{e}}.
\]
A priori, substituting $x^{E\setminus\set{e}}$ for $\psi_{W\setminus e}$ in $\psi_{W}$ is invalid (see Remark~\ref{150}).
However, this can be achieved as follows:
Ordering $E=\set{e_1,\dots,e_n}$, $W$ has a basis $w^i=e_i+c_i\cdot e_n$ with $c_i\in\KK^*$ where $i=1,\dots,n-1$.
Scaling first $w^1,\dots,w^{n-1}$ and then $e_1,\dots,e_{n-1}$ makes $c_1=\dots=c_{n-1}=1$.
This turns $\psi_W$ into
\[
\psi_W=\sum_{e\in E}x^{E\setminus\set{e}},
\]
the elementary symmetric polynomial of degree $n-1$ in $n$ variables.

\item\label{122c} If $\M=\U_{n-2,n}$, then $\M$ has $n\choose n-2$ bases, and $\psi_W$ has $n\choose n-2$ monomials whose coefficients depend on the choice of $W$.
For instance, the row span $W$ of the matrix
\[
\begin{pmatrix}
1 & 0 & 1 & 1\\ 0 & 1 & 1 & -1
\end{pmatrix}
\]
realizes $\U_{2,4}$ and
\[
\psi_W=x_1x_2+x_1x_3+x_1x_4+x_2x_3+x_2x_4+4x_3x_4.
\]
Realizations of $\U_{2,n}$ are treated in Example~\ref{148}.\qedhere

\end{asparaenum}
\end{exa}


In the following, we put matroid connectivity in correspondence with irreducibility of configuration polynomials.


\begin{prp}[Connectedness and irreducibility]\label{4}
Let $\M$ be a matroid of rank $\rk\M\ge1$ with realization $W\subseteq\KK^E$.
Then $\M$ is connected if and only if $\M$ has no loops and $\psi_W$ is irreducible.
In particular, if $\M=\bigoplus_{i=1}^n\M_i$ with connected components $\M_i$ and induced decomposition $W=\bigoplus_{i=1}^nW_i$ (see Lemma~\ref{72}), then $\psi_W=\prod_{i=1}^n\psi_{W_i}$ where $\psi_{W_i}$ is irreducible if $\rk\M_i\ge1$, and $\psi_{W_i}=1$ otherwise.
\end{prp}

\begin{proof}
First suppose that $\M=\M_1\oplus\M_2$ is disconnected with underlying proper partition $E=E_1\sqcup E_2$.
By Lemma~\ref{72}, $W=W_1\oplus W_2$ where $W_i\subseteq\KK^{E_i}$ realizes $\M_i$.
Then $\alpha_{W,B}=\alpha_{W_1,B_1}\oplus\alpha_{W_2,B_2}$ and hence $c_{W,B}=c_{W_1,B_1}\cdot c_{W_2,B_2}$ for all $B=B_1\sqcup B_2\in\B_\M$ where $B_i\in\B_{\M_i}$ for $i=1,2$ (see \eqref{209}).
It follows that $\psi_W=\psi_{W_1}\cdot\psi_{W_2}$.
This factorization is proper if $\M$ and hence each $\M_i$ has no loops (see Remark~\ref{6}).
Thus, $\psi_W$ is reducible in this case.

Suppose now that $\psi_W$ is reducible.
Then
\[
\psi_W=\psi_1\cdot\psi_2
\]
with $\psi_i$ homogeneous non-constant for $i=1,2$.
Since $\psi_W$ is a linear combination of square-free monomials (see Definition~\ref{48}), this yields a proper partition $E=E_1\sqcup E_2$ such that $\psi_i\in\KK[x_{E_i}]$ for $i=1,2$.
Set
\begin{equation}\label{211}
\M_i:=\M\vert_{E_i},\quad i=1,2.
\end{equation}

Each basis $B\in\B_\M$ indexes a monomial $x^B$ in $\psi_W$ (see \eqref{205}).
Set $B_i:=B\cap E_i\in\I_{\M_i}$ for $i=1,2$ (see \eqref{153}).
Then $x^B=x^{B_1}\cdot x^{B_2}$ where $x^{B_i}$ is a monomial in $\psi_i$ for $i=1,2$.
By homogeneity of $\psi_i$, $B_i\in\B_{\M_i}$ for $i=1,2$ and hence $B=B_1\sqcup B_2\in\B_{\M_1\oplus\M_2}$ (see \eqref{209}).
It follows that $\B_\M\subseteq\B_{\M_1\oplus\M_2}$.

Conversely, let $B=B_1\sqcup B_2\in\B_{\M_1\oplus\M_2}$ where $B_i\in\B_{\M_i}$ for $i=1,2$.
Then $B_i=B_i'\cap E_i$ for some $B_i'\in\B_\M$ for $i=1,2$ (see \eqref{153} and \eqref{211}).
As above, $x^{B_i}$ is a monomial in $\psi_i$ for $i=1,2$.
Then $x^B=x^{B_1}\cdot x^{B_2}$ is a monomial in $\psi_W$ and hence $B\in\B_\M$ (see \eqref{205}).
It follows that $\B_\M\supseteq\B_{\M_1\oplus\M_2}$ as well.

So $\M=\M_1\oplus \M_2$ is a proper decomposition and $\M$ is disconnected.

This proves the equivalence and the particular claims follow.
\end{proof}


We use the following well-known fact from linear algebra.


\begin{rmk}[Determinant formula]\label{56}
Consider a short exact sequence of finite dimensional $\KK$-vector spaces
\[
\xymatrix{
0\ar[r] & W\ar[r] & V\ar[r] & U\ar[r] & 0.
}
\]
Abbreviate $\bigwedge V:=\bigwedge^{\dim V}V$.
There is a unique isomorphism
\[
\bigwedge W\otimes\bigwedge U=\bigwedge V
\]
that fits into a commutative diagram of canonical maps
\[
\xymatrix{
\bigwedge W\otimes \bigwedge^{\dim U}V\ar[d]\ar[r] & \bigwedge^{\dim W} V\otimes \bigwedge^{\dim U}V\ar[d]\\
\bigwedge W\otimes \bigwedge U\ar@{=}[r] & \bigwedge V.
}
\]
Tensored with
\[
(\bigwedge U)^\vee=\bigwedge(U^\vee),\quad
(\bigwedge W)^\vee=\bigwedge(W^\vee),
\]
respectively, it induces identifications
\[
\bigwedge W=\bigwedge V\otimes\bigwedge U^\vee,\quad
\bigwedge U=\bigwedge W^\vee\otimes\bigwedge V.
\]
Consider a commutative diagram of finite dimensional $\KK$-vector spaces with short exact rows
\[
\xymatrix{
0\ar[r] & W\ar[d]_-\alpha^-\cong\ar[r] & V\ar[d]_-\gamma^-\cong\ar[r] & U\ar[r] & 0\\
0 & U'\ar[l] & V'\ar[l] & W'\ar[u]^-\beta_-\cong\ar[l] & 0.\ar[l]
}
\]
Then the above identifications for both rows fit into a commutative diagram
\[
\xymatrix{
\bigwedge W\ar[d]_-{\bigwedge\alpha}^-\cong\ar@{=}[r] & \bigwedge W\otimes\bigwedge U\otimes\bigwedge U^\vee\ar[d]_-{\bigwedge\alpha\otimes\bigwedge\beta^{-1}\otimes\bigwedge\beta^\vee}^-\cong\ar@{=}[r] & \bigwedge V\otimes\bigwedge U^\vee\ar[d]_-{\bigwedge\gamma\otimes\bigwedge\beta^\vee}^-\cong\\
\bigwedge U'\ar@{=}[r] & \bigwedge U'\otimes\bigwedge W'\otimes\bigwedge W'^\vee\ar@{=}[r] & \bigwedge V'\otimes\bigwedge W'^\vee.
}
\]
\end{rmk}


The following result of Bloch, Esnault and Kreimer describes the behavior of configuration polynomials under duality (see \cite[Prop.~1.6]{BEK06}).


\begin{prp}[Dual configuration polynomials]\label{30}
Let $W\subseteq\KK^E$ be a realization of a matroid $\M$.
For a suitable choice of $c_W$ (see \eqref{156}),
\[
\det\alpha_{W^\perp,S^\perp}=\det\alpha_{W,S}
\]
for all $S\subseteq E$ of size $\abs{S}=\rk\M$.
In particular,
\[
\psi_{W^\perp}=x^{E^\vee}\cdot\psi_W((x^{-1}_{e^\vee})_{e\in E}).
\]
\end{prp}


\begin{proof}
Let $S\subseteq E$ be of size $\abs{S}=\rk\M$.
Then $S\in\B_\M$ if and only if $S^\perp\in\B_{\M^\perp}$ (see Remark~\ref{150}).
We may assume that this is the case as otherwise both determinants are zero.
Then there is a commutative diagram with exact rows
\[
\xymatrix{
0\ar[r] & W\ar[r]\ar[d]_-{\alpha_{W,S}}^-\cong & \KK^E\ar[r]\ar[d]_-\nu^-\cong & \KK^E/W\ar[r] & 0\\
0 & \KK^S\ar[l] & \KK^{E^\vee}\ar[l]_-{\pi_S\circ\nu^{-1}} & \KK^{S^\perp}\ar[l]_-{\pi_{S^\perp}^\vee}\ar[u]_-{\alpha_{W^\perp,S^\perp}^\vee}^-\cong & 0\ar[l]
}
\]
where the middle isomorphism is induced by \eqref{158}.
This yields a commutative diagram (Remark~\ref{56} and \eqref{202})
\[
\xymatrix{
\KK\ar[d]_-{c_{W}}\ar[r]_-\cong & \bigwedge^\abs{E}\KK^E\otimes_\KK\KK\ar[d]^-{\id\otimes c_{W^\perp}}\\
\bigwedge^{\rk\M}W\ar[d]_-{\bigwedge^{\rk\M}\alpha_{W,S}}\ar@{=}[r] & \bigwedge^\abs{E}\KK^E\otimes\bigwedge^{\rk\M^\perp} W^\perp
\ar[d]^-{\bigwedge^\abs{E}\nu\otimes\bigwedge^{\rk\M^\perp}\alpha_{W^\perp,S^\perp}}
\\
\bigwedge^{\rk\M}\KK^S\ar@{=}[r] & \bigwedge^\abs{E}\KK^{E^\vee}\otimes\bigwedge^{\rk\M^\perp}\KK^{S^\perp}.
}
\]
Using \eqref{152}, we may drop $\bigwedge^\abs{E}\KK^E$ and $\bigwedge^\abs{E}\KK^{E^\vee}$.
A suitable choice of $c_W$ turns the upper isomorphism into an equality.
The claim follows by definition (see \eqref{35} and Definition~\ref{48}).
\end{proof}


The coefficients of the configuration polynomial satisfy the following restriction--contraction formula.

\begin{lem}[Restriction--contraction for coefficients]\label{2}
Let $W\subseteq\KK^E$ be a realization of a matroid $\M$, and let $F\subseteq E$ be any subset.
For any basis $B\in\B_\M$, $B\cap F\in\B_{\M\vert_F}$ if and only if $B\setminus F\in\B_{\M/F}$.
In this case,
\[
c_{W,B}=c^2\cdot c_{W/F,B\setminus F}\cdot c_{W\vert_F,B\cap F}
\]
where $c\in\KK^*$ is independent of $B$.
\end{lem}

\begin{proof}
The equivalence for $B\in\B_\M$ holds by definition of matroid contraction (see \eqref{154}). 
For any such $B$, there is a commutative diagram with exact rows (see Definition~\ref{33}.\eqref{33c} and \eqref{33e})
\[
\xymatrix{
0\ar[r] & W/F\ar[r]\ar@{^(->}[d] & W\ar[r]\ar@{^(->}[d] & W\vert_F\ar[r]\ar@{^(->}[d] & 0\\
0\ar[r] & \KK^{E\setminus F}\ar[r]\ar[d] & \KK^E\ar[r]\ar[d] & \KK^F\ar[r]\ar[d] & 0\\
0\ar[r] & \KK^{B\setminus F}\ar[r] & \KK^B\ar[r] & \KK^{B\cap F}\ar[r] & 0.
}
\]
Taking exterior powers yields (see Remark~\ref{56} and \eqref{202})
\[
\xymatrix{
\KK\ar[d]_-{c_W}^-\cong\ar[r]_-\cong^-c & \KK=\KK\otimes\KK\ar[d]^-{c_{W/F}\otimes c_{W\vert_F}}_-\cong\\
\bigwedge^{\rk\M} W\ar@{=}[r]\ar[d]_-{\bigwedge^{\rk\M}\alpha_{W,B}}^-\cong & \bigwedge^{\rk(\M/F)}W/F\otimes\bigwedge^{\rk(\M\vert_F)}W\vert_F\ar[d]^-{\bigwedge^{\rk(\M/F)}\alpha_{W/F,B\setminus F}\otimes\bigwedge^{\rk(\M\vert_F)}\alpha_{W\vert_F,B\cap F}}_-\cong\\
\bigwedge^{\rk\M}\KK^B\ar@{=}[r] & \bigwedge^{\rk(\M/F)}\KK^{B\setminus F}\otimes\bigwedge^{\rk(\M\vert_F)}\KK^{B\cap F}.
}
\]
\end{proof}


The following result describes the behavior of configuration polynomials under deletion--contraction.
It is the basis for our inductive approach to Jacobian schemes of configuration polynomials.
The statement on $\partial_e\psi_W$ was proven by Patterson (see \cite[Lem.~4.4]{Pat10}).

\begin{prp}[Deletion--contraction for configuration polynomials]\label{3}
Let $W\subseteq\KK^E$ be a realization of a matroid $\M$, and let $e\in E$.
Then 
\[
\psi_W=
\begin{cases}
\psi_{W\setminus e}=\psi_{W/e} & \text{if $e$ is a loop in $\M$},\\
\psi_{W\vert_e}\cdot\psi_{W/e}=\psi_{W\vert_e}\cdot\psi_{W\setminus e} & \text{if $e$ is a coloop in $\M$},\\
\psi_{W\setminus e}
+\psi_{W\vert_e}\cdot\psi_{W/e} & \text{otherwise,}
\end{cases}\\
\]
where $\psi_{W\vert_e}=x_e$ if $e$ is not a loop in $\M$.
In particular,
\begin{align*}
\partial_e\psi_W&=
\begin{cases}
0 & \text{if $e$ is a loop in $\M$},\\
\psi_{W/e}=\psi_{W\setminus e} & \text{if $e$ is a coloop in $\M$},\\
\psi_{W/e} & \text{otherwise},
\end{cases}\\
\psi_W\vert_{x_e=0}&=
\begin{cases}
\psi_{W\setminus e}=\psi_{W/e} & \text{if $e$ is a loop in $\M$},\\
0 & \text{if $e$ is a coloop in $\M$},\\
\psi_{W\setminus e} & \text{otherwise}.
\end{cases}
\end{align*}
\end{prp}

\begin{proof}
Decompose
\begin{equation}\label{7}
\psi_W
=\sum_{e\not\in B\in\B_\M}c_{W,B}\cdot x^B
+x_e\cdot\sum_{e\in B\in\B_\M}c_{W,B}\cdot x^{B\setminus\set{e}}.
\end{equation}
The second sum in \eqref{7} is nonzero if and only if $e$ is not a loop.
Suppose that this is the case.
Then $\M\vert_e$ is free with basis $\set{e}$ and $\psi_{W\vert_e}=x_e$ by Remark~\ref{122}.\eqref{122a}.
By Lemma~\ref{2} applied to $F=\set{e}$, the second sum in \eqref{7} then equals (see \eqref{154} and Remark~\ref{150})
\[
c^2\cdot c_{W\vert_e,\set{e}}\cdot\sum_{B\in\B_{\M/e}}c_{W/e,B}\cdot x^B=\psi_{W/e}
\]
for some $c\in\KK^*$.
The first sum in \eqref{7} is nonzero if and only if $e$ is not a coloop.
By Lemma~\ref{2} applied to $F=E\setminus\set{e}$, it equals in this case (see \eqref{153} and Remark~\ref{150})
\[
c^2\cdot c_{0,\emptyset}\cdot\sum_{B\in\B_{\M\setminus e}}c_{W\setminus e,B}\cdot x^B=\psi_{W\setminus e}
\]
for some $c\in\KK^*$.
If $e$ is a (co)loop, then $W/e=W\setminus e$ (see Remark~\ref{34}.\eqref{34a}).
The claimed formulas follow.
\end{proof}


The following formula relates configuration polynomials with deletion and contraction of handles.
It is the starting point for our description of generic points of Jacobian schemes of configuration hypersurfaces in terms of handles.

\begin{cor}[Configuration polynomials and handles]\label{18}
Let $W\subseteq\KK^E$ be a realization of a connected matroid $\M$ on $E$, and let $E\ne H\in\H_\M$ be a proper handle.
Then 
\begin{align}
\label{18a}\psi_W&=\psi_{W/(E\setminus H)}\cdot\psi_{W\setminus H}+\psi_{W\vert_{H}}\cdot\psi_{W/H},\\
\label{18b}\psi_{W/(E\setminus H)}&=\sum_{h\in H}\psi_{W\vert_{H\setminus\set{h}}},\\
\label{18c}\psi_{W\vert_{H}}&=x^H,\quad
\psi_{W\vert_{H\setminus\set{h}}}=x^{H\setminus\set{h}}.
\end{align}
In particular, after suitably scaling $H$,
\begin{equation}\label{18d}
\psi_W=\sum_{h\in H}x^{H\setminus\set{h}}\cdot\psi_{W\setminus H}+x^H\cdot\psi_{W/H}.
\end{equation}
\end{cor}

\begin{proof}
By Lemma~\ref{16}.\eqref{16d}, $H\in\C_{\M/(E\setminus H)}$ and hence \eqref{18b} by Example~\ref{122}.\eqref{122b}.
By Lemma~\ref{16}.\eqref{16d} (see \eqref{153}), $\M\vert_H$ is free, and equalities~\eqref{18c} follows from Example~\ref{122}.\eqref{122a}.
Equality~\eqref{18d} follows from \eqref{18a}, \eqref{18b} and Example~\ref{122}.\eqref{122b}.
It remains to prove equality~\eqref{18a}.

We proceed by induction on $\abs{H}$.
Let $h\in H$ and set $H':=H\setminus\set{h}$.
Since $\M$ is connected, it has no (co)loops and hence
\begin{equation}\label{29}
\psi_W=\psi_{W\setminus h}+\psi_{W\vert_h}\cdot \psi_{W/h}
\end{equation}
by Proposition~\ref{3}.
If $\abs{H}=1$, then $H\in\C_{\M/(E\setminus H)}$ implies that $\rk(\M/(E\setminus h))=0$ and hence $\psi_{W/(E\setminus h)}=1$ (see Remark~\ref{6}).
Suppose now that $\abs{H}\ge2$.
By Lemma~\ref{16}.\eqref{16d} and \eqref{16a}, $\M\vert_{H'}$ is free and $H'$ consists of coloops in $\M\setminus h$.
Iterating Proposition~\ref{3} thus yields
\begin{equation}\label{22}
\psi_{W\setminus h}=\prod_{h'\in H'}\psi_{W\vert_{h'}}\cdot\psi_{W\setminus H}=\psi_{W\vert_{H'}}\cdot\psi_{W\setminus H}.
\end{equation}

By Lemma~\ref{16}.\eqref{16b}, the set $H'$ is a proper handle in the connected matroid $\M/h$.
By Lemma~\ref{16}.\eqref{16a}, $h$ is a coloop in $\M\setminus H'$ and hence
\[
W/h\setminus H'=W\setminus H'/h=W\setminus H'\setminus h=W\setminus H.
\]
by Remark~\ref{34}.\eqref{34a}.
By the induction hypothesis,
\begin{equation}\label{69}
\psi_{W/h}=\sum_{h'\in H'}\psi_{W\vert_{H'\setminus\set{h'}}}\cdot\psi_{W\setminus H}+\psi_{W\vert_{H'}}\cdot\psi_{W/H}.
\end{equation}
By Lemma~\ref{16}.\eqref{16d}, $\M\vert_H$ and $\M\vert_{H\setminus\set{h'}}$ are free.
Iterating Proposition~\ref{3} thus yields
\begin{equation}\label{73}
\psi_{W\vert_h}\cdot\psi_{W\vert_{H'}}=\psi_{W\vert_H},\quad
\psi_{W\vert_h}\cdot\psi_{W\vert_{H'\setminus\set{h'}}}=\psi_{W\vert_{H\setminus\set{h'}}}.
\end{equation}
Using equalities~\eqref{18b} and \eqref{73}, equality~\eqref{18a} is obtained by substituting \eqref{22} and \eqref{69} into \eqref{29} (see Remark~\ref{150}).
\end{proof}


The following result describes the behavior of configuration polynomials when passing to an elementary quotient.

\begin{prp}[Configuration polynomials of quotients]\label{28}
Let $W\subseteq\KK^E$ be a realization of a matroid $\M$, and let $0\ne\varphi\in W^\vee$.
Then
\[
\psi_{W_\varphi}=\sum_{\substack{S\subseteq E\\\abs{S}=\rk\M-1}}\left(\sum_{e\not\in S}\pm\tilde\varphi_e\cdot\det\alpha_{W,S\cup\set{e}}\right)^2x^S,
\]
where $\tilde\varphi=(\tilde\varphi_e)_{e\in E}\in(\KK^E)^\vee$ is any lift of $\varphi$ with a sign $\pm$ determined by a Laplace expansion.
\end{prp}

\begin{proof}
Set $V:=W^\perp$ and $V_\varphi:=W_\varphi^\perp$ and consider the commutative diagram with short exact rows and columns
\[
\xymatrix{
&&& 0\ar[d]\\
& 0\ar[d] && \KK\ar[d]\\
0\ar[r] & W_\varphi\ar[d]\ar[r] & \KK^E\ar[r]\ar@{=}[d] & V_\varphi^\vee\ar[r]\ar[d] & 0\\
0\ar[r] & W\ar[r]\ar[d]_-\varphi & \KK^E\ar[dl]^-{\tilde\varphi}\ar[r] & V^\vee\ar[r]\ar[d] & 0\\
&\KK\ar[d] && 0\\
&0.
}
\]
Dualizing and identifying the two copies of $\KK$ by the Snake Lemma yields a commutative diagram with short exact rows and columns
\begin{equation}\label{11}
\xymatrix{
&&& 0\\
& 0 && \KK\ar[u]\ar[dl]_-{\cdot\tilde\varphi}\\
0 & W_\varphi^\vee\ar[l]\ar[u] & \KK^{E^\vee}\ar[l] & V_\varphi\ar[l]\ar[u] & 0\ar[l]\\
0 & W^\vee\ar[l]\ar[u] & \KK^{E^\vee}\ar[l]\ar@{=}[u] & V\ar[l]\ar[u] & 0\ar[l]\\
& \KK\ar[u]^-{\cdot\varphi}\ar[ur]_-{\cdot\tilde\varphi} && 0\ar[u]\\
& 0.\ar[u]
}
\end{equation}
By Remark~\ref{56} and with a suitable choice of $c_V$ (see Remark~\ref{150}), the right vertical short exact sequence in \eqref{11} gives rise to a  commutative square
\[
\xymatrix{
\KK\ar[r]^-{c_{V_\varphi}}\ar@{=}[d] & \bigwedge^{\rk\M^\perp+1}V_\varphi\\
\KK\ar[r]^-{c_{V}} & \bigwedge^{\rk\M^\perp}V\ar@{=}[u]
}
\]
Let $S'\subseteq E^\vee$ with $\abs{S'}=\dim V_\varphi=\rk\M^\perp+1$ and denote (see \eqref{158})
\[
\tilde\varphi_{S'}=(\tilde\varphi_{\nu^{-1}(e)})_{e\in S'}\in\KK^{S'}.
\]
Due to \eqref{11} the maps $\alpha_{V_\varphi,S'}$ (see \eqref{62}) and
\[
\xymatrix{
\llap{$\begin{pmatrix}
\tilde\varphi_{S'} & \alpha_{V,S'}
\end{pmatrix}
\colon$}\KK\oplus V\ar[r] & \KK^{E^\vee}\ar[r]^-{\pi_{S'}} & \KK^{S'}
}
\]
agree after applying $\bigwedge^{\rk\M^\perp+1}$.
Laplace expansion thus yields
\[
\det\alpha_{V_\varphi,S'}=\sum_{e\in S'}\pm\tilde\varphi_{\nu^{-1}(e)}\cdot\det\alpha_{V,S'\setminus\set{e}}.
\]
Let $S\subseteq E$ with $\abs{S}=\dim W_\varphi=\rk\M-1$ and $S'=S^\perp$.
Then Proposition~\ref{30} yields
\[
c_{W_\varphi,S}=\left(\sum_{e\not\in S}\pm\tilde\varphi_e\cdot\det\alpha_{W,S\cup\set{e}}\right)^2.\qedhere
\]
\end{proof}

\subsection{Graph polynomials}\label{25}

We continue the discussion of graphic matroids from \S\ref{61} and consider their configuration polynomials.


\begin{dfn}[Graph polynomials]\label{76}
The \emph{(first) Kirchhoff polynomial} of a graph $G$ over $\KK$ is the polynomial
\[
\psi_G:=\sum_{T\in\T_G} x^T\in\KK[x].
\]
Replacing $x^T$ by $x^{E\setminus T}$ defines the \emph{(first) Symanzik polynomial} $\psi_G^\perp$ of a graph $G$ over $\KK$.
We refer to $\psi_G$ and $\psi_G^\perp$ as \emph{(first) graph polynomials}.
\end{dfn}


By \eqref{127}, we have $\psi_G=\psi_W$ for any totally unimodular realization $W$ of $\M_G$.
In particular, this yields the following result of Bloch, Esnault and Kreimer (see \cite[Prop.~2.2]{BEK06} and Proposition~\ref{30}).


\begin{prp}[Graph polynomials as configuration polynomials]\label{213}
The graph polynomials
\[
\psi_G=\psi_{W_G},\quad\psi_G^\perp=\psi_{W_G^\perp},
\]
are the configuration polynomials of the graph configuration and of its dual (see Definition~\ref{43}).\qed
\end{prp}


\begin{exa}[Graph polynomial of the prism]\label{128}
For the unique realization $W=W_G$ of the prism matroid (see Lemma~\ref{42}),
\begin{align*}
\psi_W=\psi_G&=x_1x_2(x_3+x_4)(x_5+x_6)\\
&+x_3x_4(x_1+x_2)(x_5+x_6)\\
&+x_5x_6(x_1+x_2)(x_3+x_4)
\end{align*}
is the Kirchhoff polynomial of the $(2,2,2)$-theta graph $G$ (see Figure~\ref{103}).
\end{exa}


Let $G=(E,V)$ be a graph.
A \emph{$2$-forest} in $G$ is an acyclic subgraph $T$ of $G$ with $\abs{V}-2$ edges.
Any such $T=\set{T_1,T_2}$ has $2$ connected components $T_1$ and $T_2$.
We denote by $\T^2_G$ the set of all $2$-forests in $G$.


\begin{dfn}[Second graph polynomials]\label{70}
The \emph{second Kirchhoff polynomial} of a graph $G$ over $\KK$ is the polynomial
\[
\psi_G(p):=\sum_{\set{T_1,T_2}\in\T^2_G}m_{T_1}(p)^2\cdot x^{T_1\sqcup T_2}\in\KK[x],\quad m_{T_i}(p):=\sum_{v\in T_i}p_v,
\]
depending on a \emph{momentum} $0\ne p\in\ker\sigma$ for $G$ over $\KK$ (see \eqref{41}).
Note that
\[
m_{T_1}(p)=\sum_{v\in T_1}p_v=-\sum_{v\in T_2}p_v=-m_{T_2}(p),
\]
and hence, the coefficient $m_{T_1}(p)^2\in\KK$ of $\psi_G(p)$ is well-defined.

Replacing the $2$-forests $T_1\sqcup T_2$ by \emph{cut sets} $E\setminus(T_1\sqcup T_2)$ defines the \emph{second Symanzik polynomial} $\psi_G^\perp(p)$ of a graph $G$ over $\KK$ (see \cite[Def.~3.6]{Pat10}).
We refer to $\psi_G(p)$ and $\psi_G^\perp(p)$ as \emph{second graph polynomials}.
\end{dfn}


The following reformulation of a result of Patterson realizes second graph polynomials as configuration polynomials of a (dual) elementary quotient (see \cite[Prop.~3.3]{Pat10} and Proposition~\ref{30}).
Patterson's proof makes the general formula in Proposition~\ref{28} explicit in case of graph configurations (see \cite[Lem.~3.4]{Pat10}).

\begin{prp}[Second graph polynomials as configuration polynomials]\label{78}
The second graph polynomials
\[
\psi_G(p)=\psi_{(W_G)_p},\quad\psi_G^\perp(p)=\psi_{((W_G)_p)^\perp},
\]
are the configuration polynomials of the quotient of the graph configuration by a momentum and of its dual (see Definitions~\ref{33}.\eqref{33a} and \eqref{33b} and \ref{43}).\qed
\end{prp}

\subsection{Configuration forms}\label{59}

The configuration form yields an equivalent definition of the configuration polynomial as a determinant of a symmetric matrix with linear entries.
Its second degeneracy locus turns out to be the non-smooth locus of the hypersurface defined by the corresponding configuration polynomial.


\begin{dfn}[Configuration forms]\label{64}
Let $\mu_\KK$ denote the multiplication map of $\KK$.
Consider the generic diagonal bilinear form on $\KK^E$,
\[
Q_{\KK^E}:=\sum_{e\in E}x_e\cdot\mu_\KK\circ (e^\vee\times e^\vee)\colon \KK^E\times\KK^E\to\KK[x].
\]
Let $W\subseteq\KK^E$ be a configuration.
Then the \emph{configuration (bilinear) form} of $W$ is the restriction of $Q_{\KK^E}$ to $W$,
\[
Q_W:=Q_{\KK^E}\vert_{W\times W}\colon W\times W\to\KK[x].
\]
Alternatively, it can be seen as the composition of canonical maps
\begin{equation}\label{163}
\xymatrix{
Q_W\colon W[x]\ar[r] & \KK^E[x]\ar[r]^-{Q_{\KK^E}} & \KK^{E^\vee}[x]\ar[r] & W^\vee[x],
}
\end{equation}
where $-[x]$ means $-\otimes\KK[x]$.
For $k=0,\dots,r:=\dim W$, it defines a map 
\[
\bigwedge^{r-k}W\otimes\bigwedge^{r-k}W\otimes\KK[x]\to\KK[x].
\]
Its image is the $k$th Fitting ideal $\Fitt_k\coker Q_W$ (see \cite[\S 20.2]{Eis95}) and defines the \emph{$k-1$st degeneracy scheme} of $Q_W$.
We set 
\[
M_W:=\Fitt_1\coker Q_W\unlhd\KK[x].
\]
Note the different fonts used for $M_W$ and $\M_W$ (see Definition~\ref{60}).
\end{dfn}


\begin{rmk}[Configuration forms as matrices]\label{151}
With respect to a basis $w=(w^1,\dots,w^r)$ of $W$, $Q_W$ becomes a matrix of Hadamard products (see Remark~\ref{65})
\[
Q_w=(\ideal{x,w^i\star w^j})_{i,j}=\left(\sum_{e\in E}x_e\cdot w^i_e\cdot w^j_e\right)_{i,j}\in\KK^{r\times r},\quad w^i_e=e^\vee(w^i).
\]
Let $Q^{i,j}$ denote the submaximal minor of a square matrix $Q$ obtained by deleting row $i$ and column $j$.
Then
\[
M_W=\ideal{Q_W^{i,j}\xmid i,j\in\set{1,\dots,r}}.
\]
Any basis of $W$ can be written as $w'=Uw$ for some $U\in\Aut_\KK W$.
Then 
\[
Q_{w'}=UQ_wU^t.
\]
and the $Q_{w'}^{i,j}$ become $\KK$-linear combinations of the $Q_w^{i,j}$.
We often consider $Q_W$ as a matrix $Q_w$ determined up to conjugation.
\end{rmk}


\begin{rmk}[Configuration forms and basis scaling]\label{161}
Scaling $E$ results in scaling $x$ in $Q_W$ and in $M_W$ (see Remark~\ref{160}).
\end{rmk}


Bloch, Esnault and Kreimer defined $\psi_W$ in terms of $Q_W$ (see \cite[Lem.~1.3]{BEK06}).

\begin{lem}[Configuration polynomial from configuration form]\label{36}
For any configuration $W\subseteq\KK^E$, the configuration polynomial
\[
\psi_W=\det Q_W\in M_W
\]
is the determinant of the configuration form (see Remarks~\ref{150} and \ref{151}).\qed
\end{lem}


\begin{exa}[Configuration form of the prism realization]\label{215}
Consider the realization $W$ of the prism matroid with basis given in Lemma~\ref{42}.
Then the corresponding matrix of $Q_W$ reads (see Remark~\ref{151})
\[
Q_W=
\begin{pmatrix}
x_1+x_2 & 0 & 0 & x_1 \\
0 & x_3+x_4 & 0 & x_3 \\
0 & 0 & x_5+x_6 & x_5 \\
x_1 & x_3 & x_5 & x_1+x_3+x_5
\end{pmatrix}.
\]
Lemma~\ref{36} recovers the polynomial $\det Q_W=\psi_W$ in Example~\ref{128}.
\end{exa}


The following result describes the behavior of Fitting ideals of configuration forms under duality.
We consider the torus
\[
\TT^E:=(\KK^*)^E\subset\KK^E,\quad\KK[\TT^E]=\KK[x^{\pm1}]=\KK[x]_{x^E}.
\]
The \emph{Cremona isomorphism} $\TT^E\cong\TT^{E^\vee}$ is defined by
\begin{equation}\label{190}
\zeta_E\colon\KK[\TT^E]\cong\KK[\TT^{E^\vee}],\quad x_e^{-1}\leftrightarrow x_{e^\vee},\quad e\in E.
\end{equation}


\begin{prp}[Duality and cokernels of configuration forms]\label{162}
Let $W\subseteq\KK^E$ be a configuration.
Then there is an isomorphism over $\zeta_E$, 
\[
\coker(Q_W)_{x^E}\cong\coker(Q_{W^\perp})_{x^{E^\vee}},
\]
where the indices denote localization (see \eqref{214}).
In particular, this induces an isomorphism
\[
(M_W)_{x^E}\cong(M_{W^\perp})_{x^{E^\vee}}.
\]
\end{prp}

\begin{proof}
Consider the short exact sequence
\begin{equation}\label{165}
\xymatrix{
0\ar[r] & W\ar[r] & \KK^E\ar[r] & \KK^E/W\ar[r] &0
}
\end{equation}
and its $\KK$-dual 
\begin{equation}\label{166}
\xymatrix{
0 & W^\vee\ar[l] & \KK^{E^\vee}\ar[l] & W^\perp\ar[l] & 0.\ar[l]
}
\end{equation}
We identify $\KK^E=\KK^{E^{\vee\vee}}$ and $\KK^E/W=W^{\perp\vee}$, and we abbreviate
\[
Q:=Q_{\KK^E},\quad Q^\vee:=Q_{\KK^{E^\vee}}.
\]
Then $Q_{x^E}$ and $Q^\vee_{x^{E^\vee}}$ are mutual inverses under $\zeta_E$.
Together with \eqref{165} and \eqref{166} tensored by $\KK[x^{\pm1}]$ and \eqref{163} for $W$ and $W^\perp$, they fit into a commutative diagram with exact rows connected vertically by morphisms over $\zeta_E$
\[
\xymatrix{
&&& 0\\
& 0\ar[d] && \coker (Q_{W^\perp})_{x^{E^\vee}}\ar[u]\\
0\ar[r] & W[x^{\pm1}]\ar[d]_-{(Q_W)_{x^E}}\ar[r] & \KK^E[x^{\pm1}]\ar@{-->}[ur]\ar@<-5pt>[d]_-{Q_{x^E}}\ar[r] & W^{\perp\vee}[x^{\pm1}]\ar[u]\ar[r] & 0\\
0 & W^\vee[x^{\pm1}]\ar[d]\ar[l] & \KK^{E^\vee}[x^{\pm1}]\ar@{-->}[dl]\ar@<-5pt>[u]_-{Q^\vee_{x^{E^\vee}}}\ar[l] & W^\perp[x^{\pm1}]\ar[u]_-{(Q_{W^\perp})_{x^{E^\vee}}}\ar[l] & 0\ar[l]\\
& \coker (Q_W)_{x^E}\ar[d] && 0\ar[u]\\
& 0,
}
\]
where $-[x^{\pm1}]$ means $-\otimes\KK[x^{\pm1}]$.
Exactness of the columns is due to $\det Q_W=\psi_W\ne0$ (see Lemma~\ref{36} and Remark~\ref{6}).
Composing the middle vertical isomorphism over $\zeta_E$ with (taking preimages along) the dashed compositions yields the claimed isomorphism by a diagram chase.
\end{proof}


The following result describes the behavior of submaximal minors of configuration forms under deletion--contraction.
It is the basis for our inductive approach to second degeneracy schemes.


\begin{lem}[Deletion--contraction for submaximal minors]\label{40}
Let $W\subseteq\KK^E$ be a realization of a matroid $\M$ of rank $r=\rk\M$, and let $e\in E$.
Then any basis of $W/e$ can be extended to bases of $W$ and $W\setminus e$ such that $Q_W^{i,j}=$
\[
\begin{cases}
Q_{W\setminus e}^{i,j}=Q_{W/e}^{i,j} & \text{if $e$ is a loop in $\M$,}\\
\psi_{W\setminus e}=\psi_{W/ e} & \text{if $e$ is a coloop in $\M$, $i=r=j$,}\\
x_e\cdot Q_{W\setminus e}^{i,j}=x_e\cdot Q_{W/ e}^{i,j} & \text{if $e$ is a coloop in $\M$, $i\ne r\ne j$,}\\
0 & \text{if $e$ is a coloop in $\M$, otherwise,}\\
\psi_{W/e}& \text{if $e$ is not a (co)loop in $\M$, $i=r=j$,}\\
Q_{W\setminus e}^{i,j}& \text{if $e$ is not a (co)loop in $\M$, $i=r$ or $j=r$,}\\
Q_{W\setminus e}^{i,j}+x_e\cdot Q_{W/e}^{i,j} & \text{if $e$ is not a (co)loop in $\M$, $i\ne r\ne j$,}
\end{cases}
\]
for all $i,j\in\set{1,\dots,r}$.
In particular, the $Q_W^{i,j}$ are linear combinations of square-free monomials for any basis of $W$.
\end{lem}

\begin{proof}
Pick a basis $w^1,\dots,w^r$ of $W\subseteq\KK^E$ and consider 
\[
Q_W=\left(\sum_{e\in E}x_e\cdot w^i_e\cdot w^j_e\right)_{i,j}\in\KK^{r\times r}
\]
as a matrix (see Remark~\ref{151}).
Recall that (see Definition~\ref{33}.\eqref{33d} and \eqref{33e}),
\[
W\backslash e=\pi_{E\setminus\set{e}}(W),\quad
W/e=W\cap\KK^{E\setminus\set{e}},
\]
and the description of (co)loops in Remark~\ref{34}.\eqref{34a}:
\begin{asparaitem}

\item If $e$ is a loop, then $w^i_e=0$ for all $i=1,\dots,r$ and hence $W\setminus e=W=W/e$.

\item If $e$ is not a loop, then we may adjust $w^1,\dots,w^r$ such that $w^i_e=\delta_{i,r}$ for all $i=1,\dots,r$ and then $w^1,\dots,w^{r-1}$ is a general basis of $W/e$.

\item If $e$ is a coloop, then we may adjust $w^r=e$ and $\pi_{E\setminus\set{e}}$ identifies $w^1,\dots,w^{r-1}$ with a basis of $W\setminus e=W/e$.
\end{asparaitem}
In the latter case,
\begin{equation}\label{51}
Q_W=\begin{pmatrix}
Q_{W\setminus e} & 0\\
0 & x_e
\end{pmatrix},
\end{equation}
and the claimed equalities follow (see Lemma~\ref{36}).

It remains to consider the case in which $e$ is not a (co)loop.
Then $\iota_{E\setminus\set{e}}$ and $\pi_{E\setminus\set{e}}$ (see \eqref{206}) identify $w^1,\dots,w^{r-1}$ and $w^1,\dots,w^r$ with bases of $W/e$ and $W\setminus e$, respectively.
Hence,
\begin{equation}\label{52}
Q_{W\setminus e}=
\begin{pmatrix}
Q_{W/e} & b\\
b^t & a
\end{pmatrix},\quad
Q_W=\begin{pmatrix}
Q_{W/e} & b\\
b^t & x_e+a
\end{pmatrix},
\end{equation}
where both the entry $a$ and column $b$ are independent of
$x_e$.
We consider two cases.
If $i=r$ or $j=r$, then clearly $Q_W^{i,j}=Q_{W\setminus e}^{i,j}$.
Otherwise,
\[
Q_W^{i,j}=Q_{W\setminus e}^{i,j}+x_e\cdot Q_{W/e}^{i,j}.
\] 
This proves the claimed equalities also in this case (see Lemma~\ref{36}) and the particular claim follows.
\end{proof}


As an application of Lemma~\ref{36}, we describe the behavior of configuration polynomials under $2$-separations.


\begin{prp}[Configuration polynomials and $2$-separations]\label{63}
Let $W\subseteq\KK^E$ be a realization of a connected matroid $\M$.
Suppose that $E=E_1\sqcup E_2$ is an (exact) $2$-separation of $\M$.
Then
\[
\psi_W=\psi_{W/E_1}\cdot\psi_{W\vert_{E_1}}+\psi_{W\vert_{E_2}}\cdot\psi_{W/E_2}.
\]
\end{prp}

\begin{proof}
We adopt the notation from \cite[\S8.2]{Tru92}.
Extend a basis $B_2\in\B_{\M\vert_{E_2}}$ to a basis $B\in\B_\M$.
Then $W$ is the row span of a matrix (see \cite[(8.1.1)]{Tru92} and Remark~\ref{65})
\[
A=\begin{pmatrix}
I & 0 & A_1 & 0\\
0 & I & D & A_2'
\end{pmatrix},
\]
where the block columns are indexed by $B\setminus B_2,B_2,E_1\setminus B,E_2\setminus B_2$, and $\rk D=1$.
After suitably ordering and scaling $B_2$, $E_1\setminus B$ the lower rows of $A$, we may assume that
\begin{align*}
D&=(1\ b)^ta_1,\\
a_1&=
\begin{pmatrix}
1 & \cdots & 1 & 0 & \cdots & 0
\end{pmatrix}\ne0,\\
b&=
\begin{pmatrix}
1 & \cdots & 1 & 0 & \cdots & 0
\end{pmatrix}.
\end{align*}
The size of $b$ and $a_1$ is determined by number of rows and columns of $D$, respectively.
While $b$ could be $0$, at least one entry of $a_1$ is a $1$.
After suitable row operations and adjusting signs of $B_2$, we can repartition 
\begin{equation}\label{130}
A=\begin{pmatrix}
I & 0 & 0 & A_1 & 0\\
0 & 1 & 0 & a_1 & a_2\\
0 & b^t & I & 0 & A_2
\end{pmatrix}.
\end{equation}
Denote by $e\in E$ the index of the column $(0\ 1\ b)^t$.
Let $X_1,x_e,X_2,X_1',X_2'$ be diagonal matrices of variables corresponding to the block columns of $A$.
Then the configuration form of $W$ becomes (see Remark~\ref{151})
\[
Q_W=
\begin{pmatrix}
X_1+A_1X_1'A_1^t & A_1X_1'a_1^t & 0 \\
a_1X_1'A_1^t & x_e+a_1X_1'a_1^t+a_2X_2'a_2^t & x_eb+a_2X_2'A_2^t \\
0 & b^tx_e+A_2X_2'a_2^t & b^tx_eb+X_2+A_2X_2'A_2^t
\end{pmatrix},
\]
which involves 
\begin{align*}
Q_{W\vert_{E_1}}&=
\begin{pmatrix}
Q_{W/E_2} & A_1X_1'a_1^t\\
a_1X_1'A_1^t & a_1X_1'a_1^t \\
\end{pmatrix},\\
Q_{W/E_2}&=X_1+A_1X_1'A_1^t,\\
Q_{W\vert_{E_2}}&=
\begin{pmatrix}
x_e+ a_2X_2'a_2^t & x_eb+a_2X_2'A_2^t \\
b^tx_e+A_2X_2'a_2^t & Q_{W/E_1}
\end{pmatrix},\\
Q_{W/E_1}&=b^tx_eb+X_2+A_2X_2'A_2^t.
\end{align*}
Laplace expansion of $\psi_W=\det Q_W$ (see Lemma~\ref{36}) along the $e$th column yields the claimed formula.
\end{proof}


\begin{rmk}[Configuration polynomials and handles]\label{67}
Let $W\subseteq\KK^E$ be a realization of a connected matroid $\M$, and let $H\in\H_\M$ be a separating handle.
By Lemma~\ref{16}.\eqref{16c}, $H$ is a $2$-separation of $\M$.
Proposition~\ref{63} applied to $E=(E\setminus H)\sqcup H$ thus yields the statement of Corollary~\ref{18} in this case.
\end{rmk}

\section{Configuration hypersurfaces}\label{91}

In this section, we establish our main results on Jacobian and second degeneracy schemes of realizations of connected matroids:
the second degeneracy scheme is Cohen--Macaulay, the Jacobian scheme equidimensional, of codimension $3$ (see Theorem~\ref{100}).
The second degeneracy scheme is reduced, the Jacobian scheme generically reduced if $\ch\KK\ne2$ (see Theorem~\ref{100}).

\subsection{Commutative ring basics}\label{108}

In this subsection, we review the relevant preliminaries on equidimensionality and graded Cohen--Macaulay\-ness using the books of Matsumura (see \cite{Mat89}) and Bruns and Herzog (see \cite{BH93}) as comprehensive references.
For the benefit of the non-experts we provide detailed proofs.
Further we relate generic reducedness for a ring and an associated graded ring (see Lemma~\ref{95}).

\subsubsection{Equidimensionality of rings}\label{79}

Let $R$ be a Noetherian ring.
We denote by $\Min\Spec R$ and $\Max\Spec R$ the sets of minimal and maximal elements of the set $\Spec R$ of prime ideals of $R$ with respect to inclusion.
The subset $\Ass R\subseteq\Spec R$ of \emph{associated primes} of $R$ is finite and $\Min\Spec R\subseteq\Ass R$ (see \cite[Thm.~6.5]{Mat89}).

One says that $R$ is \emph{catenary} if every saturated chain of prime ideals joining $\pp,\qq\in\Spec R$ with $\pp\subseteq\qq$ has (maximal) length $\height(\qq/\pp)$ (see \cite[31]{Mat89}).
We say that $R$ is \emph{equidimensional} if it is catenary and
\[
\forall\pp\in\Min\Spec R\colon\forall\mm\in\Max\Spec R\colon
\pp\subseteq\mm\implies\height(\mm/\pp)=\dim R.
\]
If $R$ is a finitely generated $\KK$-algebra, then these two conditions reduce to (see  \cite[Thm.~2.1.12]{BH93} and \cite[Thm.~5.6]{Mat89})
\[
\forall\pp\in\Min\Spec R\colon
\dim(R/\pp)=\dim R.
\]
We say that $R$ is \emph{pure-dimensional} if 
\[
\forall\pp\in\Ass R\colon
\dim(R/\pp)=\dim R,
\]
which implies in particular that $\Ass R=\Min\Spec R$.
It follows that pure-dimensional finitely generated $\KK$-algebras are equidimensional.


The following lemma applies to any equidimensional finitely generated $\KK$-algebra.

\begin{lem}[Height bound for adding elements]\label{109}
Let $R$ be a Noetherian ring such that $R_\mm$ is equidimensional for all $\mm\in\Max\Spec R$.
\begin{enumerate}[(a)]
\item\label{109a} All saturated chains of primes in $\pp\in\Spec R$ have length $\height\pp$.
\item\label{109b} For any $\pp\in\Spec R$, $x\in R$ and $\qq\in\Spec R$ minimal over $\pp+\ideal{x}$,
\[
\height\qq\le\height\pp+1.
\]
\end{enumerate}
\end{lem}

\begin{proof}\
\begin{asparaenum}[(a)]

\item Take two such chains of length $n$ and $n'$ starting at minimal primes $\pp_0$ and $\pp_0'$, respectively.
Extend both by a saturated chain of primes of length $m$ containing $\pp$ and ending in a maximal ideal $\mm$.
Since $R_\mm$ is equidimensional by hypothesis, these extended chains have length $n+m=n'+m$.
Therefore, the two chains have length $n=n'$.

\item By Krull's principal ideal theorem, $\height(\qq/\pp)\le1$.
Take a chain of primes in $\pp$ of length $\height\pp$ and extend it by $\qq$ if $\pp\ne\qq$.
By \eqref{109a}, this extended chain has length $\height\qq$ and the claim follows.\qedhere

\end{asparaenum}
\end{proof}


\begin{lem}[Equidimensional finitely generated algebras and localization]\label{117}
Let $R$ be an equidimensional finitely generated $\KK$-algebra and $x\in R$.
If $R_x\ne0$, then $R_x$ is equidimensional of dimension $\dim R_x=\dim R$.
\end{lem}

\begin{proof}
Any minimal prime ideal of $R_x$ is of the form $\pp_x$ where $\pp\in\Min\Spec R$ with $x\not\in\pp$.
By the Hilbert Nullstellensatz (see \cite[Thm.~5.5]{Mat89}), 
\[
\bigcap\Max V(\pp)=\pp.
\]
This yields an $\mm\in\Max\Spec R$ such that $\pp\subseteq\mm\not\ni x$ and hence $\pp_x\subseteq\mm_x\in\Max\Spec R_x$.
Since $R$ and hence $R_x$ is a finitely generated $\KK$-algebra,
\[
\dim(R_x/\pp_x)=\height(\mm_x/\pp_x)=\height(\mm/\pp)=\dim R
\]
by equidimensionality of $R$.
The claim follows.
\end{proof}

\subsubsection{Generic reducedness}\label{81}

The following types of Artinian local rings coincide: field, regular ring, integral domain and reduced ring (see \cite[Thms.~2.2, 14.3]{Mat89}).
A Noetherian ring $R$ is \emph{generically reduced} if the Artinian local ring $R_\pp$ is reduced for all $\pp\in\Min\Spec R$ (see \cite[Exc.~5.2]{Mat89}).
This is equivalent to $R$ satisfying Serre's condition ($R_0$).
We use the same notions for the associated affine scheme $\Spec R$.

\begin{dfn}[Generic reducedness]\label{131}
We call a Noetherian scheme $X$ \emph{generically reduced along a subscheme $Y$} if $X$ is reduced at all generic points specializing to a point of $Y$.
If $X=\Spec R$ is an affine scheme, then we use the same notions for the Noetherian ring $R$.
\end{dfn}


\begin{lem}[Reducedness and purity]\label{176}
A Noetherian ring $R$ is reduced if it is generically reduced and pure-dimensional.
\end{lem}

\begin{proof}
Since $R$ is pure-dimensional, $\Ass R=\Min\Spec R$, and hence, $R$ becomes a subring of localizations (see \cite[Thm.~6.1.(i)]{Mat89})
\[
R\into\bigoplus_{\pp\in\Ass R}R_\pp=\bigoplus_{\Min\Spec R}R_\pp.
\]
The latter ring is reduced since $R$ is generically reduced, and the claim follows.
\end{proof}


\begin{lem}[Reducedness and reduction]\label{94}
Let $(R,\mm)$ be a local Noetherian ring.
Suppose that $R/tR$ is reduced for a system of parameters $t$.
Then $R$ is regular and, in particular, an integral domain and reduced.
\end{lem}

\begin{proof}
By hypothesis, $R/tR$ is local Artinian with maximal ideal $\mm/tR$.
Reducedness makes $R/tR$ a field, and hence, $\mm=tR$.
By definition, this means that $R$ is regular.
In particular, $R$ is an integral domain and reduced (see \cite[Thm.~14.3]{Mat89}).
\end{proof}


\begin{dfn}[Rees algebras]\label{173}
Let $R$ be a ring and $I\unlhd R$ an ideal.
The \emph{(extended) Rees algebra} is the $R[t]$-algebra (see \cite[Def.~5.1.1]{HS06})
\[
\Rees_IR:=R[t,It^{-1}]\subseteq R[t^{\pm1}].
\]
The \emph{associated graded algebra} is the $R/I$-algebra
\[
\gr_IR:=\bigoplus_{i=0}^\infty I^i/I^{i+1}.
\]
\end{dfn}


\begin{lem}[Generic reducedness from associated graded ring]\label{95}
Let $R$ be a Noetherian $d$-dimensional ring, $I\unlhd R$ an ideal, $S:=\Rees_IR$ and $\bar R:=\gr_IR$.

\begin{enumerate}[(a)]

\item\label{95a} Suppose $R$ is an equidimensional finitely generated $\KK$-algebra.
Then $S$ is a $(d+1)$-equidimensional finitely generated $\KK$-algebra.

\item\label{95b} If $S$ is $(d+1)$-equidimensional and $I\ne R$, then $\bar R$ is $d$-equidimensional.

\item\label{95c} If $S$ is equidimensional and $\bar R$ is generically reduced, then $R$ is generically reduced along $V(I)$.

\end{enumerate}
\end{lem}

\begin{proof}
There are ring homomorphisms
\[
R\to R[t]\to S\to S/tS\cong\bar R.
\]
Since $R$ is Noetherian, $I$ is finitely generated and $S$ finite type over $R$.

\begin{asparaenum}[(a)]

\item If $R$ is an integral domain, then so are $S\subseteq R[t^{\pm1}]$.
By definition, formation of the Rees ring commutes with base change.
After base change to $R/\pp$ for some $\pp\in\Min\Spec R$, we may assume that $R$ is a $d$-dimensional integral domain.
Then $S$ is a $(d+1)$-dimensional integral domain (see \cite[Thm.~5.1.4]{HS06}).
Since $S$ is a finitely generated $\KK$-algebra (as $R$ is one), $S$ is equidimensional.

\item Multiplication by $t$ is injective on $R[t^{\pm1}]$ and hence on $S$.
If $I\ne R$, then $S/tS\cong\bar R\ne0$ and $t$ is an $S$-sequence.
Since $S$ is $(d+1)$-equidimensional, $\bar R$ is $d$-equidimensional by Krull's principal ideal theorem.

\item Let $\pp\in\Min\Spec R$ and consider the extension $\pp[t^{\pm1}]\in\Spec R[t^{\pm1}]$.
Then (see \cite[p.~96]{HS06})
\[
t\not\in\tilde\pp:=\pp[t^{\pm1}]\cap S\in\Min\Spec S
\]
and hence
\begin{equation}\label{110}
S_{\tilde\pp}=(S_t)_{\tilde\pp_t}=R[t^{\pm1}]_{\pp[t^{\pm1}]}.
\end{equation}
Since $\pp[t^{\pm1}]\cap R=\pp$, the map $R\to R[t^{\pm1}]$ induces an injection
\begin{equation}\label{111}
R_\pp\into R[t^{\pm1}]_{\pp[t^{\pm1}]}.
\end{equation}
To check injectivity, consider $R_\pp\ni x/1\mapsto0\in R[t^{\pm1}]_{\pp[t^{\pm1}]}$.
Then $0=xy\in R[t^{\pm1}]$ for some $y=\sum_iy_it^i\in R[t^{\pm1}]\setminus\pp[t^{\pm1}]$.
Then $0=xy_i\in R$ for all $i$ and $y_j\in R\setminus\pp$ for some $j$.
It follows that $0=x/1\in R_\pp$.
Combining \eqref{110} and \eqref{111} reducedness of $R_\pp$ follows from reducedness of $S_{\tilde\pp}$.

Suppose now that $V(\pp)\cap V(I)\ne\emptyset$ and hence (the subscript denoting graded parts)
\[
R\ne\pp+I=\tilde\pp_0+(tS)_0=(\tilde\pp+tS)_0
\]
implies that $\tilde\pp+tS\ne S$.
Let $\qq\in\Spec S$ be a minimal prime ideal over $\tilde\pp+tS$.
No minimal prime ideal of $S$ contains the $S$-sequence $t\in\qq$.
By Lemma~\ref{109}.\eqref{109b}, $\height\qq=1$ and $\qq$ is minimal over $t$.
This makes $t$ a parameter of the localization $S_\qq$.
Under $S/tS\cong\bar R$, the minimal prime ideal $\qq/tS\in\Spec(S/tS)$ corresponds to a minimal prime ideal $\bar\qq\in\Spec\bar R$.
Suppose that $\bar R$ is generically reduced.
Then 
\[
S_\qq/tS_\qq=(S/tS)_{\qq/tS}\cong\bar R_{\bar\qq}
\]
is reduced.
By Lemma~\ref{94}, $S_\qq$ and hence its localization $(S_\qq)_{\tilde\pp_\qq}=S_{\tilde\pp}$ is reduced.
Then also $R_\pp$ is reduced, as shown before.\qedhere

\end{asparaenum}
\end{proof}

\subsubsection{Graded Cohen--Macaulay rings}\label{80}

Let $(R,\mm)$ be a Noetherian $^*$local ring (see \cite[Def.~1.5.13]{BH93}).
By definition, this means that $R$ is a graded ring with unique maximal graded ideal $\mm$.
For any $\pp\in\Spec R$, denote by $\pp^*\in\Spec R$ the maximal graded ideal contained in $\pp$ (see \cite[Lem.~1.5.6.(a)]{BH93}).
For any $\pp\in\Spec R$, there is a chain of maximal length of graded prime ideals strictly contained in $\pp$ (see \cite[Lem.~1.5.8]{BH93}).
If $\mm\not\in\Max\Spec R$, then such a chain for $\nn\in\Max\Spec R$ ends with $\mm\subsetneq\nn$.
It follows that
\begin{equation}\label{8}
\dim R=
\begin{cases}
\dim R_\mm & \text{if }\mm\in\Max\Spec R,\\
\dim R_\mm+1 & \text{if }\mm\not\in\Max\Spec R.
\end{cases}
\end{equation}
For any proper graded ideal $I\lhd R$ also $(R/I,\mm/I)$ is $^*$local and
\begin{equation}\label{23}
\mm\in\Max\Spec R\iff\mm/I\in\Max\Spec(R/I).
\end{equation}
Any associated prime $\pp\in\Ass R$ is graded (see \cite[Lem.~1.5.6.(b).(ii)]{BH93}) and hence $\pp\subseteq\mm$.
This yields a bijection (see \cite[Thm.~6.2]{Mat89})
\begin{equation}\label{26}
\Ass R\to\Ass R_\mm,\quad\pp\mapsto\pp_\mm.
\end{equation}
If $I\unlhd R$ is a graded ideal and $\pp\in\Spec R$ minimal over $I$, then $\pp/I\in\Min\Spec(R/I)\subseteq\Ass(R/I)$, and hence, $\pp$ is graded.


The following lemma shows in particular that $^*$local Cohen--Macaulay rings are pure- and equidimensional.


\begin{lem}[Height and codimension]\label{27}
Let $(R,\mm)$ be a $^*$local Cohen--Macaulay ring and $I\unlhd R$ a graded ideal.
Then $R$ is pure-dimensional and
\begin{equation}\label{174}
\height I=\codim I.
\end{equation}
In particular, $R/I$ is equidimensional if and only if $\height\pp=\codim I$ for all minimal $\pp\in\Spec R$ over $I$.
\end{lem}

\begin{proof}
The $^*$local ring $(R,\mm)$ is Cohen--Macaulay if and only if the localization $R_\mm$ is Cohen--Macaulay (see \cite[Exc.~2.1.27.(c)]{BH93}).
In particular, $R_\mm$ is pure-dimensional (see \cite[Prop.~1.2.13]{BH93}) and (see \cite[Cor.~2.1.4]{BH93})
\begin{equation}\label{175}
\height I_\mm=\codim I_\mm
 \end{equation}
Using \eqref{8}, \eqref{23} for $I=\pp$ and bijection~\eqref{26}, it follows that $R$ is pure-dimensional:
\begin{align*}
\forall\pp\in\Ass R\colon
\dim R&=
\begin{cases}
\dim R_\mm & \text{if }\mm\in\Max\Spec R,\\
\dim R_\mm+1 & \text{if }\mm\not\in\Max\Spec R,\\
\end{cases}\\
&=
\begin{cases}
\dim(R_\mm/\pp_\mm) & \text{if }\mm\in\Max\Spec R,\\
\dim(R_\mm/\pp_\mm)+1 & \text{if }\mm\not\in\Max\Spec R,\\
\end{cases}\\
&=
\begin{cases}
\dim(R/\pp)_{\mm/\pp} & \text{if }\mm\in\Max\Spec R,\\
\dim(R/\pp)_{\mm/\pp}+1 & \text{if }\mm\not\in\Max\Spec R,\\
\end{cases}\\
&=\dim(R/\pp).
\end{align*}
Using \eqref{8} and \eqref{23}, \eqref{174} follows from \eqref{175}:
\begin{align*}
\height I
&=\height I_\mm
=\codim I_\mm\\\nonumber
&=\dim R_\mm-\dim(R_\mm/I_\mm)\\\nonumber
&=\dim R_\mm-\dim(R/I)_{\mm/I}\\\nonumber
&=\dim R-\dim(R/I)
=\codim I.
\end{align*}
Since $R$ is Cohen--Macaulay, it is (universally) catenary (see \cite[Thm.~2.1.12]{BH93}).
By \eqref{23} and the preceding discussion of chains of prime ideals in $R/I$ and $R/\pp$, $I$ is equidimensional if and only if $\dim(R/I)=\dim(R/\pp)$ for all prime ideals $\pp\in\Spec R$ minimal over $I$.
The particular claim then follows by \eqref{174} for $I$ and $\pp$.
\end{proof}

\subsection{Jacobian and degeneracy schemes}\label{77}

In this subsection, we associate Jacobian and second degeneracy schemes to a configuration.
By results of Patterson and Kutz, their supports coincide and their codimension is at most $3$.


For a Noetherian ring $R$, we consider the associated affine (Noetherian) \emph{scheme} $\Spec R$, whose underlying set consists of all prime ideals of $R$.
We refer to elements of $\Min\Spec R$ as \emph{generic points}, of $\Ass R$ as  \emph{associated points}, and of $\Ass R\setminus\Min\Spec R$ as \emph{embedded points} of $\Spec R$.
An ideal $I\unlhd R$ defines a \emph{subscheme} $\Spec(R/I)\subseteq\Spec R$.

By abuse of notation we identify
\[
\KK^E=\Spec\KK[x].
\]
Due to Lemma~\ref{27},
\[
\codim_{\KK^E}\Spec(\KK[x]/I)=\height I
\]
for any graded ideal $I\unlhd\KK[x]$.


\begin{dfn}[Configuration schemes]\label{49}
Let $W\subseteq\KK^E$ be a configuration.
Then the subscheme
\[
X_W:=\Spec(\KK[x]/\ideal{\psi_W})\subseteq\KK^E
\]
is called the \emph{configuration hypersurface} of $W$.
In particular, $X_G:=X_{W_G}$ is the \emph{graph hypersurface} of $G$ (see Definition~\ref{43}).
The ideal
\[
J_W:=\ideal{\psi_W}+\ideal{\partial_e\psi_W\xmid e\in E}\unlhd\KK[x]
\]
is the \emph{Jacobian ideal} of $\psi_W$.
We call the subschemes (see Definition~\ref{64})
\[
\Sigma_W:=\Spec(\KK[x]/J_W)\subseteq\KK^E,\quad\Delta_W:=\Spec(\KK[x]/M_W)\subseteq\KK^E,
\]
the \emph{Jacobian scheme} of $X_W$ and the \emph{second degeneracy scheme} of $Q_W$.
\end{dfn}


\begin{rmk}[Degeneracy and non-smooth loci]\label{14}
If $\ch\KK\not\divides\rk\M=\deg\psi$ (see Remark~\ref{6}), then $\psi_W$ is a redundant generator of $J_W$ due to the Euler identity.
By Lemma~\ref{36}, $X_W^\red$ and $\Delta_W^\red$ are the first and second degeneracy loci of $Q_W$ (see Definition~\ref{64}), whereas $\Sigma_W^\red$ is the \emph{non-smooth locus} of $X_W$ over $\KK$ (see \cite[Thm.~30.3.(1)]{Mat89}).
If $\KK$ is perfect, then $\Sigma_W^\red$ is the \emph{singular locus} of $X_W$ (see \cite[\S28, Lem.~1]{Mat89}).
\end{rmk}


\begin{rmk}[Loops and line factors]\label{84}
Let $W\subseteq\KK^E$ be a realization of matroid $\M$.
Suppose that $e$ is a loop in $\M$, that is, $e^\vee\vert_W=0$.
Then $\psi_W$ and $Q_W$ are independent of $x_e$ (see Remark~\ref{6} and Definition~\ref{64})
\[
X_W=X_{W\setminus e}\times\AA^1,\quad
\Sigma_W=\Sigma_{W\setminus e}\times\AA^1,\quad
\Delta_W=\Delta_{W\setminus e}\times\AA^1.\qedhere
\]
\end{rmk}


\begin{lem}[Inclusions of schemes]\label{50}
For any configuration $W\subseteq\KK^E$, there are inclusions of schemes $\Delta_W\subseteq\Sigma_W\subseteq X_W\subseteq\KK^E$.
\end{lem}

\begin{proof}
By definition, $\psi_W\in J_W$ and hence the second inclusion.
By Lemma~\ref{36}, $\psi_W=\det Q_W\in M_W$ and hence $\partial_e\psi_W\in M_W$ for all $e\in E$.
Thus, $J_W\subseteq M_W$ and the first inclusion follows.
\end{proof}


\begin{rmk}[Schemes for matroids of small rank]\label{20}
Let $W\subseteq\KK^E$ be a realization of a matroid $\M$.
\begin{asparaenum}[(a)]

\item\label{20a} 
If $\rk\M\le1$, then $\psi_W=1$ (see Remark~\ref{6}) or $\psi_W\ne0$ is a $\KK$-linear form.
In both cases, $\Sigma_W=\emptyset=\Delta_W$.
If $\rk\M\ge2$, then $\ideal{x}\in\Sigma_W\ne\emptyset\ne\Delta_W\ni\ideal{x}$.
\item\label{20b} If $\rk\M=2$, then $\Delta_W$ is a $\KK$-linear subspace of $\KK^E$ and hence an integral scheme.
If $\ch\KK\ne2$, the same holds for $\Sigma_W$ due to the Euler identity (see Remark~\ref{14}). 
Otherwise, the non-redundant quadratic generator $\psi_W$ of $J_W$ can make $\Sigma_W$ non-reduced (see Example~\ref{107}).\qedhere

\end{asparaenum}
\end{rmk}


\begin{exa}[Schemes for the triangle]\label{107}
Let $\M$ be a matroid on $E\in\C_\M$ with $\abs{E}=3$ and hence $\rk\M=\abs{E}-1=2$.
Up to scaling and ordering $E=\set{e_1,e_2,e_3}$, any realization $W\subseteq\KK^E$ of $\M$ has the basis
\[
w^1:=e_1+e_3,\quad w^2:=e_2+e_3.
\]
With respect to this basis, we compute
\begin{align*}
Q_W&=
\begin{pmatrix}
x_1+x_3 & x_3\\
x_3 & x_2+x_3
\end{pmatrix},\\
M_W&=\ideal{x_1+x_3,x_2+x_3,x_3}=\ideal{x_1,x_2,x_3}.
\end{align*}
It follows that $\Delta_W$ is a reduced point.

On the other hand,
\begin{align*}
\psi_W&=\det Q_W=x_1x_2+x_1x_3+x_2x_3,\\
J_W&=\ideal{\psi_W,x_1+x_2,x_1+x_3,x_2+x_3}.
\end{align*}
The matrix expressing the linear generators $x_1+x_2,x_1+x_3,x_2+x_3$ in terms of the variables $x_1,x_2,x_3$ has determinant $2$.
If $\ch\KK\ne2$, then $J_W=\ideal{x_1,x_2,x_3}$ and $\Sigma_W$ is a reduced point.
Otherwise,
\[
J_W=\ideal{\psi_W,x_1-x_3,x_2-x_3}=\ideal{x_1-x_3,x_2-x_3,x_3^2}
\]
and $\Sigma_W$ is a non-reduced point.
\end{exa}


\begin{lem}\label{125}
Consider two sets of variables $x=x_1,\dots,x_n$ and $y=y_1,\dots,y_m$.
Let $0\ne f\in I\unlhd\KK[x]$ and $0\ne g\in J\unlhd\KK[y]$.
Then
\[
f\cdot J[x]+I[y]\cdot g=\ideal{f,g}\cap I[y]\cap J[x]\unlhd\KK[x,y].
\]
\end{lem}

\begin{proof}
For the non-obvious inclusion, take $h=af+bg\in I[y]\cap J[x]$. 
Since $f\in I[y]$, $bg\in I[y]$ and similarly $af\in J[x]$. 
Since $f\ne0$ and $J$ are in different variables, it follows that $a\in J[x]$ and similarly $b\in I[y]$. 
\end{proof}


\begin{thm}[Decompositions of schemes]\label{112}
Let $W\subseteq\KK^E$ be a realization of a matroid $\M$ without loops.
Suppose that $\M=\bigoplus_{i=1}^n\M_i$ decomposes into connected components $\M_i$ on $E_i$.
Let $W=\bigoplus_{i=1}^nW_i$ be the induced decomposition into $W_i\subseteq\KK^{E_i}$ (see Lemma~\ref{72}).
Then $X_W$ is the reduced union of integral schemes $X_{W_i}\times\KK^{E\setminus E_i}$, and $\Sigma_W$ is the union of $\Sigma_{W_i}\times\KK^{E\setminus E_i}$ and integral schemes $X_{W_i}\times X_{W_j}\times\KK^{E\setminus(E_i\cup E_j)}$ for $i\ne j$.
The same holds for $\Sigma$ replaced by $\Delta$.
In particular, $X_W$ is generically smooth over $\KK$.
\end{thm}

\begin{proof}
Proposition~\ref{4} yields the claim on $X_W$ (see Remark~\ref{6}).
For the claims on $\Sigma_W$ and $\Delta_W$, we may assume that $n=2$ with $\M_1$ possibly disconnected.
The general case then follows by induction on $n$.

By Proposition~\ref{4} and Definition~\ref{64}, $\psi_W=\psi_{W_1}\cdot\psi_{W_2}$ and $Q_W=Q_{W_1}\oplus Q_{W_2}$.
Then Lemma~\ref{125} yields
\begin{align*}
J_W&=\psi_{W_1}\cdot J_{W_2}[x_{E_1}]+J_{W_1}[x_{E_2}]\cdot\psi_{W_2}\\
&=\ideal{\psi_{W_1},\psi_{W_2}}\cap J_{W_1}[x_{E_2}]\cap J_{W_2}[x_{E_1}],
\end{align*}
and hence,
\[
\Sigma_W=(X_{W_1}\times X_{W_2})\cup(\Sigma_{W_1}\times\KK^{E_2})\cup(\KK^{E_1}\times\Sigma_{W_2}).
\]
The same holds for $J$ and $\Sigma$ replaced by $M$ and $\Delta$, respectively.

Suppose now that $\M$ is connected.
By Proposition~\ref{3}, $\psi_W\nmid\partial_e\psi_W$ for any $e\in E$ and hence $\Sigma_W\subsetneq X_W$.
The particular claim follows.
\end{proof}


Patterson proved the following result (see \cite[Thm.~4.1]{Pat10}).
While Patterson assumes $\ch\KK=0$ and excludes the generator $\psi_W\in J_W$, his proof works in general (see Remark~\ref{14}).
We give an alternative proof using Dodgson identities.

\begin{thm}[Non-smooth loci and second degeneracy schemes]\label{13}
Let $W\subseteq\KK^E$ be a configuration.
Then there is an equality of reduced loci
\[
\Sigma_W^\red=\Delta_W^\red.
\]
In particular, $\Sigma_W$ and $\Delta_W$ have the same generic points, that is,
\[
\Min\Sigma_W=\Min\Delta_W.
\]
\end{thm}

\begin{proof}
Order $E=\set{e_1,\dots,e_n}$ and pick a basis $w=(w^1,\dots,w^r)$ of $W$.
We may assume that its coefficients with respect to $e_1,\dots,e_r$ form an identity matrix, that is, $w^i_{e_j}=\delta_{i,j}$ for $i,j\in\set{1,\dots,r}$.
For $i,j\in\set{1,\dots,r}$ denote by $Q_W^{\set{i,j},\set{i,j}}$ the minor of $Q_W$ obtained by deleting rows and columns $i,j$.
Then there are Dodgson identities (see Remark~\ref{151}, Lemma~\ref{36} and \cite[Lem.~8.2]{BEK06})
\begin{align*}
(Q_W^{i,j})^2=Q_W^{i,j}\cdot Q_W^{j,i}
&=Q_W^{i,i}\cdot Q_W^{j,j}-\det Q_W\cdot Q_W^{\set{i,j},\set{i,j}}\\
&=\partial_i\psi_W\cdot\partial_j\psi_W-\psi_W\cdot Q_W^{\set{i,j},\set{i,j}}\in J_W
\end{align*}
for $i,j\in\set{1,\dots,r}$.
In particular, any prime ideal $\pp\in\Spec\KK[x]$ over $J_W$ contains $M_W$ and hence $\Sigma_W^\red\subseteq\Delta_W^\red$.
The opposite inclusion is due to Lemma~\ref{50}.
\end{proof}


\begin{cor}[Cremona isomorphism]\label{71}
Let $W\subseteq\KK^E$ be a configuration.
Then the Cremona isomorphism $\TT^E\cong\TT^{E^\vee}$ identifies
\begin{align*}
X_W\cap\TT^E&\cong X_{W^\perp}\cap\TT^{E^\vee},\\
\Sigma_W\cap\TT^E&\cong\Sigma_{W^\perp}\cap\TT^{E^\vee},\\
\Delta_W\cap\TT^E&\cong\Delta_{W^\perp}\cap\TT^{E^\vee}.
\end{align*}
In particular, $\Sigma_W$, $\Delta_W$, $\Sigma_{W^\perp}$ and $\Delta_{W^\perp}$ have the same generic points in $\TT^E\cong\TT^{E^\vee}$.
\end{cor}

\begin{proof}
Propositions~\ref{30} and \ref{162} yield the statements for $X_W$ and $\Delta_W$.
The statement for $\Sigma_W$ follows using that $\zeta_E$ (see \eqref{190}) identifies $x_e\partial_e=-x_{e^\vee}\partial_{e^\vee}$ for $e\in E$.
The particular claim follows with Theorem~\ref{13}.
\end{proof}


\begin{prp}[Codimension bound]\label{44}
Let $W\subseteq\KK^E$ be a configuration.
Then the codimensions of $\Sigma_W$ and $\Delta_W$ in $\KK^E$ are bounded by
\[
\codim_{\KK^E}\Sigma_W=\codim_{\KK^E}\Delta_W\le 3.
\]
In case of equality, $\Delta_W$ is Cohen--Macaulay (and hence pure-dimensional) and $\Sigma_W$ is equidimensional.
\end{prp}

\begin{proof}
The equality of codimensions follows from Theorem~\ref{13}.
The scheme $\Delta_W$ is defined by the ideal $M_W$ of submaximal minors of the symmetric matrix $Q_W$ with entries in the Cohen--Macaulay ring $\KK[x]$ (see \cite[2.1.9]{BH93}).
In particular, $\codim_{\KK^E}\Sigma_W=\grade M_W$ (see \cite[2.1.2.(b)]{BH93}).
Kutz proved the claimed inequality and that $M_W$ is a perfect ideal in case of equality (see \cite[Thm.~1]{Kut74}).
In the latter case, $\KK[x]/M_W=\KK[\Delta_W]$ is a Cohen--Macaulay ring (see \cite[Thm.~2.1.5.(a)]{BH93}) and hence pure-dimensional (see Lemma~\ref{27}).
Then $\Sigma_W$ is equidimensional by Theorem~\ref{13}.
\end{proof}

\subsection{Generic points and codimension}\label{75}

In this subsection, we show that the Jacobian and second degeneracy schemes reach the codimension bound of $3$ in case of connected matroids.
The statements on codimension and Cohen--Macaulayness in our main result follow.
In the process, we obtain a description of the generic points in relation with any non-disconnective handle.


\begin{lem}[Primes over the Jacobian ideal and handles]\label{46}
Let $W\subseteq\KK^E$ be a realization of a connected matroid $\M$, and let $H\in\H_\M$ be a proper handle.
\begin{enumerate}[(a)]

\item\label{46a} For any $h\in H$, $x^{H\setminus\set{h}}\cdot\psi_{W\setminus H}\in J_W$.

\item\label{46b} For any $e,f\in H$ with $e\ne f$, $x^{H\setminus\set{e,f}}\cdot\psi_{W\setminus H}\in J_W+\ideal{x_e,x_f}$.

\item\label{46c} For any $d\in H$ and $e\in E\setminus H$, $x^{H\setminus\set{d}}\cdot\partial_e\psi_{W\setminus H}\in J_W+\ideal{x_d}$.

\item\label{46d} If $\pp\in\Spec\KK[x]$ with $J_W\subseteq\pp\not\ni\psi_{W\setminus H}$, then $\ideal{x_e,x_f,x_g}\subseteq\pp$ for some $e,f,g\in H$ with $e\ne f\ne g\ne e$.

\end{enumerate}
\end{lem}

\begin{proof}
By Remark~\ref{160} and Corollary~\ref{18}, we may assume that 
\[
\psi_W=\sum_{h\in H}x^{H\setminus\set{h}}\cdot\psi_{W\setminus H}+x^H\cdot\psi_{W/H}
\]
has the form \eqref{18d}.

\begin{asparaenum}[(a)]

\item Using that $\psi_W$ is a linear combination of square-free monomials (see Definition~\ref{48}),
\[
x^{H\setminus\set{h}}\cdot\psi_{W\setminus H}=\psi_W\vert_{x_h=0}=\psi_W-x_h\cdot\partial_h\psi_W\in J_W.
\]

\item This follows from
\begin{align*}
J_W\ni\partial_e\psi_W&=\sum_{h\in H}x^{H\setminus\set{e,h}}\cdot\psi_{W\setminus H}+x^{H\setminus\set{e}}\cdot\psi_{W/H}\\
&\equiv x^{H\setminus\set{e,f}}\cdot\psi_{W\setminus H}\mod\ideal{x_e,x_f}.
\end{align*}

\item This follows from
\begin{align*}
J_W\ni\partial_e\psi_W&=\sum_{h\in H}x^{H\setminus\set{h}}\cdot\partial_e\psi_{W\setminus H}+x^H\cdot\partial_e\psi_{W/H}\\
&\equiv x^{H\setminus\set{d}}\cdot\partial_e\psi_{W\setminus H}\mod\ideal{x_d}.
\end{align*}

\item By \eqref{46a}, the hypotheses force $x^{H\setminus\set{h}}\in\pp$ for all $h\in H$ and hence $\ideal{x_e,x_f}\subseteq\pp$ for some $e,f\in H$ with $e\ne f$.
Then $x^{H\setminus\set{e,f}}\in\pp$ by \eqref{46b} and the claim follows.\qedhere

\end{asparaenum}
\end{proof}


\begin{rmk}[Primes over the Jacobian ideal and $2$-separations]\label{141}
Let $W\subseteq\KK^E$ be a realization of a connected matroid $\M$.
Suppose that $E=E_1\sqcup E_2$ is an (exact) $2$-separation of $\M$.
For $\set{i,j}=\set{1,2}$, note that
\[
d_i:=\deg\psi_{W\vert_{E_i}}=\deg\psi_{W/E_j}+1
\]
and hence by Proposition~\ref{63}
\begin{align*}
J_W\ni\psi_W&=\psi_{W/E_i}\cdot\psi_{W\vert_{E_i}}+\psi_{W\vert_{E_j}}\cdot\psi_{W/E_j},\\
J_W\ni\sum_{e\in E_i}x_e\partial_e\psi_W&=d_i\cdot\psi_{W/E_i}\cdot\psi_{W\vert_{E_i}}+(d_i-1)\cdot\psi_{W\vert_{E_j}}\cdot\psi_{W/E_j}.
\end{align*}
Subtracting $d_i\cdot\psi_W$ from the latter yields $\psi_{W\vert_{E_j}}\cdot\psi_{W/E_j}\in J_W$, for $j=1,2$.
It follows that, for every prime ideal $\pp\in\Spec\KK[x]$ over $J_W$ and every $2$-separation $F$ of $\M$, we have $\psi_{W\vert_F}\in\pp$ or $\psi_{W/F}\in\pp$.
\end{rmk}


\begin{lem}[Inductive codimension bound]\label{99}
Let $W\subseteq\KK^E$ be a realization of a connected matroid $\M$, and let $H\in\H_\M$ be a proper non-disconnective handle.
Suppose that $\codim_{\KK^{E\setminus H}}\Sigma_{W\setminus H}=3$.
Then $\Sigma_W$ is equidimensional of codimension 
\[
\codim_{\KK^E}\Sigma_W=3
\]
with generic points of the following types:
\begin{enumerate}[(a)]
\item\label{99a} $\pp=\ideal{x_e,x_f,x_g}=:\pp_{e,f,g}$ for some $e,f,g\in H$ with $e\ne f\ne g\ne e$,
\item\label{99b} $\pp=\ideal{\psi_{W\setminus H},x_d,x_h}=:\pp_{H,d,h}$ for some $d,h\in H$ with $d\ne h$,
\item\label{99c} $\psi_{W\setminus H},\psi_{W/H}\in\pp\not\ni x_h$ for all $h\in H$.
\end{enumerate}
\end{lem}


\begin{proof}
Since $H$ is non-disconnective, $\psi_{W\setminus H}\in\KK[x_{E\setminus H}]$ is irreducible by Proposition~\ref{4}.
Since $d,h\in H$ with $d\ne h$, $\pp_{H,d,h}\in\Spec\KK[x]$ with $\height\pp_{H,d,h}=3$.
The same holds for $\pp_{e,f,g}$.

By Lemma~\ref{27} and the dimension hypothesis, $J_{W\setminus H}\unlhd\KK[x_{E\setminus H}]$ has height $3$.
Thus, for any $d\in H$, 
\begin{equation}\label{193}
\height(\ideal{J_{W\setminus H},x_d})=\height J_{W\setminus H}+1=4.
\end{equation}
In particular, $\Sigma_{W\setminus H}\ne\emptyset$ and hence $\Sigma_W\ne\emptyset$ by Remark~\ref{20}.\eqref{20a}.

Let $\pp\in\Spec\KK[x]$ be any minimal prime ideal over $J_W$.
By Lemma~\ref{27} and Proposition~\ref{44}, it suffices to show for the equidimensionality that $\height\pp\ge3$.
This follows in particular if $\pp$ contains a prime ideal of type $\pp_{e,f,g}$ or $\pp_{H,d,h}$.
By Lemma~\ref{46}.\eqref{46d}, the former is the case if $\psi_{W\setminus H}\not\in\pp$.
We may thus assume that $\psi_{W\setminus H}\in\pp$.
By Lemma~\ref{46}.\eqref{46c}, 
\begin{equation}\label{192}
x^{H\setminus\set{d}}\cdot\partial_e\psi_{W\setminus H}\in\pp+\ideal{x_d}.
\end{equation}
for any $d\in H$ and $e\in E\setminus H$.

First suppose that $x_d\in\pp$ for some $d\in H$.
If $x^{H\setminus\set{d}}\in\pp$, then $\pp$ contains a prime ideal of type $\pp_{H,d,h}$ for some $h\in H\setminus\set{d}$.
Otherwise, $\ideal{J_{W\setminus H},x_d}\subseteq\pp$ by \eqref{192} and hence $\height\pp\ge4$ by \eqref{193} (see Remark~\ref{98}).

Now suppose that $x_h\not\in\pp$ for all $h\in H$ and hence $\psi_{W/H}\in\pp$ by \eqref{18a} and \eqref{18c} in Corollary~\ref{18}.
Let $\qq\in\Spec\KK[x]$ be any minimal prime ideal over $\pp+\ideal{x_d}$. 
By \eqref{192}, $\qq$ contains one of the ideals
\begin{equation}\label{194}
\ideal{\psi_{W\setminus H},\psi_{W/H},x_d,x_h}=\pp_{H,d,h}+\ideal{\psi_{W/H}},\quad\ideal{J_{W\setminus H},x_d},
\end{equation}
for some $h\in H\setminus\set{d}$.
By Lemma~\ref{16}.\eqref{16d} and \eqref{16c} (see Remark~\ref{6}),
\begin{align*}
\deg\psi_{W/H}&=\rk(\M/H)=\rk\M-\abs{H}\\
&=\rk\M-\rk(H)=\rk(\M\setminus H)-\lambda_\M(H)<\deg\psi_{W\setminus H}
\end{align*}
and hence $\psi_{W\setminus H}\not\divides\psi_{W/H}$ and $\psi_{W/H}\not\in\pp_{H,d,h}$.
Thus, both ideals in \eqref{194} have height at least $4$ (see \eqref{192}) and hence $\height\qq\ge4$.
It follows that $\height(\pp+\ideal{x_d})\ge4$ and then $\height\pp\ge3$ by Lemma~\ref{109}.\eqref{109b}.
\end{proof}


\begin{rmk}\label{98}
The case where $\height\pp\ge4$ in the proof of Lemma~\ref{99} does finally not occur due to the Cohen--Macaulayness of $\Delta_W$ achieved by the argument (see Proposition~\ref{4}).
\end{rmk}


\begin{lem}[Generic points for circuits]\label{105}
Let $W\subseteq\KK^E$ be a realization of a matroid $\M$ on $E\in\C_\M$ with $\abs{E}-1=\rk\M\ge2$.
Then $\Sigma_W^\red$ is the union of all codimension-$3$ coordinate subspaces of $\KK^E$.
\end{lem}

\begin{proof}
We apply the strategy of the proof of Lemma~\ref{99}.
By Remark~\ref{20}.\eqref{20}, the rank hypothesis implies that $\Sigma_W\ne\emptyset$.
Let $\pp\in\Spec\KK[x]$ be any minimal prime ideal over $J_W$.
If $\psi_{W\setminus H}\not\in\pp$ for some $E\ne H\in\H_\M$, then Lemma~\ref{46}.\eqref{46d} yields $e,f,g\in H$ with $e\ne f\ne g\ne e$ such that $\ideal{x_e,x_f,x_g}\subseteq\pp$.
Otherwise, $\pp$ contains $x^{E\setminus H}=\psi_{W\setminus H}\in\pp$ for all $E\ne H\in\H_\M$ and hence all $x_e$ where $e\in E$.
(This can only occur if $\abs{E}=3$.)
By Lemma~\ref{27} and Proposition~\ref{44}, it follows that $\pp=\ideal{x_e,x_f,x_g}$.
By symmetry, all such triples $e,f,g\in E$ occur (see Example~\ref{122}).
\end{proof}


\begin{thm}[Cohen--Macaulayness of degeneracy schemes]\label{100}
Let $W\subseteq\KK^E$ be a realization of a connected matroid $\M$ of rank $\rk\M\ge2$.
Then $\Delta_W$ is Cohen--Macaulay (and hence pure-dimensional) and $\Sigma_W$ is equidimensional, both of codimension $3$ in $\KK^E$.
\end{thm}

\begin{proof}
By Proposition~\ref{44}, it suffices to show that $\codim_{\KK^E}\Sigma_W=3$.
Lemma~\ref{66} yields a circuit $C\in\C_\M$ of size $\abs{C}\ge3$ and $\codim_{\KK^C}\Sigma_{W\vert C}=3$ by Lemma~\ref{105}.
Proposition~\ref{12} yields a handle decomposition of $\M$ of length $k$ with $F_1=C$.
By Lemma~\ref{99} and induction on $k$, then also $\codim_{\KK^E}\Sigma_{W}=3$.
\end{proof}


\begin{cor}[Types of generic points]\label{104}
Let $W\subseteq\KK^E$ be a realization of a connected matroid $\M$ of rank $\rk\M\ge2$, and let $H\in\H_\M$ be a non-disconnective handle such that $\rk(\M\setminus H)\ge2$.
Then all generic points of $\Sigma_W$ and $\Delta_W$ are of the types listed in Lemma~\ref{99} with respect to $H$.
\end{cor}

\begin{proof}
Applying Theorem~\ref{100} to the matroid $\M\setminus H$ with realization $W\setminus H$, the claim follows from Lemma~\ref{99} and Theorem~\ref{13}.
\end{proof}


\begin{cor}[Generic points for $3$-connected matroids]\label{106}
Let $W\subseteq\KK^E$ be a realization of a $3$-connected matroid $\M$ with $\abs{E}>3$ if rank $\rk\M\ge2$.
Then all generic points of $\Sigma_W$ and $\Delta_W$ lie in $\TT^E$, that is,
\[
\Min\Sigma_W=\Min\Delta_W\subseteq\TT^E.
\]
\end{cor}

\begin{proof}
The equality is due to Theorem~\ref{13}.
We may assume that $\Sigma_W\ne\emptyset$ and hence $\rk\M\ge2$ by Remark~\ref{20}.\eqref{20a}.
Let $\pp\in\Min\Sigma_W$ be a generic point of $\Sigma_W$.
For any $e\in E$, consider the $1$-handle $H:=\set{e}\in\H_\M$.
By Proposition~\ref{53} and Lemma~\ref{16}.\eqref{16c}, $H$ is non-disconnective with $\rk(\M\setminus H)=\rk\M\ge2$.
Corollary~\ref{104} forces $\pp$ to be of type \eqref{99c} in Lemma~\ref{99}.
It follows that $\pp\in\bigcap_{e\in E}D(x_e)=\TT^E$.
\end{proof}

\subsection{Reducedness of degeneracy schemes}\label{82}

In this subsection, we prove the reducedness statement in our main result as outlined in \S\ref{191}.


\begin{lem}[Generic reducedness for the prism]\label{19}
Let $W\subseteq\KK^E$ be any realization of the prism matroid (see Definition~\ref{200}).
Then $\Delta_W\cap\TT^E$ is an integral scheme of codimension $3$, defined by $3$ linear binomials, each supported in a corresponding handle.
If $\ch\KK\ne2$, then also $\Sigma_W\cap\TT^E=\Delta_W\cap\TT^E$.
\end{lem}

\begin{proof}
By Remark~\ref{161}, we may assume that $W$ is the realization from Lemma~\ref{42}.
A corresponding matrix of $Q_W$ is given in Example~\ref{215}.
Reducing its entries modulo $\pp:=\ideal{x_1+x_2,x_3+x_4,x_5+x_6}$ makes all its $3\times 3$-minors $0$.
Therefore, $J_W\subseteq M_W\subseteq\pp$ by Lemma~\ref{50}.
Using the minors
\begin{align*}
Q_W^{2,3}&=(x_1+x_2)\cdot (-x_3x_5),\\
Q_W^{2,4}&=(x_1+x_2)\cdot (-x_3)\cdot (x_5+x_6),\\
Q_W^{3,4}&=(x_1+x_2)\cdot (x_3+x_4)\cdot x_5,\\
Q_W^{4,4}&=(x_1+x_2)\cdot (x_3+x_4)\cdot (x_5+x_6),
\end{align*}
one computes that
\[
-Q_W^{2,3}+Q_W^{2,4}-Q_W^{3,4}+Q_W^{4,4}=(x_1+x_2)\cdot x_4x_6.
\]
By symmetry, it follows that $x_2x_4x_6\cdot\pp\subseteq M_W$ and hence
\[
\Delta_W\cap D(x_2x_4x_6)=V(\pp)\cap D(x_2x_4x_6).
\]
Using $\psi_W$ from Example~\ref{128}, one computes that
\begin{gather*}
(x_2\cdot(x_2\partial_2-1)+x_4x_6\cdot(\partial_3+\partial_5)+(x_4+x_6)\cdot(1-x_4\partial_4-x_6\partial_6))\psi_W\\
=2\cdot(x_1+x_2)\cdot x_4^2x_6^2.
\end{gather*}
By symmetry, it follows that $2\cdot x_2^2x_4^2x_6^2\cdot\pp\subseteq J_W$ and hence
\[
\Sigma_W\cap D(x_2x_4x_6)=V(\pp)\cap D(x_2x_4x_6).
\]
if $\ch\KK\ne2$.
\end{proof}

More details on the prism matroid can be found in Example~\ref{45}.


\begin{lem}[Reduction and deletion of non-(co)loops]\label{31}
Let $e\in E$ be a non-(co)loop in a matroid $\M$.
For any $I\unlhd\KK[x]$ set 
\[
\bar I:=(I+\ideal{x_e})/\ideal{x_e}\unlhd\KK[x]/\ideal{x_e}=\KK[x_{E\setminus\set{e}}].
\]
Then $J_{W\setminus e}\subseteq\bar J_W$ and $M_{W\setminus e}=\bar M_W$ for any realization $W\subseteq\KK^E$ of $\M$.
\end{lem}

\begin{proof}
This follows from Proposition~\ref{3} and Lemma~\ref{40}.
\end{proof}


\begin{lem}[Generic reducedness and deletion of non-(co)loops]\label{88}
Let $W\subseteq\KK^E$ be a realization of a matroid $\M$, and let $e\in E$ be a non-(co)loop.
Then $\Sigma_{W\setminus e}=\emptyset$ implies $\Sigma_W=\emptyset$.
Suppose that $\Min\Sigma_W\subseteq D(x_e)$ and that $\Sigma_W$ and $\Sigma_{W\setminus e}$ are equidimensional of the same codimension.
If $\Sigma_{W\setminus e}$ is generically reduced, then $\Sigma_W$ is generically reduced.
In this case, each $\pp\in\Min\Sigma_W$ defines a non-empty subset $\gamma(\pp)\subseteq\Min\Sigma_{W\setminus e}$ such that
\begin{gather}
\label{168}V(\pp)\cap V(x_e)=\bigcup_{\qq\in\gamma(\pp)}V(\qq),\\
\label{169}\pp\ne\pp'\implies\gamma(\pp)\cap\gamma(\pp')=\emptyset.
\end{gather}
In particular, $\abs{\Min\Sigma_W}\le\abs{\Min\Sigma_{W\setminus e}}$.
The same statements hold for $\Sigma$ replaced by $\Delta$.
\end{lem}

\begin{proof}
The subscheme $\Sigma_W\cap V(x_e)\subseteq\KK^{E\setminus\set{e}}$ is defined by the ideal $\bar J_W$ (see Lemma~\ref{31}).
By Lemma~\ref{31} and since $J_W$ is graded,
\begin{align*}
\Sigma_{W\setminus e}=\emptyset
&\iff J_{W\setminus e}=\KK[x_{E\setminus\set{e}}]
\implies\bar J_W=\KK[x]/\ideal{x_e}\\
&\iff J_W+\ideal{x_e}=\KK[x]
\iff J_W=\KK[x]
\iff\Sigma_W=\emptyset
\end{align*}
which is the first claim.

Let $\pp\in\Min\Sigma_W$ be a generic point of $\Sigma_W$.
Considered as an element of $\Spec\KK[x]$ it is minimal over $J_W$.
Since $J_W$ and hence $\pp$ is graded, $\pp+\ideal{x_e}\ne\KK[x]$.
Let $\qq\in\Spec\KK[x]$ be minimal over $\pp+\ideal{x_e}$.
By Lemma~\ref{31},
\begin{equation}\label{170}
J_{W\setminus e}\subseteq\bar J_W\subseteq\bar\qq.
\end{equation}
Since $x_e\not\in\pp$ by hypothesis, Lemma~\ref{109} shows that
\begin{align*}
\height\qq&=\height\pp+1,\\
\nonumber\height\bar\qq&=\height\qq-\height\ideal{x_e}=\height\pp.
\end{align*}
By the dimension hypothesis, Lemma~\ref{27} and \eqref{170}, it follows that $\bar\qq$ is minimal over both $J_{W\setminus e}$ and $\bar J_W$.
The former means that $\bar\qq\in\Min\Sigma_{W\setminus e}$.
The set $\gamma(\pp)$ of all such $\bar\qq$ is non-empty and satisfies condition~\eqref{168}.

Denote by $t\in\KK[\Sigma_W]$ the image of $x_e$.
Then $\qq\not\in\Min\KK[\Sigma_W]$ by hypothesis and $\qq$ is minimal over $t$ since $\bar\qq$ is minimal over $\bar J_W$.
This makes $t$ is a parameter of the localization 
\[
R:=\KK[\Sigma_W]_\qq.
\]
The inclusion \eqref{170} gives rise to a surjection of local rings
\begin{equation}\label{171}
\KK[\Sigma_{W\setminus e}]_{\bar\qq}\onto\KK[\Sigma_W\cap V(x_e)]_{\bar\qq}=R/tR.
\end{equation}
Suppose now that $\Sigma_{W\setminus e}$ is generically reduced.
Then $\KK[\Sigma_{W\setminus e}]_{\bar\qq}$ is a field which makes \eqref{171} an isomorphism.
By Lemma~\ref{94}, $R$ is then an integral domain with unique minimal prime ideal $\pp_\qq$.
Thus, $\KK[\Sigma_W]_\pp=R_{\pp_\qq}$ is reduced and $\pp$ is uniquely determined by $\bar\qq$.
This uniqueness is condition~\eqref{169}.
The particular claim follows immediately.

The preceding arguments remain valid if $\Sigma$ and $J$ are replaced by $\Delta$ and $M$, respectively: 
Lemma \ref{31} applies in both cases.
\end{proof}


\begin{lem}[Initial forms and contraction of non-(co)loops]\label{32}
Let $W\subseteq\KK^E$ be a realization of a matroid $\M$.
Suppose $E=F\sqcup G$ is partitioned in such a way that $\M/G$ is
obtained from $\M$ by successively contracting non-(co)loops.
For any ideal $J\unlhd\KK[x]_{x^G}=\KK[x_F,x_G^{\pm1}]$, denote by $J^{\inf}$ the ideal generated by the lowest $x_F$-degree parts of the elements of $J$.
Then $J_{W/G}[x_G^{\pm1}]\subseteq(J_W^{\inf})_{x^G}$ and $M_{W/G}[x_G^{\pm1}]\subseteq(M_W^{\inf})_{x^G}$.
\end{lem}

\begin{proof}
We iterate Proposition~\ref{3} and Lemma~\ref{40}, respectively, to pass from $W$ to $W/G$ by successively contracting non-(co)loops $e\in G$.
This yields a basis of $W$ extending a basis $w^1,\dots,w^s$ of $W/G$ such that
\begin{align}\label{74}
\psi_W&=x^G\cdot\psi_{W/G}+p,\\
\partial_f\psi_W&=x^G\cdot\partial_f\psi_{W/G}+\partial_fp,\nonumber\\
Q_W^{i,j}&=x^G\cdot Q_{W/G}^{i,j}+q_{i,j},\nonumber
\end{align}
for all $f\in F$ and $i,j\in\set{1,\dots,s}$, where $p,q_{i,j}\in\KK[x]$ are polynomials with no term divisible by $x^G$.
Since $\psi_W$ and $Q_W^{i,j}$ are homogeneous linear combinations of square-free monomials (see Definition~\ref{48} and Lemma~\ref{40}), $x^G\cdot\psi_{W/G}$, $x^G\cdot\partial_f\psi_{W/G}$ and $x^G\cdot Q_{W/G}^{i,j}$ are the respective lowest $x_F$-degree parts in \eqref{74}.
The claimed inclusions follow.
\end{proof}


\begin{lem}[Generic reducedness and contraction of non-(co)loops]\label{87}
Let $W\subseteq\KK^E$ be a realization of a matroid $\M$.
Suppose $E=F\sqcup G$ is partitioned in such a way that $\M/G$ is obtained from $\M$ by successively contracting non-(co)loops.
Then $\Sigma_{W/G}=\emptyset$ implies $\Sigma_W\cap D(x^G)\cap V(x_F)=\emptyset$.
Suppose that $\Sigma_W\cap D(x^G)$ and $\Sigma_{W/G}$ are equidimensional of the same codimension.
If $\Sigma_{W/G}$ is generically reduced, then $\Sigma_W\cap D(x^G)$ is generically reduced along $V(x_F)$.
The same statements hold for $\Sigma$ replaced by $\Delta$.
\end{lem}

\begin{proof}
Consider the ideal
\begin{align*}
I:=\ideal{x_F}&\unlhd\KK[\Sigma_W\cap D(x^G)]=:R\\
&=\KK[\Sigma_W]_{x^G}=(\KK[x]/J_W)_{x^G}=\KK[x_F,x_G^{\pm1}]/(J_W)_{x^G},
\end{align*}
$R$ being equidimensional by hypothesis.
With notation from Lemma~\ref{32}
\begin{align*}
\bar R=\gr_I R&=\gr_I((\KK[x]/J_W)_{x^G})
\cong(\gr_{\ideal{x_F}}(\KK[x]/J_W))_{x^G}\\
&\cong(\KK[x]/J_W^{\inf})_{x^G}
=\KK[x_F,x_G^{\pm1}]/(J_W^{\inf})_{x^G}.
\end{align*}
Lemma~\ref{32} then yields the first claim:
\begin{align*}
\Sigma_{W/G}=\emptyset&
\iff J_{W/G}=\KK[x_F]
\iff J_{W/G}[x_G^{\pm1}]=\KK[x_F,x_G^{\pm1}]\\&
\implies(J_W^{\inf})_{x^G}=\KK[x_F,x_G^{\pm1}]
\iff\bar R=0\iff I=R\\&
\iff\Sigma_W\cap D(x^G)\cap V(x_F)=\emptyset.
\end{align*}
The latter equality makes the second claim vacuous.

We may thus assume that $I\ne R$.
Lemma~\ref{32} yields a surjection
\begin{align*}
\pi\colon\KK[\Sigma_{W/G}\times\TT^G]&=(\KK[x_F]/J_{W/G})[x_G^{\pm1}]\\
&=\KK[x_F,x_G^{\pm1}]/(J_{W/G}[x_G^{\pm1}])\onto\bar R.
\end{align*}
By Lemmas~\ref{117} and \ref{95} and the dimension hypothesis, source and target are equidimensional of the same dimension and hence $\pi^{-1}$ induces
\[
\Min\Spec\bar R\subseteq\Min(\Sigma_{W/G}\times\TT^G).
\]
Suppose now that $\Sigma_{W/G}$ and hence $\Sigma_{W/G}\times\TT^G$ is generically reduced.
For any $\pp\in\Min\Spec\bar R$, this makes $\KK[\Sigma_{W/G}\times\TT^G]_\pp$ a field and due to 
\[
\pi_\pp\colon\KK[\Sigma_{W/G}\times\TT^G]_\pp\onto\bar R_\pp
\]
also $\bar R_\pp$ is a field.
It follows that $\bar R$ is generically reduced.
By Lemma~\ref{95}, $R$ is then generically reduced along $V(I)$.
This means that $\Sigma_W\cap D(x^G)$ is generically reduced along $V(x_F)$.

The preceding arguments remain valid if $\Sigma$ and $J$ are replaced by $\Delta$ and $M$, respectively: 
Lemma \ref{32} applies in both cases.
\end{proof}


\begin{lem}[Generic reducedness for circuits]\label{97}
Let $W\subseteq\KK^E$ be a realization of a matroid $\M$ on $E\in\C_\M$ of rank $\rk\M=\abs{E}-1\ge2$.
Then $\Delta_W$ is generically reduced.
If $\ch\KK\ne2$, then also $\Sigma_W$ is generically reduced.
\end{lem}

\begin{proof}
We proceed by induction on $\abs{E}$.
The case $\abs{E}=3$ is covered by Example~\ref{107}; here we use $\ch\KK\ne2$.

Suppose now that $\abs{E}>3$.
Let $\pp\in\Min\Sigma_W$ be a generic point of $\Sigma_W$.
By Lemma~\ref{105}, $\pp=\ideal{x_e,x_f,x_g}$ for some $e,f,g\in H$ with $e\ne f\ne g\ne e$.
Pick $d\in E\setminus\set{e,f,g}$.
Then $E\setminus\set{d}\in\C_{\M/d}$ and hence $\Sigma_{W/d}$ is generically reduced by induction.
By Lemmas~\ref{117} and \ref{87}, $\Sigma_W\cap D(x_d)$ is then along $V(x_{E\setminus\set{d}})$.
By choice of $d$, $\ideal{x_{E\setminus\set{d}}}\in V(\pp)\cap D(x_d)$.
In other words, $\pp\in\Min(\Sigma_W\cap D(x_d))$ specializes to a point in $V(x_{E\setminus\set{d}})\cap D(x_d)$.
Thus, $\Sigma_W$ is reduced at $\pp$.
It follows that $\Sigma_W$ is generically reduced.

By Theorem~\ref{13}, $\Delta_W$ has the same generic points as $\Sigma_W$.
Therefore, the preceding arguments remain valid if $\Sigma$ is replaced by $\Delta$. 
\end{proof}


\begin{lem}[Generic reducedness and contraction of non-maximal handles]\label{86}
Let $W\subseteq\KK^E$ be a realization of a connected matroid $\M$ of rank $\rk\M\ge2$.
Assume that $\abs{\Max\H_\M}\ge2$ and set
\[
\hbar:=\abs{E}-\abs{\Max\H_\M}\ge0.
\]
Suppose that $\Sigma_{W'}$ is generically reduced for every realization $W'\subseteq\KK^{E'}$ of every connected matroid $\M'$ of rank $\rk\M'\ge2$ with $\abs{E'}<\abs{E}$.
\begin{enumerate}[(a)]

\item\label{86a} If $\hbar>3$, then $\Sigma_W$ is generically reduced.

\item\label{86b} If $\hbar>2$ and $e\in E$, then $\Sigma_W$ is reduced at all $\pp\in\Min\Sigma_W\cap V(x_e)$.

\end{enumerate}
The same statements hold for $\Sigma$ replaced by $\Delta$.
\end{lem}

\begin{proof}
Let $\pp\in\Spec\KK[x]$ with $\height\pp=3$.
Pick a subset $F\subseteq E$ such that $\abs{F\cap H'}=1$ for all $H'\in\Max\H_M$.
If possible, pick $F\cap H'=\set{e}$ such that $x_e\in\pp$.
If $\hbar>3$, then by Lemma~\ref{109}.\eqref{109b}
\begin{equation}\label{120}
\height(\pp+\ideal{x_F})
\le 3+\abs{F}
=3+\abs{\Max\H_\M}
<\abs{E}=\height\ideal{x}.
\end{equation}
If $\hbar>2$ and $\pp\in V(x_e)$, then \eqref{120} holds with $3$ replaced by $2$.
In either case pick $\qq\in\Spec\KK[x]$ such that
\begin{equation}\label{121}
\pp+\ideal{x_F}\subseteq\qq\subsetneq\ideal{x}.
\end{equation}
Add to $F$ all $f\in E$ with $x_f\in\qq$.
This does not affect \eqref{121}.
Then $x_g\not\in\qq$ and hence $x_g\not\in\pp$ for all $g\in G:=E\setminus F\ne\emptyset$.
In other words, 
\begin{equation}\label{123}
\pp\in D(x^G),\quad \qq\in V(\pp)\cap D(x^G)\cap V(x_F)\ne\emptyset.
\end{equation}
By the initial choice of $F$, $G\cap H'\subsetneq H'$ for each $H'\in\Max\H_\M$.
By Lemma~\ref{16}.\eqref{16b}, successively contracting all elements of $G$ does, up to bijection, not affect circuits and maximal handles.
In particular, $\M/G$ is a connected matroid on the set $F$, obtained from $\M$ by successively contracting non-(co)loops.

Since $\abs{F}\ge\abs{\Max\H_\M}\ge2$, connectedness implies that $\rk(\M/G)\ge1$.
If $\rk(\M/G)=1$, then $\Sigma_{W/G}=\emptyset$ by Remark~\ref{20}.\eqref{20a}.
Then $\Sigma_W\cap D(x^G)\cap V(x_F)=\emptyset$ by Lemma~\ref{87} and hence $\pp\not\in\Sigma_W$ by \eqref{123}.

Suppose now that $\pp\in\Sigma_W$ and hence $\rk(\M/G)\ge2$.
Then $\Sigma_{W/G}$ is generically reduced by hypothesis, and $\pp\in\Sigma_W\cap D(x^G)$ specializes to a point in $V(x_F)\cap D(x^G)$ by \eqref{123}.
By Theorem~\ref{100} and Lemma~\ref{117}, $\Sigma_W$, $\Sigma_W\cap D(x^G)$ and $\Sigma_{W/G}$ are equidimensional of codimension $3$.
By Lemma~\ref{27}, $\height\pp=3$ means that $\pp\in\Min\Sigma_W$.
By Lemma~\ref{87}, $\Sigma_W$ is thus reduced at $\pp$.
The claims follow.

The preceding arguments remain valid if $\Sigma$ is replaced by $\Delta$.
\end{proof}


\begin{lem}[Reducedness for connected matroids]\label{115}
Let $W\subseteq\KK^E$ be a realization of a connected matroid $\M$ of rank $\rk\M\ge2$.
Then $\Delta_W$ is reduced.
If $\ch\KK\ne2$, then $\Sigma_W$ is generically reduced.
\end{lem}

\begin{proof}
By Theorem~\ref{100}, $\Delta_W$ is pure-dimensional.
By Lemma~\ref{176}, $\Delta_W$ is thus reduced if it is generically reduced.
By Lemma \ref{50} and Theorem~\ref{13}, the first claim follows if $\Sigma_W$ is generically reduced.

Assume that $\ch\KK\ne2$.
We proceed by induction on $\abs{E}$.
By Lemma~\ref{97}, $\Sigma_W$ is generically reduced if $E\in\C_\M$; the base case where $\abs{E}=3$ needs $\ch\KK\ne2$.
Otherwise, by Proposition~\ref{12}, $\M$ has a handle decomposition of length $k\ge2$.
By Proposition~\ref{15}, $\M$ has $k+1$ (disjoint) non-disconnective handles $H=H_1,\dots,H_\ell\in\H_\M$ with
\begin{equation}\label{116}
\ell\ge k+1\ge3.
\end{equation}
Note that $H_1,\dots,H_\ell\in\Max\H_\M\cap\I_\M$ by Lemma~\ref{16}.\eqref{16a} and \eqref{16d}.
In particular, $\rk(\M\setminus H)\ne0$.

Suppose first that $H=\set{h}$.
Then $\rk(\M\setminus h)\ge2$ by Remark~\ref{20}.\eqref{20a} and Lemma~\ref{88}, and $\Min\Sigma_W\subseteq D(x_h)$ by Corollary~\ref{104}.
By Theorem~\ref{100}, both $\Sigma_W$ and $\Sigma_{W\setminus h}$ are equidimensional of codimension $3$.
Thus, $\Sigma_W$ is generically reduced by Lemma~\ref{88} and the induction hypothesis.

Suppose now that $\abs{H_i}\ge2$ for all $i=1,\dots,\ell$, and set (see Lemma~\ref{86})
\[
m:=\abs{\Max\H_\M},\quad\hbar:=\abs{E}-m.
\]
If $\hbar>3$, then $\Sigma_W$ is generically reduced by Lemma~\ref{86}.\eqref{86a} and the induction hypothesis.
Otherwise, 
\[
2\ell+(m-\ell)\le\sum_{i=1}^\ell\abs{H_i}+(m-\ell)\le\abs{E}=\hbar+m\le3+m
\]
and hence $2\ell\le\sum_{i=1}^\ell\abs{H_i}\le3+\ell$.
Comparing with \eqref{116} yields $\ell=3$, $k=2$ and $\abs{H_i}=2$ for $i=1,2,3$.
By Lemma~\ref{17}, $E=H_1\sqcup H_2\sqcup H_3$ is then the handle partition of $\M$.
In particular, $\hbar=6-3=3>2$.
By Lemma~\ref{42}, $\M$ must be the prism matroid.

Let now $\pp\in\Min\Sigma_W$ be a generic point of $\Sigma_W$, with $\M$ the prism matroid.
If $\pp\in\TT^E$, then $\Sigma_W$ is reduced at $\pp$ by Lemma~\ref{19}; here we use $\ch\KK\ne2$ again.
Otherwise, $\pp\in V(x_e)$ for some $e\in E$.
Then $\Sigma_W$ is reduced at $\pp$ by Lemma~\ref{86}.\eqref{86b} and the induction hypothesis.

The preceding arguments remain valid for arbitrary $\ch\KK$ if $\Sigma$ is replaced by $\Delta$.
\end{proof}


\begin{thm}[Reducedness]\label{101}
Let $W\subseteq\KK^E$ be a realization of a matroid $\M$.
Then 
\[
\Delta_W=\Sigma_W^\red
\]
is reduced.
If $\ch\KK\ne2$, then $\Sigma_W$ is generically reduced.
\end{thm}

\begin{proof}
By Theorem~\ref{112} and Lemma~\ref{115} (see Remarks~\ref{84} and \ref{20}.\eqref{20a}), $\Delta_W$ is reduced and $\Sigma_W$ is generically reduced if $\ch\KK\ne2$.
The claimed equality is then due to Theorem~\ref{13}.
\end{proof}

\subsection{Integrality of degeneracy schemes}\label{83}

In this subsection, we prove the following companion result to Proposition~\ref{4} as outlined in \S\ref{191}.


\begin{thm}[Integrality for $3$-connected matroids]\label{133}
Let $W\subseteq\KK^E$ be a realization of a $3$-connected matroid $\M$ of rank $\rk\M\ge2$.
Then $\Delta_W$ is integral and hence $\Sigma_W$ is irreducible.
\end{thm}

\begin{proof}
The claim on $\Delta_W$ follows from Remark~\ref{20}.\eqref{20a} and Lemmas~\ref{132} and \ref{137} and Corollary~\ref{138}.
Theorem~\ref{13} yields the claim on $\Sigma_W$.
\end{proof}


In the following, we use notation from Example~\ref{145}.


\begin{lem}[Reduction to wheels and whirls]\label{132}
It suffices to verify Theorem~\ref{133} for $\M\in\set{\W_n,\W^n}$ with $n\ge3$.
\end{lem}

\begin{proof}
Let $\M$ and $W$ be as in Theorem~\ref{133}.
By Remark~\ref{20}.\eqref{20b} and Theorem~\ref{13}, the claim holds if $\rk\M=2$.
If $\abs{E}\le4$, then $\M=\U_{2,n}$ where $n\in\set{3,4}$ (see \cite[Tab.~8.1]{Oxl11}) and hence $\rk\M=2$.
We may thus assume that $\rk\M\ge3$ and $\abs{E}\ge5$.

The $3$-connectedness hypothesis on $\M$ holds equivalently for $\M^\perp$ (see \ref{207}).
By Corollaries~\ref{71} and \ref{106}, the Cremona isomorphism thus identifies
\begin{equation}\label{179}
\TT^E\supseteq\Min\Delta_W=\Min\Delta_{W^\perp}\subseteq\TT^{E^\vee}.
\end{equation}
It follows that integrality is equivalent for $\Delta_W$ and $\Delta_{W^\perp}$.
In particular, we may also assume that $\rk\M^\perp\ge3$.

We proceed by induction on $\abs{E}$.
Suppose that $\M$ is not a wheel or a whirl.
Since $\rk\M\ge3$, Tutte's wheels-and-whirls theorem (see \cite[Thm.~8.8.4]{Oxl11}) yields an $e\in E$ such that $\M\setminus e$ or $\M/e$ is again $3$-connected.
In the latter case, we replace $W$ by $W^\perp$ and use \eqref{144}.
We may thus assume that $\M\setminus e$ is $3$-connected.
Then $\Delta_{W\setminus e}$ is integral by induction hypothesis.
Note that $\Min\Delta_W\subseteq D(x_e)$ by \eqref{179}.
By Theorem~\ref{100}, $\Delta_W$ and $\Delta_{W\setminus e}$ are equidimensional of codimension $3$.
By Remark~\ref{20}.\eqref{20a} and Lemma~\ref{88}, $\Delta_W\ne\emptyset$ and $\abs{\Min\Delta_W}\le\abs{\Min\Delta_{W\setminus e}}=1$.
It follows that $\Delta_W$ is integral.
\end{proof}


\begin{lem}[Turning wheels]\label{204}
Let $W\subseteq\KK^E$ be the realization of $\W_n$ from Lemma~\ref{134}.
Then the cyclic group $\ZZ_n$ acts on $X_W$, $\Sigma_W$ and $\Delta_W$ by \enquote{turning the wheel}, induced by the generator $1\in\ZZ_n$ mapping
\begin{equation}\label{182}
s_i\mapsto s_{i+1},\quad r_i\mapsto r_{i+1},\quad w^i\mapsto w^{i+1}.
\end{equation}
\end{lem}

\begin{proof}
By Lemma~\ref{134}, $W$ has a basis $w=(w_1,\dots,w_n)$ where $w^i=s_i+r_i-r_{i-1}$ for all $i\in\ZZ_n$.
The assignment \eqref{182} stabilizes $W\subseteq\KK^E$.
The resulting $\ZZ_n$-action stabilizes $\psi_W$ and $Q_W$, and hence $J_W$ and $M_W$.
As a consequence, it induces an action on $X_W$, $\Sigma_W$ and $\Delta_W$.
\end{proof}


The graph hypersurface of the $n$-wheel was described by Bloch, Esnault and Kreimer (see \cite[(11.5)]{BEK06}).
We show that it is also the unique configuration hypersurface of the $n$-whirl.

 
\begin{prp}[Schemes for wheels and whirls]\label{146}
Let $W\subseteq\KK^E$ be any realization of $\M\in\set{\W_n,\W^n}$ where $E=S\sqcup R$.
Then there are coordinates $z'_1,\dots,z'_n,y_1,\dots,y_n$ on $\KK^E$ such that 
\[
\psi_W=\det Q_n,\quad M_W=I_{n-1}(Q_n),
\]
where
\[
Q_n:=
\begin{pmatrix}
z'_1 & y_1 & 0 & \cdots & \cdots & 0 & y_n\\
y_1 & z'_2 & y_2 & 0 & \cdots & \cdots & 0 \\
0 & y_2 & z'_3 & y_3 & 0 & \cdots & 0\\
\vdots & \ddots & \ddots & \ddots & \ddots & \ddots & \vdots\\
0 & \cdots & 0 & y_{n-3} & z'_{n-2} & y_{n-2} & 0\\
0 & \cdots & \cdots & 0 & y_{n-2} & z'_{n-1} & y_{n-1} \\
y_n & 0 & \cdots & \cdots & 0 & y_{n-1} & z'_n
\end{pmatrix}.
\]
In particular, $X_W$, $\Sigma_W$ and $\Delta_W$ depend only on $n$ up to isomorphism.
\end{prp}

\begin{proof}
We may assume that $W$ is the realization from Lemma~\ref{134}.
Denote the coordinates on $\KK^E=\KK^{S\sqcup R}$ by
\begin{equation}\label{181}
z_1,\dots,z_n,y_1,\dots,y_n:=s_1^\vee,\dots,s_n^\vee,r_1^\vee,\dots,r_n^\vee,
\end{equation}
and consider the $\KK$-linear automorphism defined by
\[
z'_1:=z_1+y_1+t^2\cdot y_n,\quad z'_i:=z_i+y_i+y_{i-1},\quad i=2,\dots,n.
\]
Then $Q_W$ is represented by the matrix
\[
\begin{pmatrix}
z'_1 & -y_1 & 0 & \cdots & \cdots & 0 & -t\cdot y_n\\
-y_1 & z'_2 & -y_2 & 0 & \cdots & \cdots & 0 \\
0 & -y_2 & z'_3 & -y_3 & 0 & \cdots & 0\\
\vdots & \ddots & \ddots & \ddots & \ddots & \ddots & \vdots\\
0 & \cdots & 0 & -y_{n-3} & z'_{n-2} & -y_{n-2} & 0\\
0 & \cdots & \cdots & 0 & -y_{n-2} & z'_{n-1} & -y_{n-1} \\
-t\cdot y_n & 0 & \cdots & \cdots & 0 & -y_{n-1} & z'_n
\end{pmatrix}.
\]
Suitable scaling of $y_1,\dots,y_n$ turns this matrix into $Q_n$.
The particular claim follows with Lemma~\ref{36}.
\end{proof}


\begin{cor}[Small wheels and whirls]\label{138}
Theorem~\ref{133} holds for the matroids $\M=\W_3$ and $\M=\W^n$ for $n\le 4$.
\end{cor}

\begin{proof}
Let $W$ be any realization of $\M$.
By Theorem~\ref{101}, $\Delta_W$ is reduced and it suffices to check
irreducibility, replacing $\KK$ by its algebraic closure.
By Proposition~\ref{146}, we may assume that $\Delta_W=V(I_{k+1}(Q_n))$ for $k=n-2$.

Consider the morphism of algebraic varieties of matrices
\[
Y:=\KK^{n\times k}\to\set{A\in\KK^{n\times n}\mid A=A^t,\ \rk A\le k}=:Z,\quad B\mapsto BB^t.
\]
Let $y_{i,j}$ and $z_{i,j}$ be the coordinates on $Y$ and $Z$, respectively.
Then $\Delta_W$ identifies with $V(z_{1,3},z_{2,4})\subseteq Z$ for $n=4$ and with $Z$ itself for $n\le3$.
Both the preimage $Y$ of $Z$ and for $n=4$ the preimage
\[
V(y_{1,1}y_{1,3}+y_{1,2}y_{2,3},y_{2,1}y_{1,4}+y_{2,2}y_{2,4})
\]
of $V(z_{1,3},z_{2,4})$ are irreducible.
It thus suffices to show that $Y$ surjects onto $Z$, which holds for all $k\le n$.

Let $A\in Z$ and $I\subseteq\set{1,\dots,n}$ with $\abs{I}=\rk A=k$ and rows $i\in I$ of $A$ linearly independent.
Apply row operations $C$ to make the rows $i\not\in I$ of $CA$ zero.
Then $CAC^t$ is nonzero only in rows and columns $i\in I$.
Modifying $C$ to include further row operations turns $CAC^t$ into a diagonal matrix.
As $\KK$ is algebraically closed, $CAC^t=D^2$ where $D$ has exactly $k$ nonzero diagonal entries.
Then $A=BB^t$ where $B:=C^{-1}D$, considered as an element of $Y$ by dropping zero columns.
\end{proof}


\begin{lem}[Operations on wheels and whirls]\label{135}
Let $\M\in\set{\W_n,\W^n}$.
\begin{enumerate}[(a)]
\item\label{135a} The bijection
\[
E=S\sqcup R\to E^\vee,\quad s_i\mapsto r_i^\vee,\quad r_i\mapsto s_i^\vee, 
\]
identifies $\M=\M^\perp$.
\end{enumerate}
Suppose now that $n$ is not minimal for $\M$ to be defined, that is, $n>3$ if $\M=\W_n$ and $n>2$ if $\M=\W^n$.
\begin{enumerate}[(a)]\setcounter{enumi}{1}

\item\label{135b} The matroid $\M\setminus s_n$ is connected of rank $\rk(\M\setminus s_n)\ge2$.
Its handle partition consists of non-disconnective handles, the $2$-handle $\set{r_{n-1},r_n}$ and $1$-handles.

\item\label{135c} The matroid $\M/r_n$ is connected of rank $\rk(\M/r_n)\ge2$.
Its handle partition consists of non-disconnective $1$-handles.

\item\label{135d} We can identify $\W_n\setminus s_n/r_n=\W_{n-1}$ and $\W^n\setminus s_n/r_n=\W^{n-1}$.

\end{enumerate}
\end{lem}

\begin{proof}\
\begin{asparaenum}[(a)]

\item The self-duality claim is obvious (see \cite[Prop.~8.4.4]{Oxl11}).

\item This follows from the description of connectedness in terms of circuits (see \eqref{142} and Example~\ref{145}).

\item This follows from the description of connectedness in terms of circuits (see \eqref{143} and Example~\ref{145}).

\item The operation $\M\mapsto\M\setminus s_n/r_n$ deletes the triangle $\set{s_{n-1},r_{n-1},s_n}$ and maps the triangle $\set{s_{n},r_{n},s_1}$ to $\set{s_{n-1},r_{n-1},s_1}$ (see \eqref{142} and \eqref{143}).
By duality, it acts on triads in the same way (see \eqref{135a} and \eqref{144}).
Moreover, $R\in\C_{\M\setminus s_n/r_n}$ is equivalent to $R\in\C_\M$ and hence $\M=\W_n$ (see \eqref{142}, \eqref{143} and Example~\ref{145}).
The claim then follows using the characterization of wheels and whirl in terms of triangles and triads (see Example~\ref{145}).\qedhere

\end{asparaenum}
\end{proof}


\begin{lem}[Induction on wheels and whirls]\label{137}
Theorem~\ref{133} for $\M=\W_n$ and $\M=\W^n$ follows from the cases $n=3$ and $n\le4$, respectively.
\end{lem}

\begin{proof}
Suppose that $n$ is not minimal for $\M\in\set{\W_n,\W^n}$ to be defined.
Let $W'$ be any realization of $\M/r_n$.
Then $W'\setminus s_n$ is a realization of
\[
\M/r_n\setminus s_n=\M\setminus s_n/r_n=\M_{n-1}
\]
by Lemma~\ref{135}.\eqref{135d}.
By induction hypothesis and Corollary~\ref{106}, $\Delta_{W'\setminus s_n}$ is integral with generic point in $\TT^{E\setminus\set{s_n,r_n}}$.
By Lemma~\ref{135}.\eqref{135c} and Corollary~\ref{104}, $\Min\Delta_{W'}\subseteq\TT^{E\setminus\set{r_n}}\subseteq D(s_n)$.
By Lemma~\ref{135}.\eqref{135c} and Theorems~\ref{100}, $\Delta_{W'}$ and $\Delta_{W'\setminus s_n}$ are equidimensional of codimension $3$.
By Remark~\ref{20}.\eqref{20a} and Lemma~\ref{88}, $\Delta_{W'}$ is then integral.

Let $W$ be any realization of $\M$ and use the coordinates from \eqref{181}.
By Lemma~\ref{135}.\eqref{135b} and Corollary~\ref{104}, $\Delta_{W\setminus s_n}$ has at most one generic point $\qq'$ in $V(y_{n-1},y_n)$ while all the others lie in $\TT^{E\setminus\set{s_n}}$.
By Corollary~\ref{71}, the Cremona isomorphism identifies the latter with generic points of $\Delta_{(W\setminus s_n)^\perp}$ in $\TT^{E^\vee\setminus\set{s_n^\vee}}$.
Use \eqref{144} and Lemma~\ref{135}.\eqref{135a} to identify
\[
(\M\setminus s_n)^\perp=\M^\perp/s_n^\vee=\M/r_n,\quad E^\vee\setminus\set{s_n^\vee}=E\setminus\set{r_n},
\]
and consider $(W\setminus s_n)^\perp$ as a realization $W'$ of $\M/r_n$.
By the above, $\Delta_{W'}$ is integral with generic point in $\TT^{E\setminus\set{r_n}}$.
Thus, $\Delta_{W\setminus s_n}$ has a unique generic point $\qq$ in $\TT^{E\setminus\set{s_n}}$.
To summarize,
\begin{equation}\label{183}
\Min\Delta_{W\setminus s_n}=\set{\qq,\qq'},\quad
\qq\in\TT^{E\setminus\set{s_n}},\quad
\qq'\in V(y_{n-1},y_n).
\end{equation}

By Lemma~\ref{135}.\eqref{135b} and Theorems~\ref{100} and \ref{101}, $\Delta_W$ and $\Delta_{W\setminus s_n}$ are equidimensional of codimension $3$ and reduced.
It suffices to show that $\Delta_W$ is irreducible.
By way of contradiction, suppose that $\pp\ne\pp'$ for some $\pp,\pp'\in\Min\Delta_W$.
By Corollary~\ref{106}, $\Min\Delta_W\subseteq\TT^E\subseteq D(s_n)$.
By Lemma~\ref{88} and \eqref{183}, it follows that 
\[
\Delta_W=\set{\pp,\pp'}.
\]
By \eqref{168} in Lemma~\ref{88}, we may assume that $\sqrt{\bar\pp}=\qq$ and $\sqrt{\bar\pp'}=\qq'$ where $\bar I:=(I+\ideal{z_n})/\ideal{z_n}$.

Consider first the case where $\M=\W_n$ with $n\ge4$.
By Remark~\ref{161}, we may assume that $W$ is the realization from Lemma~\ref{134}.
By Lemma~\ref{204}, the cyclic group $\ZZ_n$ acts on $\set{\pp,\pp'}$ by \enquote{turning the wheel}.
If it acts identically, then $\sqrt{\pp'+\ideal{z_i}}\supseteq\ideal{y_{i-1},y_i}$ for all $i=1,\dots,n$ and hence
\[
\sqrt{\pp'+\ideal{z_1,\dots,z_n}}=\ideal{z_1,\dots,z_n,y_1,\dots,y_n}.
\]
Then $\height(\pp'+\ideal{z_1,\dots,z_n})=2n$ which implies $\height\pp'\ge n>3$ by Lemma~\ref{109}.\eqref{109b}, contradicting Theorem~\ref{100} (see Lemma~\ref{27}).
Otherwise, the generator $1\in\ZZ_n$ switches the assignment $\pp\mapsto\qq$ and $\pp\mapsto\qq'$ and $n=2m$ must be even.
Then $\sqrt{\pp+\ideal{z_{2i}}}\supseteq\ideal{y_{2i-1},y_{2i}}$ for all $i=1,\dots,m$ and hence
\[
\sqrt{\pp+\ideal{z_2,z_4,z_6,\dots,z_n}}\supseteq\ideal{z_2,z_4,z_6,\dots,z_n,y_1,\dots,y_n}.
\]
This leads to a contradiction as before.

Consider now the case where $\M=\W^n$ with $n\ge5$.
For $i=1,\dots,n$, denote by $\qq_i$ and $\qq'_i$ the generic points of $\Delta_{W\setminus s_i}$ as in \eqref{183}.
By the pigeonhole principle, one of $\pp$ and $\pp'$, say $\pp$, is assigned to $\qq'_i$ for $3$ spokes $s_i$.
In particular, $\pp$ is assigned to $\qq'_i$ and $\qq'_j$ for two non-adjacent spokes $s_i$ and $s_j$.
Then
\[
\sqrt{\pp+\ideal{z_i,z_j}}\supseteq\ideal{z_i,z_j,y_{i-1},y_i,y_{j-1},y_j}.
\]
This leads to a contradiction as before.
The claim follows.
\end{proof}


Theorem~\ref{133} proves the \enquote{only if} part of the following conjecture.

\begin{cnj}[Irreducibility and $3$-connectedness]
Let $\M$ be a matroid of rank $\rk\M\ge2$ on $E$.
Then $\M$ is $3$-connected if and only if, for some/any realization $W\subseteq\KK^E$ of $\M$, both $\Delta_W$ and $\Delta_{W^\perp}$ are integral.
\end{cnj}

\section{Examples}\label{90}

In this section, we illustrate our results with examples of prism, whirl and uniform matroids.


\begin{exa}[Prism matroid]\label{45}
Consider the prism matroid $\M$ (see Definition~\ref{200}) with its unique realization $W$ (see Lemma~\ref{42}).
Then
\[
\psi_W=x_1x_2(x_3+x_4)(x_5+x_6)+x_3x_4(x_1+x_2)(x_5+x_6)+x_5x_6(x_1+x_2)(x_3+x_4)
\]
by Example~\ref{128}.
By Lemma~\ref{19}, $\Delta_W$ has the unique generic point
\[
\ideal{x_1+x_2,x_3+x_4,x_5+x_6}
\]
in $\TT^6$.
By Corollary~\ref{104}, there can be at most $3$ more generic points symmetric to
\[
\ideal{x_1,x_2,\psi_{W\setminus\set{1,2}}}=\ideal{x_1,x_2,x_3x_4x_5+x_3x_4x_6+x_3x_5x_6+x_4x_5x_6}.
\]
Over $\KK=\FF_2$, their presence is confirmed by a computation in \texttt{Singular} (see \cite{DGPS18}).
It reveals a total of $7$ embedded points in $\Sigma_W$.
There is $\ideal{x_1,\dots,x_6}$, and $3$ symmetric to each of
\[
\ideal{x_3,x_4,x_5,x_6}\quad\text{and}\quad\ideal{x_1,x_2,x_3+x_4,x_5+x_6}.
\]
Moreover, $\Sigma_W$ is not reduced at any generic point.
Since the above associated primes are geometrically prime, the conclusions remain valid over any field $\KK$ with $\ch\KK=2$.

A \texttt{Singular} computation over $\QQ$ shows that $\Sigma_W$ has exactly the above associated points for any field $\KK$ with $\ch\KK=0$ or $\ch\KK\gg0$.
We expect that this holds in fact for $\ch\KK\ne2$.

To verify at least the presence of these associated points in $\Sigma_W$ for $\ch\KK\ne2$, we claim that
\begin{align*}
\ideal{x_1,x_2,\psi_{W\setminus\set{1,2}}}&=J_W\colon 2((x_3+x_4)x_5^2-(x_3+x_4)x_6^2),\\
\ideal{x_3,x_4,x_5,x_6}&=J_W\colon 2(x_1+x_2)^2x_4x_6,\\
\ideal{x_1,x_2,x_3+x_4,x_5+x_6}&=J_W\colon 2x_2(x_3+x_4)x_6^2,\\
\ideal{x_1,\dots,x_6}&=J_W\colon 2(x_1+x_2)(x_3+x_4)x_6.
\end{align*}
The colon ideals on the right hand side can be read off from a suitable Gr\"obner basis (see \cite[Lems.~1.8.3, 1.8.10 and 1.8.12]{GP08}).
Using \texttt{Singular} we compute such a Gr\"obner basis over $\ZZ$ which confirms our claim.
There are no odd prime numbers dividing its leading coefficients.
It is therefore a Gr\"obner basis over any field $\KK$ with $\ch\KK\ne2$ and the argument remains valid.
\end{exa}


\begin{exa}[Whirl matroid]\label{55}
Consider the whirl matroid $\M:=\W^3$ (see Example~\ref{145}).
It is realized by $6$ points in $\PP^2$ with the collinearities shown in Figure~\ref{102}.
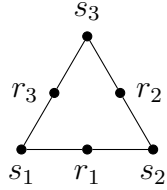
\begin{figure}[h]
\caption{Points in $\PP^2$ defining the whirl matroid $\W^3$.}\label{102}
\begin{tikzpicture}[scale=1,baseline=(current bounding box.center)]
\tikzstyle{root}=[circle,draw,inner sep=1.2pt,fill=black]
\node (1) at (210:1) [root,label=below:$s_1$] {};
\node (2) at (-30:1) [root,label=below:$s_2$] {};
\node (3) at (90:1) [root,label=above:$s_3$] {};
\node (4) at ($(1)!0.5!(2)$) [root,label=below:$r_1$] {};
\node (5) at ($(2)!0.5!(3)$) [root,label=right:$r_2$] {};
\node (6) at ($(1)!0.5!(3)$) [root,label=left:$r_3$] {};
\draw (1) -- (4) -- (2) -- (5) -- (3) -- (6) -- (1);
\end{tikzpicture}
\end{figure}
Since $\M$ contracts to the uniform matroid $\U_{2,4}$, $\M$ is not regular (see \cite[Thm.~6.6.6]{Oxl11}).
The configuration polynomial reflects this fact.
Using the realization $W$ of $\M$ from Lemma~\ref{134} with $t=-1$, $\set{s_1,s_2,s_3}=\set{1,2,3}$ and $\set{r_1,r_2,r_3}=\set{4,5,6}$, we find
\begin{align*}
\psi_W&=x_1x_2x_3+x_1x_3x_4+x_2x_3x_4+x_1x_2x_5+x_1x_3x_5+x_1x_4x_5\\
&+x_2x_4x_5+x_3x_4x_5+x_1x_2x_6+x_2x_3x_6+x_1x_4x_6+x_2x_4x_6\\
&+x_3x_4x_6+x_1x_5x_6+x_2x_5x_6+x_3x_5x_6+4x_4x_5x_6.
\end{align*}
Replacing in $\psi_W$ the coefficient $4$ of $x_4x_5x_6$ by a $1$ yields the matroid polynomial $\psi_\M$ (see Remark~\ref{149}).

By Theorem~\ref{100}, the configuration hypersurface $X_W$ defined by $\psi_W$ has $3$-codimensional non-smooth locus $\Sigma_W^\red$.
Using \texttt{Singular} (see \cite{DGPS18}) we compute a Gr\"obner basis over $\ZZ$ of the ideal of partial derivatives of $\psi_\M$.
The only prime numbers dividing leading coefficients are $2$, $3$ and $5$.
For $\ch\KK\ne2,3,5$, it is therefore a Gr\"obner basis over $\KK$.
From its leading exponents we calculate that the non-smooth locus of the hypersurface defined by $\psi_\M$ has codimension $4$ (see \cite[Cor.~5.3.14]{GP08}).
By further \texttt{Singular} computations, this codimension is $4$ for $\ch\KK=2,5$, and $3$ for $\ch\KK=3$.
\end{exa}


\begin{exa}[Uniform rank-$3$ matroid]\label{167}
Suppose that $\ch\KK\ne2,3$.
Then the configuration $W=\ideal{w^1,w^2,w^3}\subseteq\KK^3$ defined by
\[
(w^i_j)_{i,j}=
\begin{pmatrix}
1 & 0 & 0 & 1 & 2 & 3 \\
0 & 1 & 0 & 2 & 3 & 4 \\
0 & 0 & 1 & 2 & 6 & 12 
\end{pmatrix}
\]
realizes the uniform matroid $\U_{3,6}$ (see Example~\ref{119}).
The entries of $Q_w=(q_{i,j})_{i,j}$ satisfy the linear dependence relation (see Remark~\ref{151})
\[
q_{1,2}+q_{1,3}=q_{2,3}.
\]
By Lemma~\ref{36}, $\psi_W$ thus depends on fewer than $6$ variables.
More precisely, a \texttt{Singular} computation shows that $\Sigma_W$ has Betti numbers $(1,5,10,10,5,1)$, is not reduced and hence not Cohen--Macaulay.

Now, take $W'$ to be a generic realization of $\U_{3,6}$.
Then the entries of $Q_{W'}$ with indices $(i,j)$ where $i\le j$ are linearly independent (see \cite[Prop.~6.4]{BCK16}), and $\Sigma_{W'}$ is reduced Cohen--Macaulay with Betti numbers $(1,6,8,3)$.
So basic geometric properties of the configuration hypersurface $X_W$ are not determined by the matroid $\M$, but depend on the realization $W$.
\end{exa}


\begin{exa}[Uniform rank-$2$ matroid]\label{148}
Suppose that $\ch\KK\ne2$ and consider the uniform matroid $\U_{2,n}$ for $n\geq 3$ (see Examples~\ref{118} and \ref{122}.\eqref{122c}).
A realization $W$ of $\U_{2,n}$ is spanned by two vectors $w^1,w^2\in\KK^n$ for which (see Example~\ref{119})
\[
c_{W,\set{i,j}}=\det\begin{pmatrix}w^1_i & w^1_j\\w^2_i & w^2_j\end{pmatrix}^2\ne0,
\]
for $1\leq i<j\leq n$.
Then
\[
\psi_W=\sum_{1\leq i<j\leq n}c_{W,\set{i,j}}\cdot x_i\cdot x_j,
\]
and the ideal $J_W$ is generated by $n$ linear forms.
These forms may be written as the rows of the Hessian matrix
\[
H_W:=H_{\psi_W}=(c_{W,\set{i,j}})_{i,j},
\]
where by convention $c_{W,\set{i,i}}=0$.
Since uniform matroids are connected, Theorem~\ref{100} implies that $H_W$
has rank exactly $3$.

For $n\geq 4$, this amounts to a classical-looking linear algebra fact:
suppose that $A=(a_{i,j}^2)_{i,j}\in\KK^{n\times n}$ is a matrix with squared entries. 
Then its $4\times 4$ minors are zero provided that the numbers $a_{i,j}$ satisfy the Pl\"ucker relations defining the Grassmannian $\Gr_{2,n}$. 
An elementary direct proof was shown to us by Darij Grinberg (see \cite{Gri18}).
\end{exa}

\printbibliography
\end{document}